\documentclass[11pt]{amsart}
\usepackage{amsmath}
\usepackage{amsfonts}
\usepackage{latexsym}
\usepackage{amssymb}
\usepackage[dvips]{graphics}

\newcommand{\Z}{{\mathbb Z}}
\newcommand{\ZZ}{{\mathcal Z}}
\newcommand{\Q}{{\mathbb Q}}
\newcommand{\R}{{\mathbb R}}
\newcommand{\C}{{\mathbb C}}

\newcommand{\kk}{{\mathbf k}}
\newcommand{\cat}{{\rm {cat }}}
\newcommand{\Cat}{{\rm {Cat}}}
\newcommand{\id}{{\rm {id }}}

\newcommand{\Hom}{{\rm Hom}}
\newcommand{\im}{\rm im}

\newcommand{\ind}{{\rm {ind}}}
\newcommand{\Int}{{\rm {Int}}}
\newcommand{\Supp}{{\rm {Supp}}}
\newcommand{\cwgt}{{\rm {cwgt}}}

\newcommand{\RP}{\mathbf {RP}}

\newcommand{\clz}{{\rm {cl}}}
\DeclareMathOperator{\px}{{\widehat{\Z\pi_\xi}}}
\DeclareMathOperator{\sx}{{\Sigma_\xi^{-1}(\Z[\pi])}}

\newcommand{\rk}{{\rm rk}}
\newcommand{\comment}[1]{}

\newcommand{\Nov}{{\rm \mathbf {Nov}}}

\def\V{\mathcal V}
\def\A{\mathcal A}

\def\ker{{\rm Ker }}
\def\im{{\rm Im }}

\def\L{\mathcal L}

\newtheorem{theorem}{Theorem}
\newtheorem{proposition}{Proposition}
\newtheorem{lemma}[proposition]{Lemma}
\newtheorem{corollary}[proposition]{Corollary}

\theoremstyle{definition}
\newtheorem{definition}{Definition}
\newtheorem{example}{Example}
\newtheorem{remark}{Remark}

\begin{document}
\title{Closed 1-forms in topology and dynamics}

\author{Michael Farber and Dirk Sch\"utz}
\address{Department of Mathematics, University of Durham, Durham DH1 3LE, UK}
\email{Michael.Farber@durham.ac.uk}

\address{Department of Mathematics, University of Durham, Durham DH1 3LE, UK}
 \email{dirk.schuetz@durham.ac.uk}



\date{\today}


\begin{abstract}
This article surveys recent progress of results in topology and dynamics
based on techniques of closed one-forms. Our approach allows us to draw conclusions about properties of flows by studying homotopical and cohomological features of manifolds. More specifically we describe a Lusternik - Schnirelmann type theory
for closed one-forms, the focusing effect for flows and the theory of Lyapunov one-forms.
We also discuss recent results about cohomological treatment of the invariants $\cat(X, \xi)$ and $\cat^1(X, \xi)$ and their explicit computation in certain examples.
\end{abstract}

\maketitle

\begin{center}
{\it To S.P. Novikov on the occasion of his 70-th birthday}
\end{center}
\vskip 1cm

In 1981 S.P. Novikov \cite{N1, N2} initiated a generalization of Morse theory which gives topological estimates on the numbers of zeros of closed 1-forms, see also \cite{N3, N5}. Novikov was motivated by a variety of important problems of mathematical physics, leading, in one way or another, to the problem of finding relations between the topology of the underlying manifold and the number of zeros which closed 1-forms in a specific one-dimensional real cohomology class possess. Recall that a closed 1-form can be viewed as a multi-valued function or functional whose branching behavior is fully characterized by the cohomology class.

In \cite{N2} Novikov studies various problems of physics where the motion can be reduced to the principle of an extremal action $S$, which is a multi-valued functional on the space of curves, with the variation $\delta S$ being a well-defined closed 1-form, cf. \cite{N4, NSm, NT}. Among problems of this kind are the Kirchhoff equations for the motion of a rigid body in ideal fluid, the Leggett equation for the magnetic momentum, and others.

His fundamental idea in \cite{N2} was based on a plan to construct a chain complex, now called the Novikov complex, which uses dynamics of the gradient flow in the abelian covering associated with the cohomology class. The dynamics of gradient flows appears traditionally in Morse theory providing a bridge between the critical set of a function and the global ambient topology.

At present the Novikov theory of closed one-forms is a rapidly developing area of topology which interacts with various mathematical theories.
J.-Cl. Sikorav \cite{Si1, Sik} was the first who applied Novikov theory in symplectic topology. Hofer and Salamon \cite{HS} pioneered exploiting the ideas of Novikov theory in Floer theory. More recent applications of Novikov theory in symplectic topology and Hamiltonian dynamics can be found, for example, in the work of Oh \cite{O}, Usher \cite{U} and references mentioned there. Combinatorial group theory is another area where Novikov theory plays an important role; this connection was also discovered by J.-Cl. Sikorav who realized that the Bieri - Neumann - Strebel invariant can be expressed in terms of (a generalized noncommutative) Novikov homology.

Many problems of Novikov theory are still being actively developed in current research. The list of such topics includes (a) constructions of chain complexes (more general than the Novikov complex) which are able to capture the link between topology of the manifold and topology of the set of zeros, (b) various types of inequalities for closed 1-forms (with Morse or Bott type nondegeneracy assumptions), (c) equivariant inequalities and (d) problems about sharpness of these inequalities.

Two recent monographs \cite{farbook}, \cite{P5} give expositions of Novikov theory from quite different angles.

Topology of closed 1-forms is a broader research area which together with the Novikov theory studies a new Lusternik - Schnirelmann type theory for closed 1-forms initiated in 2002 in \cite{farber}.
The latter theory also aims at finding relations between topology of the zero set of a closed one-form and homotopy information, based mainly on the cohomology class of the form. The words \lq\lq Lusternik - Schnirelmann type\rq\rq\, intend to emphasize that no assumptions on the character of zeros are made, unlike the ones which appear in the Novikov theory requiring that all zeros are nondegenerate in the sense of Morse.

Although Lusternik - Schnirelmann theory for closed one-forms shares many common features with the Lusternik - Schnirelmann theory of functions and with the Novikov theory of closed one-forms, it is also very different from both these classical theories. The most striking new phenomenon is based on the fact that in any nonzero cohomology class there always exists a closed 1-form having at most one zero (see Theorem \ref{colliding} below). Hence, the new theory is not merely about the number of zeros but, as we show in this article, about qualitative dynamical properties of smooth flows on manifolds.

Given a smooth flow one wants to {\it \lq\lq tame\rq\rq}\  it by a closed 1-form lying in a prescribed cohomology class. More precisely, one wants to find a closed 1-form such that the flow is locally decreasing. This idea leads to the notion of a Lyapunov closed one-form generalizing the classical notion of a Lyapunov function. The key question about the existence of Lyapunov closed 1-forms for flows has been resolved in a series of papers of Farber, Kappeler, Latschev and Zehnder \cite{farbe4, FKLZ, FKLZ1, Lat2, Lat3}. Note that the classical theorem of C. Conley \cite{Cnl, Cnl1} gives the answer in the special case of the zero cohomology class, i.e. when one deals with Lyapunov functions. When the cohomology class is nontrivial the notion of asymptotic cycle of a flow introduced by Schwartzman \cite{Sc, Sc1} plays a crucial role.

 A brief account of Lusternik - Schnirelmann theory for closed one-forms can be found in chapter 10 of the book \cite{farbook} published in 2004. The main idea of this paper is to survey the new results obtained after 2004. To put the material in context we start this survey with a section describing the basic results of the Novikov theory.

  In a forthcoming paper \cite{FGS} we introduce and study the notion of sigma invariants for a finite CW-complex $X$ analogously to the sigma invariants of groups of Bieri et al \cite{binest, bieren}. These invariants carry information about the finiteness properties of infinite abelian covers of $X$.
  Both the present survey and \cite{FGS} have a common main theme: they are based on movability properties of subsets of $X$ with respect to a closed 1-form which are very close in spirit to the idea of Novikov homology.

 \vskip 0.5cm
\begin{center}
Table of contents
\end{center}
\vskip 0.5cm
1. Fundamentals of the Novikov theory

2. The colliding theorem

3. Closed 1-forms on general topological spaces

4. Lyapunov 1-forms for flows

5. Notions of category with respect to a cohomology class

6. Focussing effect

7. Existence of flow-convex neighbourhoods

8. Proof of Theorem \ref{jankothm1}

9. Topology of the chain recurrent set $R_\xi$

10. Proof of Theorem \ref{tkthm}

11. Cohomological estimates for $\cat(X, \xi)$

12. Upper bounds for $\cat(X, \xi)$ and relations with the Bieri - Neumann - Strebel invariants

13. Homological category weights, estimates for $\cat(X, \xi)$ and calculation of $\cat(X, \xi)$, $\cat^1(X, \xi)$ for products of surfaces

 \section{Fundamentals of the Novikov theory}

 S.P. Novikov \cite{N1}, \cite{N2} suggested a generalization of the classical Morse theory which gives lower bounds on the number of zeros of
 Morse closed 1-forms. In this section we will give a brief account of the Novikov theory. We touch here only the following selected topics: Novikov inequalities, Novikov numbers, and Novikov Principle leading to the notion of the Novikov complex. We refer the reader to the original papers \cite{N1, N2} and to the monograph \cite{farbook} for proofs and more details. Historical information about the development of the subject after \cite{N1,N2} as well as bibliographic references can also be found in \cite{farbook}.

\subsection{Novikov inequalities} Let $M$ be a smooth manifold. A smooth closed 1-form $\omega$ on $M$ is defined as a smooth section of the cotangent bundle $T^\ast M\to M$ satisfying $d\omega=0$. By the Poincar\'e lemma for any simply connected open set $U\subset M$ one has $\omega|U=df_U$ where $f_U:U\to \R$ is a smooth function, defined uniquely up to addition of a locally constant function. The zeros of $\omega$ are points $p\in M$ such that $\omega_p=0$. If $p$ lies in a simply connected domain $U$ then $\omega_p=0$ if and only if $p$ is a critical point of $f_U$.

 \begin{definition} A zero $p\in M$ of a smooth closed 1-form $\omega$ is said to be non-degenerate if it is a non-degenerate critical point of the function $f_U$. \end{definition}
 Clearly, this property is independent of the choice of the simply connected domain $U$ and the function $f_U$.
 \begin{definition}
 The Morse index of a non-degenerate zero $p$ of $\omega$ is defined as the Morse index of $p$ viewed as a critical point of $f_U$.
 \end{definition}
 We denote the Morse index by $\ind(p)$. It takes values $0, 1, 2, \dots, n$ where $n=\dim M$.

The main problem of the Novikov theory is as follows. Let $\omega$ be a smooth closed one-form on a closed smooth manifold $M$.
Let us additionally assume that $\omega$ is Morse, i.e. all its zeros are nondegenerate in the sense explained above. We denote by $c_i(\omega)$ the number of zeros of $\omega$ having Morse index $i$, where $i=0, 1, \dots, n$. One wants to estimate the numbers $c_i(\omega)$ in terms of information about the topology of $M$ and of the cohomology class \begin{eqnarray}\label{class}\xi=[\omega]\in H^1(M;\R)\end{eqnarray} represented by $\omega$.

In the case of classical Morse theory one has $\xi=[\omega]=0$ (i.e. $\omega=df$ where $f: M\to \R$ is a smooth function) and the answer is given by the Morse inequalities
\begin{eqnarray}\label{morse1}c_i(\omega)\ge b_i(M) +q_i (M)+q_{i-1}(M)\end{eqnarray}
where $b_i(M)$ denotes the $i$-th Betti number of $M$ and $q_i(M)$ denotes the minimal number of generators of the torsion subgroup of $H_i(M;\Z)$.

S.P. Novikov \cite{N1,N2} introduced generalizations of the numbers $b_i(M)$ and $q_i(M)$ which depend on the cohomology class $\xi$ of $\omega$  (see (\ref{class})) and are denoted $b_i(\xi)$ and $q_i(\xi)$ correspondingly. We call $b_i(\xi)$ the Novikov - Betti number; the number $q_i(\xi)$ is the Novikov torsion number. Their definitions will be given below. The Novikov inequality (in its simplest form) states:

\begin{theorem}\label{thmnov} Let $\omega$ be a smooth closed 1-form on a smooth closed manifold $M$. Assume that all zeros of $\omega$ are nondegenerate. Then
\begin{eqnarray}\label{novikov}
c_i(\omega) \ge b_i(\xi) + q_i(\xi) + q_{i-1}(\xi),
\end{eqnarray}
where $\xi=[\omega]\in H^1(M;\R)$ denotes the cohomology class represented by $\omega$.
\end{theorem}

Note that for $\xi=0$ the number $b_i(\xi)$ coincides with $b_i(M)$ and the number $q_i(\xi)$ coincides with $q_i(M)$, i.e. the Novikov inequalities (\ref{novikov}) contain the Morse inequalities (\ref{morse1}) as a special case.

Our goal in this section is to explain the main ideas of the proof of Theorem \ref{thmnov} without intending to present full details. A complete proof can be found in \cite{farbook}, pages 46 - 48.

Next we introduce the Novikov ring $\Nov_\kk$ where $\kk$ is a commutative ring.
Elements of $\Nov_\kk$ are formal \lq\lq power series\rq\rq\, of the form
\begin{eqnarray*}
x = \sum_{\gamma\in \R} n_\gamma t^{\gamma},\label{eks}
\end{eqnarray*}
where $t$ is a formal variable, the coefficients are elements of $\kk$, $n_\gamma\in \kk$, and the exponents $\gamma\in \R$
are arbitrary real numbers satisfying the following condition: for any $c\in \R$ the set
\begin{eqnarray}\label{set}
 \{\gamma\in \R; n_\gamma \ne 0, \, \gamma >c\}\label{finite}
\end{eqnarray} is finite.
Equivalently, an element $x\in \Nov_\kk$ can be represented in the form
\[x =n_0t^{\gamma_0} + n_1t^{\gamma_1} + n_2t^{\gamma_2} + \, \dots \]
where $n_i\in \kk$, $\gamma_i\in \R$, $\gamma_0> \gamma_1> \gamma_2 > \dots$ and $\gamma_i$ tends to $-\infty$. In other words, element of the Novikov ring $\Nov_\kk$ are \lq\lq Laurent like\rq\rq\, power series with integral coefficients and with arbitrary real exponents tending to $-\infty$.

Addition in $\Nov_\kk$ is given by adding the coefficients of powers of $t$
with equal exponents. Multiplication in $\Nov_\kk$ is given by the formula
\[\sum n_\gamma t^{\gamma}\cdot
\sum n'_{\gamma} t^{\gamma} \, =\,
\sum n_\gamma^{''}t^\gamma\]
where
\[ n_\gamma^{''}=
\sum_{{\gamma_1 +\gamma_{2} =\gamma}} n_{\gamma_1}
n'_{\gamma_2}.\]
Note that the last sum contains only finitely many nonzero terms and moreover, the set of all the
exponents $\gamma\in \R$ for which $n_\gamma^{''}$ is nonzero also satisfies the property that the sets (\ref{set}) are finite.

For simplicity, we will abbreviate the notation ${\Nov}_{\Z}$ to $\Nov$.

\begin{lemma}\label{lmpr}
The ring $\Nov$ is a principal ideal domain. Moreover, for any field $\kk$ the ring $\Nov_\kk$ is a field.
\end{lemma}

The second statement is in fact trivial; the proof of the first statement can be found in \cite{farbook}, page 8.
The credit for Lemma \ref{lmpr} should be mainly given to S.P. Novikov who discovered and stated it without proof. Full proofs appear in
  \cite{HS}, \cite{P2} and \cite{farbook}. Similar statements concerning related rings of rational functions were discovered in \cite{F1}, \cite{Si1}.

Now we can explain how one associates to a cohomology class $\xi\in H^1(M;\R)$ the Novikov numbers $b_i(\xi)$ and $q_i(\xi)$. Here we assume that $M$
is a smooth closed manifold although the construction is applicable to finite polyhedra.
Consider the group ring
$\Z[\pi]$ where $\pi$ denotes the fundamental group $\pi=\pi_1(X, x_0)$. The class $\xi$ determines a ring homomorphism
\begin{eqnarray}\label{rep}
\phi_\xi: \Z[\pi] \to \Nov
\end{eqnarray}
defined on a group element $g\in \pi$ by the formula
$$\phi_\xi(g) \, = \, t^{\langle \xi, g\rangle} \, \in \Nov.$$
Here $\langle \xi, g\rangle \in \R$ is the evaluation of the cohomology class $\xi$ on the homotopy class $g$. Recall that $t$ denote the formal indeterminant
of the Novikov ring. Clearly,
$$\phi_\xi(gg') = t^{\langle \xi, gg'\rangle} = t^{\langle \xi, g\rangle + \langle \xi, g'\rangle}= \phi_\xi(g)\phi_\xi(g')$$
i.e. $\phi_\xi$ is multiplicative on group elements. It follows that $\phi_\xi$ extends by linearity on the whole group ring as a ring homomorphism.
By the well-known construction homomorphism (\ref{rep}) defines a local system of left $\Nov$-modules over $M$ which we denote by $\L_\xi$.
The homology $H_i(M;\L_\xi)$ of this local system is a finitely generated module over the Novikov ring $\Nov$. As we know, $\Nov$ is a principal ideal domain, and therefore $H_i(M;\L_\xi)$ is a direct sum of a free and a torsion $\Nov$-module.
\begin{definition}\label{defnov}
The Novikov-Betti number $b_i(\xi)$ is defined as the rank of the free part of $H_i(X, \L_\xi)$. The Novikov torsion number $q_i(\xi)$ is defined as the minimal number of generators of the torsion submodule of $H_i(X;\L_\xi)$.
\end{definition}

Recall the definition of homology of local system $H_i(M; \L_\xi)$. Consider the universal cover $\tilde M \to M$ and the singular chain complex
$C_\ast(\tilde M)$.
The latter is a chain complex of free left modules over the group ring $\Z[\pi]$.
One views $\Nov$ as a  right $\Z[\pi]$-module where
$$x\cdot g =x \phi_\xi(g), \quad \mbox{where}\quad x\in \Nov, \quad g\in \pi.$$
Then $H_i(M;\L_\xi)$ is the homology of the chain complex
\begin{eqnarray}\label{novchain}\Nov\otimes_{\Z[\pi]} C_\ast(\tilde M).\end{eqnarray}

\subsection{Novikov and Universal complexes} The main idea of S.P. Novikov in \cite{N1, N2} which finally led him to Theorem \ref{thmnov} was the statement that {\it for any Morse closed 1-form $\omega$ on a smooth closed manifold $M$
there exists a chain complex $C^{\omega}$ (which is now known as {\it the Novikov complex}) having the following two properties:
\begin{enumerate}
\item $C^\omega$ is a free chain complex of $\Nov$-modules and each module $C^\omega_i$ has a canonical free basis which is in one-to-one correspondence with zeros of $\omega$ having Morse index $i$, where $i=0, 1, \dots, n=\dim M$.
    \item $C^\omega$ is chain homotopy equivalent to the complex $\Nov\otimes_{\Z[\pi]}C_\ast(\tilde M)$. In particular, the homology $H_i(C^\omega)$ is isomorphic to $H_i(M;\L_\xi)$.
\end{enumerate}}

This statement, which we call the {\it the Novikov Principle}, trivially implies Theorem \ref{thmnov}.
It should be compared with the classical {\it  Morse Principle} which claims that {\it for any Morse function on a closed smooth manifold there exists a chain complex
$C^f$ having the following two properties:
\begin{enumerate}
\item $C^f$ is a free chain complex of $\Z[\pi]$-modules and each module $C^f_i$ has a canonical free basis which is in one-to-one correspondence with critical points of $f$ having Morse index $i$, where $i=0, 1, \dots, n=\dim M$.
    \item $C^f$ is chain homotopy equivalent to the chain complex $C_\ast(\tilde M)$, where $\tilde M$ is the universal cover of $M$ and $\pi=\pi_1(M)$.
\end{enumerate}} There are several explicit constructions of Morse theory which lead to the complex $C^f$. One of them is based on the fact that $M$ admits a cell-decomposition with cells in one-to-one correspondence with critical points of $f$. Another well-known construction of $C^f$ is based on the Witten deformation of the de Rham complex.

One is naturally led to ask if there exist other ring homomorphisms \begin{eqnarray}\label{rep1}
\rho: \Z[\pi]\to \mathcal R,\end{eqnarray}
distinct from (\ref{rep}), for which the analogue of the Novikov Principle holds.
To be more specific, we say that {\it the Novikov Principle is valid for a group $\pi$, a group homomorphism\footnote{One has $\xi(gg') =\xi(g)+\xi(g')$ for $g,g'\in \pi$} $\xi: \pi\to \R$ and a ring homomorphism $\rho: \Z[\pi] \to \mathcal R$ if
for any smooth closed manifold $M$ with $\pi_1(M)=\pi$ and for any Morse closed 1-form $\omega$ on $M$ representing the cohomology class $\xi$ there exists a chain complex
$$C^{\omega}=(0\to C_n^\omega \to C^\omega_{n-1} \to \dots \to C_1^\omega\to C_0^\omega \to 0)$$ of free finitely generated $\mathcal R$-modules
having the following two properties:
\begin{enumerate}
\item each module $C^\omega_i$ has a canonical free basis which is in one-to-one correspondence with the zeros of $\omega$ having Morse index $i$;
    \item The complex $C^\omega$ is chain homotopy equivalent to $\mathcal R\otimes_{\Z[\pi]}C_\ast(\tilde M)$, where $\mathcal R$ is viewed as a left
    $\Z[\pi]$-module via $\rho$.
\end{enumerate}}

A positive answer to this question was given in \cite{FR} in the case of rational cohomology classes and in \cite{F11} in the general case.
Besides (\ref{rep}) the Novikov principle holds for many other ring homomorphisms (\ref{rep1}). Moreover, there exists \lq\lq the largest\rq\rq\, such homomorphism which we describe below. The chain complex $C^\omega$ in this case is called {\it the Universal complex}.
This complex lives over a localization of the group ring $\Z[\pi]$.

An element
$\alpha\in \Z[\pi]$ will be called {\it $\xi$-negative} if $\alpha=\sum n_j
g_j$ (finite sum), where $n_j\in \Z$ and $\xi(g_j)<0$ for all $j$. An $m\times
m$-matrix $A$ over the group ring $\Z[\pi]$ will be called {\it $\xi$-negative}
if all its entries are $\xi$-negative.
Consider the class $\Sigma_\xi$ of square matrices of the form
${\rm {Id}}+A$, where $A$ is an arbitrary $\xi$-negative square matrix with
entries in $\Z[\pi]$.

Note that in the case $\xi=0$ (when one studies Morse functions) the set of $\xi$-negative matrices
is empty and so the Cohn localization (\ref{la15}) is just the identity map.

We will use the notion of noncommutative localization developed by P.M. Cohn \cite{Co}.
{\it The universal Cohn localization} of the group ring
$\Z[\pi]$ with respect to the class $\Sigma_\xi$ is a ring $\sx$ together with
a ring homomorphism
\begin{eqnarray}\label{universal}\rho_\xi : \Z[\pi] \to \sx,\label{la15}
\end{eqnarray}
satisfying the following two properties: first, any matrix $\rho_\xi({\rm
{Id}}+A)$, where $A$ is $\xi$-negative, is invertible over $\sx$, and,
secondly, it is a {\it universal} homomorphism having this property, i.e., for
any ring homomorphism $\rho: \Z[\pi]\to \mathcal R$, inverting all matrices of
the form $I+A$, where $A$ is $\xi$-negative, there exists a unique ring
homomorphism $\phi: \sx \to \mathcal R$ such that the following diagram
commutes

\begin{figure}[h]
\setlength{\unitlength}{1cm}
  \begin{center}
\begin{picture}(9,2.8)
\linethickness{0.1mm}
\put(2.7,0.5){$\Z[\pi]$}
\put(4,2.){$\sx$}
\put(6,0.5){$\mathcal R$}
\put(3.7,0.6){\vector(1,0){1.9}}
\put(4.4,0.1){$\rho$}
\put(3.3,1.0){\vector(1,1){0.7}}
\put(3.3,1.5){$\rho_\xi$}
\put(5,1.7){\vector(1,-1){0.7}}
\put(5.5, 1.4){$\phi$}
\end{picture}
\end{center}
\end{figure}

We refer to the book of P.M. Cohn \cite{Co}, where the existence of
localization (\ref{la15}) is proven.

\begin{theorem}
\label{form2}
Let $\pi$ be any group and $\xi\in H^1(\pi;\R)$ be a cohomology class. Then the Novikov Principle is valid for the Cohn localization (\ref{universal}).
\end{theorem}

Farber and Ranicki in \cite{FR} proved this theorem in the case of integral cohomology classes; the general case was treated by Farber \cite{F10}.
The construction of the Universal complex $C^\omega$ employs the technique of collapse for chain complexes which is analogous to the well-known operation of combinatorial collapse. This technique was initiated in \cite{FR}.

One should mention another important closely related ring called {\it the Novikov-Sikorav completion} $\px$; it was first introduced by J.-Cl. Sikorav
\cite{Si1} who was inspired by the construction of the Novikov ring $\Nov$. One can view $\px$ as the noncommutative analogue of the Novikov ring.
Elements of $\px$ are represented by formal sums, possibly infinite,
$$\alpha = \sum n_i g_i$$ where $n_i\in \Z$ and $g_i\in \pi$, satisfying the
following condition: {\it for any real number $c\in \R$ the set $\{i; \,
\xi(g_i)\ge c\}$ is finite. } Compare this with the construction of the Novikov
ring above. Addition and multiplication are given by the usual
formulae; for example, the product of $\alpha=\sum n_i g_i\in \px$ and
$\beta=\sum m_j h_j\in \px$ is given by
\begin{eqnarray}
\alpha\cdot\beta = \sum_{i,j} (n_i m_j) (g_i h_j).
\end{eqnarray}
The ring homomorphism $\rho: \Z[\pi]\to \px$ is the inclusion. If $A$ is a $\xi$-negative
square matrix over the ring $\Z[\pi]$, then the power series $$({\rm {Id}}+A)^{-1} = {\rm
{Id}} - A + A^2 - \dots$$ converges in $\px$ and hence the matrix $I+A$ is invertible in
$\px$. From the universal property of the Cohn localization it follows that there exists a canonical ring homomorphism
\begin{eqnarray}
\sx\to \px,\label{la16}
\end{eqnarray}
extending the inclusion $\Z[\pi] \to \px$. This implies that the
homomorphism (\ref{la15}) is injective.

It is not known whether
(\ref{la16}) is always injective. Intuitively, the image of (\ref{la16})
consists of {\it \lq\lq rational power series\rq\rq}. This is a consequence of
the characterization of elements of the Cohn localization as components of
solutions of a linear system of equations; cf. \cite{Co}, Chapter 7.

\subsection{Novikov numbers and the fundamental group}
It is obvious that one-dimensional Novikov numbers $b_1(\xi)$ and $q_1(\xi)$
depend only on the fundamental group $\pi_1(X)$ and on the homomorphism of periods $\xi: \pi_1(X)\to \R$. Interestingly
there is also a strong inverse dependence, i.e. one may sometimes
recover information about the fundamental group while studying $b_1(\xi)$ and $q_1(\xi)$.
The following result shows that nontriviality of the first Novikov - Betti number $b_1(\xi)$ happens only if the fundamental group
$\pi_1(X)$ is \lq\lq large\rq\rq:

\begin{theorem}[Farber-Sch\"utz \cite{fsch}]\label{bone}
Let $X$ be a connected finite complex and let $\xi\in H^1(X;\R)$ be a
nonzero cohomology class. If the first Novikov-Betti number $b_1(\xi)$ is
positive then $\pi_1(X)$ contains a nonabelian free subgroup.
\end{theorem}

 Theorem \ref{bone} has the following Corollaries:

\begin{corollary}
Let $X$ be a connected finite polyhedron having an amenable fundamental group.
Then the first Novikov-Betti number vanishes $b_1(\xi)=0$ for any $\xi\not=0\in
H^1(X;\R)$.
\end{corollary}

\begin{corollary}\label{cor3}
Assume that $X$ is a connected finite two-dimensional polyhedron. If the Euler
characteristic of $X$ is negative $\chi(X)<0$ then $\pi_1(X)$ contains a
nonabelian free subgroup.
\end{corollary}

Corollary \ref{cor3} follows from Theorem \ref{bone} and from the Euler -Poincar\`e theorem
for the Novikov numbers
$$\sum_{i=0}^\infty (-1)^i b_i(\xi) \, =\,  \chi(X).$$

The proof of Theorem \ref{bone} given in \cite{fsch} is based on the
results of R. Bieri, W. Neumann, R. Strebel \cite{binest} and J.-Cl.
Sikorav \cite{Si1}.

We refer the reader to articles \cite{P3}, \cite{P4}, \cite{Ran} and \cite{Ran1} for additional information.

\section{The colliding theorem}

As follows from the discussion of the previous section, the Novikov theory gives bounds from below on the number of distinct
zeros which have closed Morse type 1-forms $\omega$ lying in a prescribed cohomology class
$\xi\in H^1(M;\R)$. The total number of zeros is then at least the sum $\sum_j b_j(\xi)$ of the
Novikov numbers $b_j(\xi)$.

If $\omega$ is a closed 1-form representing {\it the zero cohomology class}
then $\omega=df$ where $f: M\to \R$ is a smooth function; in this case $\omega$
must have at least $\cat(M)$ geometrically distinct zeros (which are exactly the critical points of the function $f$), according to the
classical Lusternik-Schnirelmann theory \cite{DNF}; this result requires no assumptions on the nature of the zeros.

Encouraged by the success of the Novikov theory one may ask if it is possible to construct an analogue of the Lusternik - Schnirelmann theory for closed 1-forms. It is quite surprising that in general, with the exception of two
situations mentioned above, {\it there are no obstructions for constructing
closed 1-forms possessing a single zero.} Hence, for $\xi\not=0$, the only homotopy invariant $a(M,\xi)$ (depending on the topology of $M$ and on the cohomology class $\xi$) such that any closed 1-form $\omega$ on $M$ with $[\omega]=\xi$ has at least $a(M,\xi)$ zeros is $a(M, \xi)=0$ or $a(M, \xi)=1$.

\begin{theorem} [Farber \cite{farber}, Farber - Sch\"utz \cite{farzeros}]\label{colliding}
Let $M$ be a closed connected $n$-dimensional smooth manifold, and let $\xi\in
H^1(M;\R)$ be a nonzero real cohomology class. Then there exists a smooth
closed 1-form $\omega$ in the class $\xi$ having at most one zero.
\end{theorem}

This result suggests that \lq\lq the Lusternik-Schnirelmann theory for closed
1-forms\rq\rq \,  (if it exists) must have a new
character, quite distinct from both the classical
Lusternik-Schnirelmann theory of functions and the Novikov theory of closed
1-forms.

Theorem \ref{colliding} was proven in \cite{farber} under an additional assumption that
the class $\xi$ is integral, $\xi\in H^1(M;\Z)$; see also \cite{farbook}, Theorem
10.1. In full generality it was proven in \cite{farzeros}.

Let us mention briefly a similar question. We know that if $\chi(M)=0$ then
there exists a nowhere zero 1-form $\omega$ on $M$. Given $\chi(M)=0$, one may
ask if it is possible to find a nowhere zero 1-form $\omega$ on $M$ which is
closed, i.e. $d\omega=0$? The answer is negative in general. For example,
vanishing of the Novikov numbers $b_j(\xi)=0$ is a necessary condition for the
class $\xi$ to be representable by a closed 1-form without zeros. The full list
of necessary and sufficient conditions (in the case $\dim M>5$) is given by the
theorem of Latour \cite{latour}.

Theorem \ref{colliding} was proven in \cite{farber} and \cite{farbook} in the case when the cohomology class $\xi$ is of rank 1.
In this section we follow our recent paper \cite{farzeros} and give the proof of Theorem \ref{colliding} in the case when $\rk(\xi)>1$.

\subsection{Singular foliations of closed one-forms}
Let $M$ be a smooth manifold.
A smooth closed 1-form $\omega$ with Morse zeros determines a {\it singular foliation}
$\omega=0$ on $M$. It is a decomposition of $M$ into leaves: two points $p,q
\in M$ belong to the same leaf if there exists a path $\gamma:[0,1]\to M$ with
$\gamma(0)=p$, $\gamma(1)=q$ and $\omega(\dot \gamma(t)) = 0$ for all $t$.
Locally, in a simply connected domain $U\subset M$, we have $\omega|_U =df$,
where $f:U\to \R$; each connected component of the level set $f^{-1}(c)$ lies
in a single leaf. If $U$ is small enough and does not contain the zeros of
$\omega$, one may find coordinates $x_1, \dots, x_n$ in $U$ such that $f\equiv
x_1$; hence the leaves in $U$ are the sets $x_1=c$. Near such points the
singular foliation $\omega=0$ is a usual foliation. On the contrary, if $U$ is
a small neighborhood of a zero $p\in M$ of $\omega$ having Morse index $0\le
k\le n$, then there are coordinates $x_1, \dots,x_n$ in $U$ such that
$x_i(p)=0$ and the leaves of $\omega=0$ in $U$ are the level sets $-x_1^2-\dots
-x_k^2 + x_{k+1}^2+\dots +x_n^2 =c.$ The leaf $L$ with $c=0$ contains the zero
$p$. It has a {\it singularity} at $p$: a neighborhood of $p$ in $L$ is
homeomorphic to a cone over the product $S^{k-1}\times S^{n-k-1}$. There are
finitely many {\it singular leaves}, i.e. the leaves containing the zeros of
$\omega$.

We are particularly interested in the singular leaves containing the zeros of
$\omega$ having Morse indices 1 and $n-1$. Removing such a zero $p$ {\it
locally} disconnects the leaf $L$. However globally the complement $L-p$ may or
may not be connected.

The singular foliation $\omega=0$ is {\it co-oriented}: the normal bundle to
any leaf at any nonsingular point has a specified orientation.

We shall use the notion of a weakly complete closed 1-form introduced by G.
Levitt \cite{Levitt}. A closed 1-form $\omega$ is called {\it weakly complete}
if it has Morse type zeros and for any smooth path $\sigma:[0,1]\to M^\ast$
with $\int_{\sigma}\omega=0$ the endpoints $\sigma(0)$ and $\sigma(1)$ lie in
the same leaf of the foliation $\omega=0$ on $M^\ast$. Here $M^\ast$ denotes
$M-\{p_1, \dots, p_m\}$ where $p_j$ are the zeros of $\omega$.

A weakly complete closed 1-form with $\xi=[\omega]\not=0$ has no zeros with
Morse indices $0$ and $n$. According to Levitt \cite{Levitt}, {\it any nonzero
real cohomology class $\xi\in H^1(M;\R)$ can be represented by a weakly
complete closed 1-form.}

The plan of our proof of Theorem \ref{colliding} is as follows. We start with a
weakly complete closed 1-form $\omega$ lying in the prescribed cohomology class
$\xi\in H^1(M;\R)$, $\xi\not=0$. We show that assuming $\rk(\xi)>1$ all leaves
of the singular foliation $\omega=0$ are dense (see \S \ref{density}). We
perturb $\omega$ such that the resulting closed 1-form $\omega'$ has a single
singular leaf (see \S \ref{modification}). After that we apply the technique of
Takens \cite{Ta} allowing us to collide the zeros in a single (highly
degenerate) zero. We first prove Theorem \ref{colliding} assuming that $n=\dim M
>2$; the special case $n=2$ is treated separately later.

\subsection{Density of the leaves}\label{density}

In this section we show that {\it if $\omega$ is weakly complete and
$\rk(\xi)>1$ then the leaves of $\omega=0$ are dense; moreover, given a point
$x\in M$ and a leaf $L\subset M$ of the singular foliation $\omega=0$, there
exist two sequences of points $x_k\in L$ and $y_k\in L$ such that $x_k\to x$,
$y_k\to x, $ and
\begin{eqnarray}\label{last}
\int_x^{x_k} \omega\, >0,\quad \mbox{while}\quad \int_x^{y_k}\omega \, <0.
\end{eqnarray}
 The integrals in (\ref{last}) are calculated along an arbitrary path lying in a
small neighborhood of $x$.} This can also be expressed by saying that the leaf
$L$ approaches $x$ from both the positive and the negative sides. This
statement was also observed by G. Levitt \cite{Levitt}, p. 645; we give below a
different proof. In general the assumptions that $\omega$ has no centers and
$\rk(\xi)>1$ do not imply that the leaves of the foliation $\omega=0$ are
dense, see the examples in Chapter 11, \S 9.3 of \cite{farbook}.

Let $\omega$ be a weakly complete closed 1-form in class $\xi$. Consider the
covering $\pi:\tilde M\to M$ corresponding to the kernel of the homomorphism of
periods ${\rm {Per}}_\xi : \pi_1(M) \to \R$, where $\xi=[\omega]\in H^1(M;\R)$.
Let $H\subset \R$ be the group of periods. The rank of $H$ equals $\rk(\xi)$;
since we assume that $\rk(\xi)>1$, the group $H$ is dense in $\R$. The group of
periods $H$ acts on the covering space $\tilde M$ as the group of covering
transformations. We have $\pi^\ast\omega =dF$ where $F:\tilde M\to \R$ is a
smooth function. The leaves of the singular foliation $\omega=0$ are the images
under the projection $\pi$ of the level sets $F^{-1}(c)$; this property follows
from the weak completeness of $\omega$, see \cite{Levitt}, Proposition II.1.
For any $g\in H$ and $x\in \tilde M$ one has
\begin{eqnarray}
F(g x)-F(x)\, =\, g\, \in\, \R.
\end{eqnarray}

Let $L=\pi(F^{-1}(c))$ be a leaf and let $x\in M$ be an arbitrary point. Our
goal is to show that $x$ lies in the closure $\bar L$ of $L$. Let $U\subset M$
be a path-connected neighborhood of $x$ which is assumed to be \lq\lq
small\rq\rq\, in the following sense: $\xi|_U =0$. We want to show that $U\cap
L$ contains a point $y$ satisfying $\int_x^y\omega >0$ where the integral is
calculated along a curve in $U$.

Consider a lift $\tilde x\in \tilde M$, $\pi(\tilde x)=x$. Let $\tilde U$ be a
neighborhood of $\tilde x$ which is mapped by $\pi$ homeomorphically onto $U$.
We claim that {\it the set of values $F(\tilde U)\subset \R$ contains an
interval $(a-\epsilon, a+\epsilon)$ where $a=F(\tilde x)$ and $\epsilon>0$.}
\begin{figure}[h]
\begin{center}
\resizebox{7cm}{4.2cm}{\includegraphics[48,473][497,768]{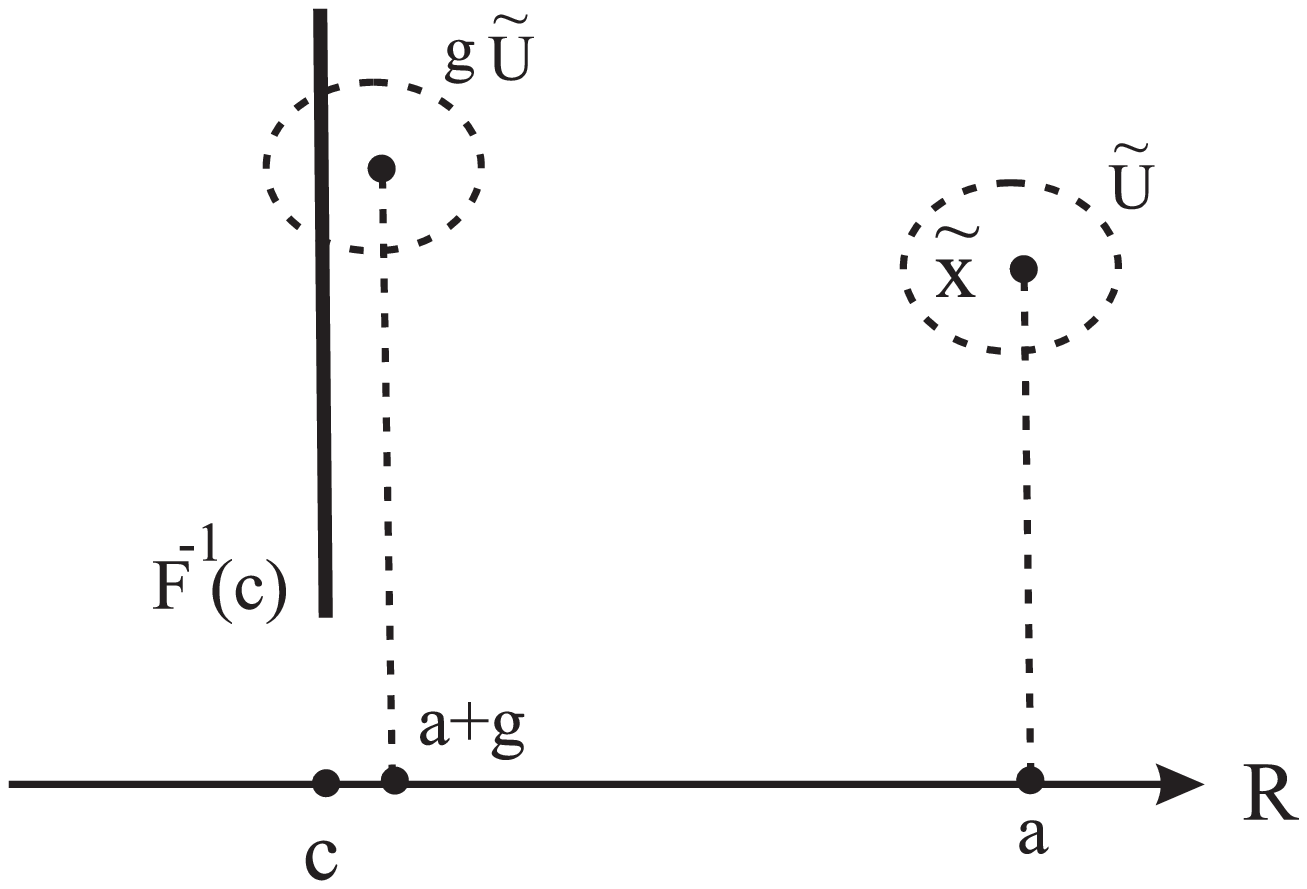}}
\end{center}
\end{figure}
This claim is obvious if $\tilde x$ is not a critical point of $F$ since in
this case one may choose the coordinates $x_1, \dots, x_n$ around $\tilde x$
such that $F(x)=a+x_1$. In the case when $\tilde x$ is a critical point of $F$,
one may choose the coordinates $x_1, \dots, x_n$ near the point $\tilde x\in
\tilde M$ such that $F(x)$ is given by $a\pm x_1^2\pm x_2^2+\dots+\pm x_n^2$
and our claim follows since we know that the Morse index is distinct from $0$
and $n$.

Because of the density of the group of translations $H\subset \R$ one may find
$g\in H$ such that the real number $F(g\tilde x)=F(\tilde x)+g=a+g$ lies in the
interval $(c,c+\epsilon)$. Then we obtain
\begin{eqnarray}
c\, \in \, (a+g-\epsilon,\,  a+g)\, \subset\, g+F(\tilde U)\, =\, F(g\tilde U).
\end{eqnarray}
We see that the sets $F^{-1}(c)$ and $g\tilde U$ have a common point $\tilde
y$. The point $y=\pi(\tilde y)\in U\cap L$ has the required property
$\int_x^y\omega >0$.

\subsection{Modification}\label{modification}

Our next goal is to replace $\omega$ by a Morse closed 1-form $\omega'$ which
has the property that all its zeros lie on the same singular leaf of the
singular foliation $\omega'=0$. In this section we assume that $n=\dim M>2$.

Let $\omega$ be a weakly complete Morse closed 1-form in class $\xi$ where
$\rk(\xi)>1$. Let $p_1, \dots, p_m\in M$ be the zeros of $\omega$. For each
$p_j$ choose a small neighborhood $ U_j\ni p_j$ and local coordinates $x_1,
\dots, x_n$ in $U_j$ such that $x_i(p_j)=0$ for $i=1, \dots, n$ and
\begin{eqnarray} \label{morse}\quad\omega|_{U_j}
=df_j, \quad \mbox{where}\quad f_j= -x_1^2-\dots -
x_{m_j}^2+x_{m_j+1}^2+\dots+x_n^2.
\end{eqnarray}
Here $m_j$ denotes the Morse index of $p_j$. We assume that the ball
$\sum_{i=1}^n x_i^2\leq 1$ is contained in $U_j$ and that $U_j\cap
U_{j'}=\emptyset$ for $j\not=j'$. Denote by $W_j$ the open ball $\sum_{i=1}^n
x_i^2< 1$.

Let $\phi:[0,1]\to [0,1]$ be a smooth function with the following properties:
(a) $\phi\equiv 0$ on $[3/4, 1]$; (b) $\phi\equiv \epsilon>0$ on $[0,1/2]$; (c)
$-1<\phi'\leq 0$.
\begin{figure}[h]
\begin{center}
\resizebox{5.9cm}{2.7cm}{\includegraphics[92,456][454,662]{phi.eps}}
\end{center}
\end{figure}
Such a function exists assuming that $\epsilon>0$ is small enough. (a), (b),
(c) imply that \begin{eqnarray}\label{stam} \phi'(r)>-2r, \quad\mbox{for}\quad
r>0.\end{eqnarray}

We replace the closed 1-form $\omega$ by
\begin{eqnarray}\label{form}
\omega'=\omega - \sum_{j=1}^m \mu_j \cdot dg_j\end{eqnarray} where $g_j: M\to
\R$ is a smooth function with support in $U_j$. In the coordinates $x_1, \dots,
x_n$ of $U_j$ (see above) the function $g_j$ is given by $g_j(x) =
\phi(||x||).$ The parameters $\mu_j\in [-1,1]$ appearing in (\ref{form}) are
specified later.

One has $\omega\equiv \omega'$ on $M-\cup_j U_j$ and near the zeros of
$\omega$. Let us show that $\omega'$ has no additional zeros. We have
$\omega'|_{U_j}=d(f_j-\mu_j g_j)$ (where $f_j$ is defined in (\ref{morse})) and
\begin{eqnarray}
\frac{\partial}{\partial x_i}(f_j-\mu_jg_j) = \pm 2x_i -\mu_j
\phi'(||x||)\frac{x_i}{||x||}
\end{eqnarray}
If this partial derivative vanishes and $x_i\not=0$ then $\phi'(r) =\pm
2r\mu_j^{-1}$ which may happen only for $r=||x||=0$ according to (\ref{stam}).

We now show how to choose the parameters $\mu_j$ so that the closed 1-form
$\omega'$ given by (\ref{form}) has a unique singular leaf. Let $L$ be a fixed
nonsingular leaf of $\omega=0$. Since $L$ is dense in $M$ (see \S
\ref{density}) for any $j=1, \dots, m$ the intersection $L\cap U_j$ contains
infinitely many connected components approaching $p_j$ from below and from
above and the function $f_j$ is constant on each of them.
\begin{figure}[h]
\begin{center}
\resizebox{4.0cm}{3.7cm}{\includegraphics[104,347][489,697]{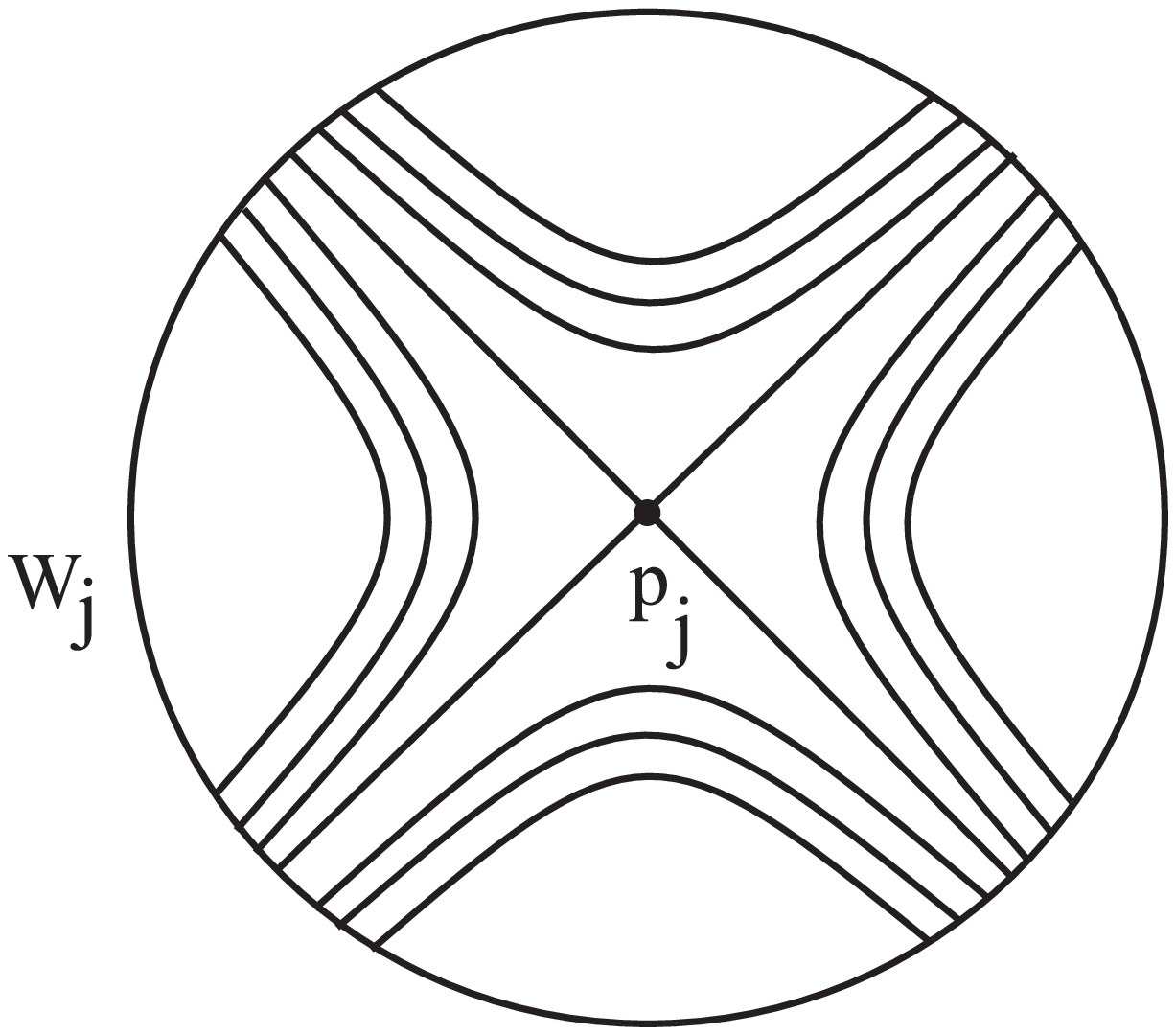}}
\end{center}
\end{figure}
We say that a subset $T_c\subset L\cap W_j$ is a \em level set\em\, if
$T_c=f_j^{-1}(c)\cap W_j$ for some $c\in \R$. Note that $f_j(p_j)=0$. The level
set $c=0$ contains the zero $p_j$; it is homeomorphic to the cone over the
product $S^{m_j-1}\times S^{n-m_j-1}$. Each level set $T_c$ with $c<0$ is
diffeomorphic to $S^{m_j-1}\times D^{n-m_j}$ and each level set $T_c$ with
$c>0$ is diffeomorphic to $D^{m_j}\times S^{n-m_j-1}$. Recall that $m_j$
denotes the Morse index of $p_j$.

Let $\V_j=f_j(L\cap W_j)\, \subset \R$ denote the set of values of $f_j$ on
different level sets belonging to the leaf $L$. The zero $0$ does not lie in
$\V_j$ since we assume that the leaf $L$ is nonsingular. However, according to
the result proven in \S \ref{density}, the zero $0\in \R$ is a limit point of
$\V_j$ and, moreover, the closure of either of the sets $\V_j\cap (0,\infty)$
and $\V_j\cap (-\infty, 0)$ contains $0\in \R$.

For the modification $\omega'$ (given by (\ref{form})) one has
$\omega'|_{U_j}=dh_j$ where $h_j=f_j-\mu_jg_j$. The level sets $T'_c$ for $h_j$
are defined as $h_j^{-1}(c)\cap W_j$. Clearly $T'_c$ is given by the equation
$$f_j(x)=\mu_j\phi(||x||) +c, \quad x\in W_j.$$
Hence for $||x||\geq 3/4$ this is the same as $T_c$; for $||x||\leq 1/2$ the
level set $T'_c$ coincides with $T_{c+\mu_j\epsilon}$. In the ring $1/2\leq
||x||\leq 3/4$ the level set $T'_c$ is homeomorphic to a cylinder.

The following figure illustrates the distinction between the level sets $T_c$
and $T'_c$ in the case $\mu_j>0$.
\begin{figure}[h]
\begin{center}
\resizebox{5.8cm}{5.5cm}{\includegraphics[94,326][461,698]{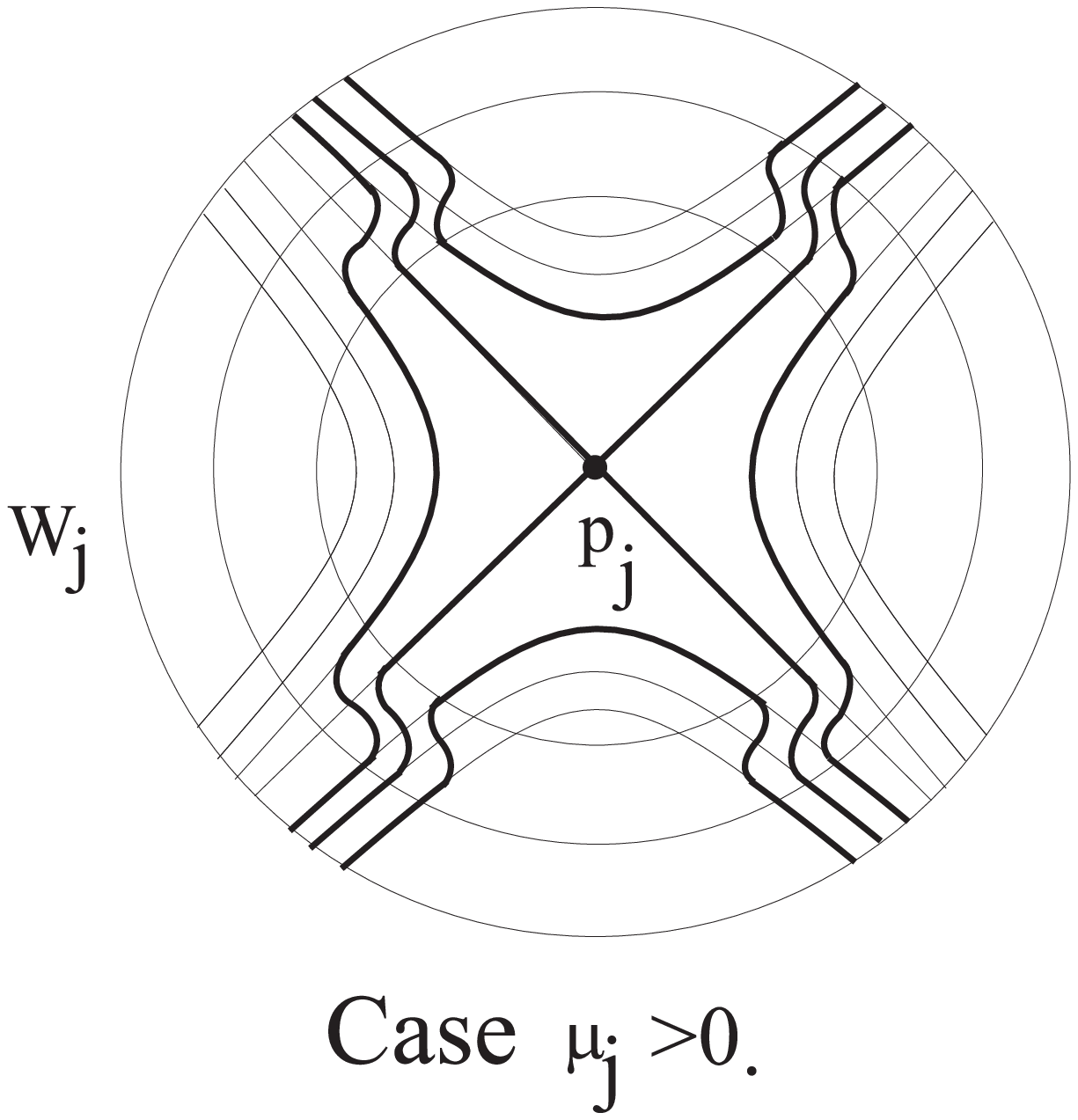}}
\end{center}
\end{figure}

Examine the changes which the leaf $L$ undergoes when we replace $\omega$ by
$\omega'$. Here we view $L$ with the {\it leaf topology}; it is the topology
induced on $L$ from the covering $\tilde M$ using an arbitrary lift $L\to
\tilde M$. First, let us assume that: (1) the Morse index $m_j$ satisfies
$m_j<n-1$; (2) the coefficient $\mu_j>0$ is positive; (3) the number
$-\epsilon\mu_j$ lies in the set $\V_j$. Then the complement
$$L-\underset{-\epsilon\mu_j<c<0}{\bigcup_{c\in \V_j}} T_c$$
is connected and it lies in a single leaf $L'$ of the singular foliation
$\omega'=0$. We see that the new leaf $L'$ is obtained from $L$ by infinitely
many surgeries. Namely, each level set $T_c\subset L$, where $c\in \V_j$
satisfies $-\epsilon\mu_j<c<0$, is removed and replaced by a copy of
$D^{m_j}\times S^{n-m_j-1}$; besides, the set $T_c\subset L$ where
$c=-\epsilon\mu_j$, is removed and gets replaced by a cone over the product
$S^{m_j-1}\times S^{n-m_j-1}$. Hence the new leaf $L'$ contains the zero $p_j$.

Let us now show how one may modify the above construction in the case
$m_j=n-1$. Since $n>2$ we have in this case $n-m_j-1<n-2$; hence removing the
sphere $S^{n-m_j-1}$ from the leaf $L$ does not disconnect $L$. We shall assume
that the coefficient $\mu_j$ is {\it negative} and that the number
$-\epsilon\mu_j$ lies in $\V_j\subset \R$. The complement
$$L-\underset{0<c<-\epsilon\mu_j}{\bigcup_{c\in \V_j}} T_c$$
is connected and it lies in a single leaf $L'$ of the singular foliation
$\omega'=0$. Clearly, $L'$ is obtained from $L$ by removing the level sets
$T_c$ where $c\in \V_j$ satisfies $0<c<-\epsilon \mu_j$ (each such $T_c$ is
diffeomorphic to $D^{m_j}\times S^{n-m_j-1}$) and by replacing them by copies
of $S^{m_j-1}\times D^{n-m_j}$. In addition, the set $T_c\subset L$ where
$c=-\epsilon\mu_j$, is removed and is replaces by a cone over the product
$S^{m_j-1}\times S^{n-m_j-1}$.

We see that $L'$ is a leaf of the singular foliation $\omega'=0$ containing all
the zeros $p_1, \dots, p_m$.

\subsection{Proof of Theorem \ref{colliding}}\label{proof}

Below we assume that $\rk(\xi)>1$. The case $\rk(\xi)=1$ is covered by Theorem
2.1 from \cite{farber}.

The results of the preceding sections allow us to complete the proof of Theorem
\ref{colliding} in the case $n=\dim M>2$. Indeed, we showed in \S \ref{modification}
how to construct a Morse closed 1-form $\omega'$ lying in the prescribed
cohomology class $\xi$ such that all zeros of $\omega'$ are Morse and belong to
the same singular leaf $L'$ of the singular foliation $\omega'=0$. Now we may
apply the colliding technique of F. Takens \cite{Ta}, pages 203--206. Namely,
we may find a piecewise smooth tree $\Gamma\subset L'$ containing all the zeros
of $\omega'$. Let $U\subset M$ be a small neighborhood of $\Gamma$ which is
diffeomorphic to $\R^n$. We may find a continuous map $\Psi: M\to M$ with the
following properties:

{\it $\Psi(\Gamma)$ is a single point $p\in \Gamma$;

$\Psi|_{M-\Gamma}$ is a diffeomorphism onto $M-p$;

$\Psi(U)=U;$

$\Psi$ is the identity map on the complement of a small neighborhood $V\subset
M$ of $\Gamma$ where the closure $\overline V$ is contained in $U$.}

Consider a smooth function $f: U\to \R$ such that $df=\omega'|_U$; it exists
and is unique up to a constant. The function $g=f\circ \Psi^{-1}:U\to \R$ is
well-defined (since $f|_\Gamma$ is constant). $g$ is continuous by the
universal property of the quotient topology. Moreover, $g$ is smooth on $M-p$.
Applying Theorem 2.7 from \cite{Ta}, we see that we can replace $g$ by a smooth
function $h:U\to \R$ having a single critical point at $p$ and such that $h=f$
on $U-\overline V$.

Let $\omega^{''}$ be a closed 1-form on $M$ given by
\begin{eqnarray}
\omega^{''}|_{M-\overline V}=\omega'|_{M-\overline V}\quad \mbox{and}\quad
\omega^{''}|_U=dh.
\end{eqnarray}
Clearly $\omega^{''}$ is a smooth closed 1-form on $M$ having no zeros in
$M-\{p\}$. Moreover, $\omega^{''}$ lies in the cohomology class $\xi=[\omega']$
(since any loop in $M$ is homologous to a loop in $M-\overline{V}$).

Now we prove Theorem \ref{colliding} the case $n=2$. We shall replace the
construction of \S \ref{modification} (which requires $n>2$) by a direct
construction. The final argument using the Takens' technique \cite{Ta} remains
the same.

Let $M$ be a closed surface and let $\xi\in H^1(M;\R)$ be a nonzero cohomology
class. We can split $M$ into a connected sum
$$M=M_1\sharp M_2\sharp \dots \sharp M_k$$
where each $M_j$ is a torus or a Klein bottle and such that the cohomology
class $\xi_j=\xi|_{M_j}\in H^1(M_j;\R)$ is nonzero. Let $\omega_j$ be a closed
1-form on $M_j$ lying in the class $\xi_j$ and having no zeros; obviously such
a form exists. \S 9.3.2 of \cite{farbook} describes the construction of connected
sum of closed 1-forms on surfaces. Each connecting tube contributes two zeros.
In fact there are three different ways of forming the connected sum, they are
denoted by A, B, C on Figure 9.8 in \cite{farbook}. In the type C connected sum the
zeros lie on the same singular leaf. Hence by using the type C connected sum
operation we get a closed 1-form $\omega$ on $M$ having $2k-2$ zeros which all
lie on the same singular leaf of the singular foliation $\omega=0$. The
colliding argument based on the technique of Takens \cite{Ta} applies as in the
case $n>2$ and produces a closed 1-form with at most one zero lying in class
$\xi$.

 \section{Closed one-forms on general topological spaces}

Since closed one-forms are central for our constructions, it will be convenient to operate with the notion of a closed one-form defined on general topological spaces. This notion will allow us to deal with spaces more general than smooth manifolds. The calculus of closed one-forms on topological spaces is very similar to that of smooth closed one-forms on manifolds: one may integrate any closed one-form along a path; the integral depends only on the homotopy class of the path; any closed one-form represents a one-dimensional cohomology class and any continuous function $f$ determines an exact closed one-form $df$, the differential of $f$.

In this section we recall the basic definitions referring to the book \cite{farbook} for proofs and more details.
\subsection{Basic definitions}
\begin{definition}
{\it A continuous closed 1-form $\omega$ on a topological space} $X$ is defined as a collection $\{f_U\}_{U\in \mathcal U}$ of continuous
real-valued functions $f_U: U\to \R$, where $\mathcal U=\{U\}$ is an open cover of $X$,
such that for any pair $U, V\in \mathcal U$ the difference
\[f_U|_{U\cap V} - f_V|_{U\cap V}: U\cap V \to \R\]
is a locally constant function.
Another such collection $\{g_V\}_{V\in \mathcal V}$
(where $\mathcal V$ is another open cover of $X$) defines {\it an equivalent} closed 1-form
if the union collection
$$\{f_U, g_V\}_{U\in \mathcal U, V\in \mathcal V}$$
is a closed 1-form (in the sense of the above definition), i.e., if for any
$U\in \mathcal U$ and $V\in \mathcal V$, the function $f_U-g_V$ is locally
constant on $U\cap V$. \end{definition}

Let $\phi: Y\to X$ be a continuous map and let $\omega$ be a continuous closed
1-form on $X$. Then the induced closed 1-form $\phi^\ast \omega$ is defined as
follows. Let $\omega=\{f_U\}_{U\in \mathcal U}$, where $\mathcal U$ is an open
cover of $X$. The family $\phi^{-1}\mathcal U=\{\phi^{-1}(U)\}_{U\in \mathcal
U}$ is an open cover of $Y$ and the functions $f_U\circ \phi$ define a
continuous closed 1-form with respect to the cover $\phi^{-1}\mathcal U$.

As an example consider an open cover $\mathcal U=\{X\}$ consisting of the whole space $X$. Then any continuous
function $f: X\to \R$ defines a closed 1-form on $X$, which is denoted by $df$.
For two continuous functions $f, g: X\to \R$, $df =dg$ holds if and only if the
difference $f-g: X\to \R$ is locally constant.

\subsection{Integration} One may integrate continuous closed 1-forms along continuous paths.
Let $\omega$ be a continuous closed 1-form  on $X$ given
by a collection of continuous real-valued functions $\{f_U\}_{U\in \mathcal U}$
with respect to an open cover $\mathcal U$ of $X$. Let $\gamma: [0,1]\to X$ be a continuous path. The line integral
$\int_\gamma \omega$ is defined as follows. Find a subdivision $t_0=0<t_1<\dots < t_N=1$ of the interval
$[0,1]$ such that for any $i$ the image $\gamma[t_i,t_{i+1}]$ is contained in a single open set $U_i\in \mathcal U$.
 Then we define
\begin{eqnarray}
\int\limits_\gamma \omega \, =\,
\sum\limits_{i=0}^{N-1} \, \, [f_{U_i}(\gamma(t_{i+1})) - f_{U_i}(\gamma(t_i))].\label{lsintegral}
\end{eqnarray}
The standard argument shows that the integral (\ref{lsintegral}) does not depend on the choice of the subdivision and the
open cover $\mathcal U$.

For any pair of continuous paths $\gamma, \gamma': [0,1]\to X$ with common
beginning $\gamma(0)=\gamma'(0)$ and common end points $\gamma(1)=\gamma'(1)$,
it holds that
$$\int_\gamma \omega \, =\,  \int_{\gamma'}\omega,$$
provided that $\gamma$ and $\gamma'$ are homotopic relative to the boundary.

\subsection{Cohomology class of a closed one-form}\label{cohclass}
Any continuous closed 1-form $\omega$ defines the {\it homomorphism of periods}
\begin{eqnarray}
\pi_1(X, x_0) \to \R, \quad [\gamma]\mapsto \int_\gamma \omega \in \R\label{lsperiod}
\end{eqnarray}
given by integration along closed loops $\gamma: [0,1]\to X$ with
$\gamma(0)=x_0=\gamma(1)$. The image of this homomorphism is a subgroup of $\R$ whose {\it rank} is called the rank of $\omega$ and is denoted
$\rk(\omega)$.

Recall that a topological space $X$ is {\it homologically locally
$n$-connected}\index{Homologically $n$-connected} if for every point $x\in X$
and any neighborhood $U$ of $x$ there exists a neighborhood $V$ of $x$ in $U$
such that the induced homomorphism of the reduced integral singular homology
$\tilde H_q(V)\to \tilde H_q(U)$ is trivial for all $q\leq n$. $X$ is {\it
locally path connected}\index{Locally path connected} iff it is
 homologically locally 0-connected.

\begin{lemma}\label{lslmlocal} Let $X$ be a locally path-connected topological space.
A continuous closed 1-form $\omega$ on $X$ equals $df$ for a continuous function $f: X\to
\R$ if and only if for any choice of the base point $x_0\in X$ the homomorphism of
periods (\ref{lsperiod}) determined by $\omega$ is trivial.
\end{lemma}
\begin{proof} If $\omega =df$,
then $\int_\gamma \omega = f(q)-f(p)$ holds for any path $\gamma$ in $X$, where
$q=\gamma(1), \, \, p=\gamma(0)$. Hence $\int_\gamma \omega =0$ if $\gamma$ is a closed loop.

Conversely, assume that the homomorphism of periods (\ref{lsperiod}) is
trivial. Our assumption about $X$ implies that connected components of $X$ are
open and path connected. In each connected component of $X$ choose a base point
$x_i\in X$. One defines a continuous function $f:X\to \R$ by
$$f(x) \, =\, \int_{x_i}^x \omega;$$
here $x$ and $x_i$ belong the same connected component of $X$ and the
integration is taken along an arbitrary path connecting $x_i$ to $x$. The
result of integration is independent of the choice of the path because of our
assumption that the homomorphism of periods (\ref{lsperiod}) is trivial. Assume
that $\omega$ is given by a collection of continuous functions $f_U: U\to \R$
with respect to an open cover $\{U\}$ of $X$. Then for any two points $x, y$
lying in the same connected component of $U$,
\[f(y) -f(x) =\int_x^y \omega = f_U(y) -f_U(x);\]
here the line integral is understood along any continuous path connecting $x$ and $y$.
This shows that the function $f-f_U$ is locally constant on $U$. Hence
$df =\omega$. \end{proof}

Any continuous closed 1-form $\omega$ on a topological space $X$ defines a
singular {\it cohomology class} $[\omega]\in H^1(X;\R)$. It is defined by the homomorphism of
periods (\ref{lsperiod}) viewed as an element of $\Hom(H_1(X);\R)=H^1(X;\R)$.
As follows from the above lemma, {\it two continuous closed 1-forms $\omega$
and $\omega'$ on $X$ lie in the same cohomology class $[\omega]=[\omega']$ if
and only if their difference $\omega - \omega'$ equals $df$, where $f:X\to \R$
is a continuous function.} Here we assume that $X$ is locally path-connected.

\begin{lemma}\label{lslmreal} Let $X$ be a paracompact Hausdorff homologically locally 1-con\-nected topological space.
Then any singular cohomology class $\xi\in H^1(X;\R)$
may be realized by a continuous closed 1-form on $X$.
\end{lemma}

\section{Lyapunov 1-forms for flows}
\label{lslyap}

C. Conley \cite{Cnl}, \cite{Cnl1} showed that any continuous flow $\Phi:
X\times \R\to X$ on a compact metric space $X$ \lq\lq decomposes\rq\rq\, into a
chain recurrent flow and a gradient-like flow. More precisely, he proved the
existence of a continuous function $L: X\to \R$ which (i) decreases along any
orbit of the flow in the complement $X-R$ of the chain recurrent set $R\subset
X$ of $\Phi$ and (ii) is constant on the connected components of $R$. Such a function $L$
is called a {\it Lyapunov function}\index{Lyapunov function of a flow} for
$\Phi$. Conley's existence result plays a fundamental role in his program of
understanding general flows as collections of isolated invariant sets linked by
heteroclinic orbits.

A more general notion of a {\it Lyapunov 1-form} was introduced in the paper \cite{farbe4}.
Lyapunov 1-forms, compared to
Lyapunov functions, allow us to go one step further and to analyze the flow
within the chain-recurrent set $R$ as well. Lyapunov 1-forms provide an
important tool in applying methods of homotopy theory to dynamical systems.

The problem of the existence of continuous Lyapunov 1-forms was first addressed in
paper \cite{FKLZ}, in the generality of compact metric
spaces, continuous flows and continuous closed 1-forms. In this section we present the result of  \cite{FKLZ1} dealing with
the smooth version of the problem: we are interested in constructing
smooth Lyapunov 1-forms for smooth flows on smooth manifolds. These
conditions are formulated in terms of homological properties of the flow; in
particular, we use Schwartzman's asymptotic cycles $\A_\mu(\Phi)\in H_1(M;\R)$
of the flow.

We would like to refer the reader to related work by J. Latschev \cite{Lat2, Lat3}.

Let $V$ be a smooth vector field on a smooth manifold $M$. Assume that $V$
generates a continuous flow $\Phi: M\times \R\to M$ and $Y\subset M$ is a
closed, flow-invariant subset.

\begin{definition}\label{def11}
\label{introlyap} A smooth closed 1-form $\omega$ on $M$ is called {\it a
Lyapunov 1-form}\index{Lyapunov 1-form of a flow} for the pair $(\Phi, Y)$ if
it has the following properties:
\begin{enumerate}
\item[{\rm ($\Lambda$1)}] The function $\iota_V(\omega)=\omega(V)$ is negative
on $M-Y$.

\item[{\rm ($\Lambda$2)}] $\omega|_Y=0$ and there exists a smooth function
$f:U\to \R$ defined on an open neighborhood $U$ of $Y$ such that
$\omega|_U=df$.
\end{enumerate}
\end{definition}

The above definition is a modification of the notion of a Lyapunov 1-form
introduced in \S 6 of \cite{farbe4}. Definition  \ref{def11} can also be compared with the definition of a Lyapunov
1-form in the continuous setting which was introduced in \cite{FKLZ}. Condition
($\Lambda$1) above is slightly stronger than condition (L1) of Definition 1 in
\cite{FKLZ}. Condition ($\Lambda$2) is similar to condition (L2) of Definition
1 from \cite{FKLZ} although they are not equivalent.

There exist several natural alternatives for condition ($\Lambda$2). One of
them is: \noindent
\begin{enumerate}
\item[{\rm ($\Lambda 2^\prime$)}] The 1-form $\omega$, viewed as a map $\omega:
M\to T^\ast(M)$, vanishes on $Y$.
\end{enumerate}
\noindent It is clear that ($\Lambda 2$) implies ($\Lambda 2^\prime$). We can
show that the converse is true under some additional assumptions:

\begin{lemma}\label{lmintegral} If the de Rham cohomology class $\xi$ of $\omega$ is integral,
$\xi=[\omega]\in H^1(M;\Z)$, then conditions ($\Lambda 2^\prime$) and ($\Lambda
2$) are equivalent.
\end{lemma}
\begin{proof}
Clearly we only need to show that ($\Lambda 2^\prime$) implies ($\Lambda 2$).
Since $\xi$ is integral, there exists a smooth map $\phi: M\to S^1$ such that
$\omega=\phi^\ast(d\theta)$, where $d\theta$ is the standard angular 1-form on
the circle $S^1$. Let $\alpha\in S^1$ be a regular value of $\phi$. Assuming
that ($\Lambda 2^\prime$) holds, it then follows that $U=M - \phi^{-1}(\alpha)$
is an open neighborhood of $Y$. Clearly $\omega|_U=df$ where $f:U \to \R$ is a
smooth function which is related to $\phi$ by $\phi(x)=\exp(if(x))$ for any
$x\in U$. Hence ($\Lambda 2$) holds.
\end{proof}

\begin{lemma}\label{lmenr} Conditions
($\Lambda 2^\prime$) and ($\Lambda 2$) are equivalent if $Y$ is a Euclidean
Neighborhood Retract (ENR).
\end{lemma}
\begin{proof} Again, we only have to establish ($\Lambda 2^\prime$) $\Rightarrow$
($\Lambda 2$). Since $Y$ is an ENR it admits an open neighborhood $U\subset M$
such that the inclusion $i_U: U\to M$ is homotopic to $i_Y\circ r$, where
$i_Y:Y\to M$ is the inclusion and $r: U\to Y$ is a retraction (see \cite{Do},
Chapter 4, \S 8, Corollary 8.7). Pick a base point $x_j$ in every
path-connected component $U_j$ of $U$ and define a smooth function $f_j:U_j\to
\R$ by
$$f_j(x) =\int\limits_{x_j}^x \omega, \quad x\in U_j.$$
The latter integral is independent of the choice of the integration path in
$U_j$ connecting $x_j$ with $x$. This claim is equivalent to the vanishing of
the integral $\int_\gamma \omega$ for any closed loop $\gamma$ lying in $U$. To
show this, we apply the retraction to see that $\gamma$ is homotopic in $M$ to
the loop $\gamma_1=r\circ \gamma$, which lies in $Y$; thus we obtain
$\int_\gamma \omega = \int_{\gamma_1} \omega=0$ because of ($\Lambda
2^\prime$). It is clear that the functions $f_j$ together determine a smooth
function $f:U\to \R$ with $df=\omega|_U$.
\end{proof}

A class of interesting examples of Lyapunov 1-forms can be obtained as follows.
Let $\omega$ be a smooth closed 1-form on a closed Riemannian manifold $M$.
Consider the negative gradient vector field $V$ of $\omega$, i.e., $\langle
V,X\rangle =-\omega(X)$ for any vector field $X$ on $M$ where $\langle\cdot
,\cdot\rangle$ denotes the Riemannian metric. Denote by $\Phi$ the flow induced
by the vector field $V$ and by $Y$ the set of zeros of $\omega$. Then clearly
conditions ($\Lambda 1$) and ($\Lambda 2^\prime$) are satisfied. If either the
cohomology class of $\omega$ is integral or $Y$ is an ENR, then (by the two
lemmas above) $\omega$ is a Lyapunov 1-form for the pair $(\Phi, Y)$.

The main goal of this section is to specify topological conditions which
guarantee that for a given vector field $V$ on $M$ and for a prescribed
cohomology class $\xi\in H^1(M;\R)$ there exists a Lyapunov 1-form $\omega$ of
the flow of $V$ with $[\omega]=\xi$.

\subsection*{Asymptotic cycles of Schwartzman}

Let $M$ be a closed smooth manifold and let $V$ be a smooth vector field. Let $\Phi:
M\times \R\to M$ be the flow generated by $V$. Consider a Borel measure $\mu$ on $M$
which is invariant under $\Phi$. According to S. Schwartzman \cite{Sc}, these data
determine a real homology class
$$\mathcal A_\mu= \A_\mu(\Phi)\, \in H_1(M;\R)$$
called {\it the asymptotic cycle of the flow $\Phi$ corresponding to the measure
$\mu$.}\index{Asymptotic cycle of a flow} The class $\A_\mu$ is defined as follows. For a
de Rham cohomology class $\xi \in H^1(M;\R)$ the evaluation $\langle \xi, \A_\mu\rangle
\in \R$ is given by the integral
\begin{eqnarray}
\langle \xi , \A_\mu\rangle  \, =\, \int_M \iota_V(\omega)d\mu,\label{integral}
\end{eqnarray}
where $\omega$ is a closed 1-form in the class $\xi$. Note that $\langle \xi,
\A_\mu\rangle$ is well-defined, i.e., it depends only on the cohomology class
$\xi$ of $\omega$; see \cite{Sc}, page 277. Indeed, when we replace $\omega$ by
$\omega'=\omega+df$, where $f: M\to \R$ is a smooth function, the integral in
(\ref{integral}) gets changed by the quantity
\begin{eqnarray}
\int_M V(f)d\mu = \lim_{s\rightarrow 0} \;\frac 1 s \int_M \bigl\{ f(x\cdot s)
- f(x)\bigr\}\, d\mu(x) .\label{integral1}
\end{eqnarray}
Here $V(f)$ denotes the derivative of $f$ in the direction of the vector field
$V$ and $x\cdot s$ stands for the flow $\Phi(x, s)$ of the vector field $V$.
Since the measure $\mu$ is flow invariant, the right integral in
(\ref{integral1}) vanishes for any $f$. It is clear that the right hand side of
(\ref{integral}) is a linear function of $\xi\in H^1(M;\R)$. Hence there exists
a unique real homology class $\A_\mu\in H_1(M;\R)$ which satisfies
(\ref{integral}) for all $\xi\in H^1(M;\R)$.

\subsection*{Necessary conditions}
We consider the flow $\Phi$ as being fixed and we vary the invariant measure
$\mu$. As the class $\A_\mu\in H_1(M;\R)$ depends linearly on $\mu$, the set of
asymptotic cycles $\A_\mu$ corresponding to all $\Phi$-invariant positive
measures $\mu$ forms a convex cone in the vector space $H_1(M;\R)$.

\begin{proposition} Assume that there exists a Lyapunov 1-form for $(\Phi,Y)$ lying in a cohomology class
$\xi\in H^1(M;\R)$. Then
\begin{eqnarray}
\langle \xi, \A_\mu\rangle \, \leq \, 0\label{new11}
\end{eqnarray}
for any $\Phi$-invariant positive Borel measure $\mu$ on $M$; equality in
(\ref{new11}) takes place if and only if the complement of $Y$ has measure
zero. Furthermore, the restriction of $\xi$ to $Y$, viewed as a \v Cech
cohomology class
$$\xi|_{Y}\in \check H^1(Y;\R),$$
vanishes, $\xi|_{Y} =0$.
\end{proposition}
\begin{proof} Let $\omega$ be a Lyapunov 1-form for $(\Phi, Y)$ lying in the class $\xi$.
According to Definition \ref{def11}, the function $\iota_V(\omega)$ is negative
on $M-Y$ and vanishes on $Y$. We obtain that the integral
$$\int_M\iota_V(\omega)d\mu = \langle \xi, \A_\mu\rangle$$
is nonpositive.

Assuming $\mu(M-Y)>0$, we find a compact $K\subset M-Y$ with $\mu(K)>0$; This
follows from the theorem of Riesz; see, e.g., \cite{Lang}, Theorem 2.3(iv),
page 256. There is a constant $\epsilon >0$ such that $\iota_V(\omega)|_K \leq
-\epsilon$. Therefore, one has
$$\int_M \iota_V(\omega)d\mu \leq -\epsilon \mu(K) \, <\, 0.$$
Hence, the value $\langle\xi,\A_\mu\rangle$ is strictly negative if the measure
$\mu$ is not supported in $Y$.

To prove the second statement, we observe (see \cite{Sp}) that the \v Cech
cohomology $\check H^1(Y;\R)$ equals the direct limit of the singular
cohomology
$$\check H^1(Y;\R) = \lim_{W\supset Y} H^1(W;\R),$$
where $W$ runs over open neighborhoods of $Y$. It is clear in view of condition
($\Lambda$2) that $\xi|_U=0\in H^1(U;\R)$ (by the de Rham theorem). Hence the
result follows. \end{proof}

\subsection*{Chain-recurrent set $R_\xi$}
Given a flow $\Phi$, our aim is to construct a Lyapunov 1-form $\omega$ for a
pair $(\Phi, Y)$ lying in a given cohomology class $\xi\in H^1(M;\R)$. A
natural candidate for $Y$ is the subset $R_\xi=R_\xi(\Phi)$ of the
chain-recurrent set $R=R(\Phi)$ which was defined in \cite{FKLZ}. We briefly
recall the definition.

Fix a Riemannian metric on $M$ and denote by $d$ the corresponding distance
function. Given any $\delta>0$, $T>1$, a $(\delta,T)$-{\it chain from $x\in M$
to $y\in M$}  is a finite sequence $x_0=x, x_1, \dots, x_{N}=y$ of points in
$M$ and numbers $t_1, \dots, t_N\in \R$ such that $t_i\geq T$ and
$d(x_{i-1}\cdot t_i, x_i) < \delta$ for all $1 \leq i \leq N$. Here we use the
notation $\Phi(x,t)=x\cdot t$. The {\it chain recurrent set} $R=R(\Phi)$ of the
flow $\Phi$ is defined as the set of all points $x\in M$ such that for any
$\delta>0$ and $T > 1$ there exists a $(\delta,T)$-chain starting and ending at
$x$. The chain recurrent set $R$ is closed and invariant under the flow.

Given a cohomology class $\xi\in H^1(M;\R)$, there is a natural covering space
$p_\xi: \tilde M_\xi\to M$ associated with $\xi$. A closed loop $\gamma:
[0,1]\to M$ lifts to a closed loop in $\tilde M_\xi$ if and only if the value
of the cohomology class $\xi$ on the homology class $[\gamma]\in H_1(M;\Z)$
vanishes, $\langle \xi, [\gamma]\rangle =0$.

The flow $\Phi$ lifts uniquely to a flow $\tilde \Phi$ on the covering $\tilde
M_\xi$. Consider the chain recurrent set $R(\tilde \Phi)\subset \tilde M_\xi$
of the lifted flow and denote by $R_\xi=p_\xi(R(\tilde \Phi))\subset M$ its
projection onto $M$. The set $R_\xi$ is referred to as {\it the chain recurrent
set associated to the cohomology class $\xi$.}\index{Chain recurrent set
$R_\xi$} It is clear that $R_\xi$ is closed, $\Phi$-invariant, and
$R_\xi\subset R$.

We denote by $C_\xi$ the complement of $R_\xi$ in $R$,
$$C_\xi = R- R_\xi.$$

Let us mention the following example illustrating the definition of $R_\xi$.
Consider a smooth flow on a closed manifold $M$ whose chain recurrent set $R$
consists of finitely many rest points and periodic orbits. Given a cohomology
class $\xi\in H^1(M;\R)$, the chain recurrent set $R_\xi$ is the union of all
the rest points and of those periodic orbits whose homology classes $z\in
H_1(M;\Z)$ satisfy $\langle \xi, z\rangle =0$.

In general, any fixed point of the flow belongs to $R_\xi$. The points of a
periodic orbit belong to $R_\xi$ if the homology class $z\in H_1(X;\Z)$ of this
orbit satisfies $\langle \xi, z\rangle =0$.

It may happen that the points of a periodic orbit belong to $R_\xi$ although
$\langle \xi, z\rangle\not=0$ for the homology class $z$ of the orbit. This
possibility is illustrated by the following example; cf. \cite{FKLZ}.
\begin{figure}[h]
\begin{center}
\resizebox{6.5cm}{5cm}{\includegraphics[230,419][410,572]{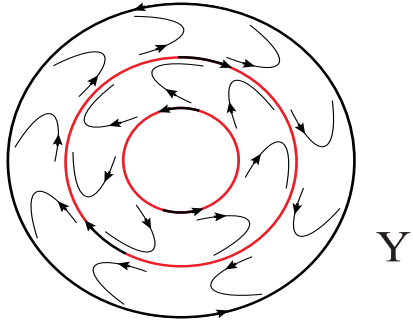}}
\end{center}
\caption{The flow on the planar ring $Y$.}\label{example2}
\end{figure}
Consider the flow on the planar ring $Y\subset \C$ shown in Figure
\ref{example2}. In the polar coordinates $(r, \phi)$ the ring $Y$ is given by
the inequalities $1\leq r\leq 3$ and the flow is given by the differential
equations
$$\dot r = (r-1)^2(r-3)^2(r-5)^2, \quad \dot\phi =\sin\left(r\cdot \frac{\pi}{2}\right).$$
Let $C_k$, where $k=1, 2, 3$, denote the circle $r=2k-1$. The circles $C_1,
C_2, C_3$ are invariant under the flow. The motion along the circles $C_1$ and
$C_3$ has constant angular velocity 1. Identifying any point $(r, \phi)\in C_1$
with $(5r, \phi)\in C_3$, we obtain a torus $X=Y/\simeq$ and a flow $\Phi:
X\times \R\to X$. The images of the circles $C_1, C_2, C_3\subset Y$ represent
two circles $C'_1=C'_3$ and $C'_2$ on the torus $X$. Let $\xi\in H^1(X;\R)$ be
a nonzero cohomology class which is the pullback of a cohomology class of $Y$.
One verifies that in this example the set $R_\xi(\Phi)$ coincides with the
whole torus $X$. In particular, $R_\xi(\Phi)$ contains the periodic orbits
$C_1'=C'_3$ and $C_2'$ although clearly $\langle \xi, [C'_k]\rangle \not=0$.

A different definition of $R_\xi$ which does not use the covering space $\tilde
M_\xi$ can be found in \cite{FKLZ}.

To state the main result of this section, we need the following notion. A
$(\delta, T)${\it -cycle} of the flow $\Phi$ is defined as a pair $(x,t)$,
where $x\in M$ and $t>T$ such that $d(x, x\cdot t)<\delta$. If $\delta$ is
small enough, then any $(\delta, T)$-cycle determines in a canonical way a
unique homology class $z\in H_1(M;\Z)$ which is represented by the flow
trajectory from $x$ to $x\cdot t$ followed by a \lq\lq {\it short}\rq\rq\, arc
connecting $x\cdot t$ with $x$. See \cite{FKLZ} for more details.
\index{cycle@$(\delta, T)$-Cycle}
\begin{figure}[h]
\begin{center}
\resizebox{6cm}{5cm} {\includegraphics[80,360][429,651]{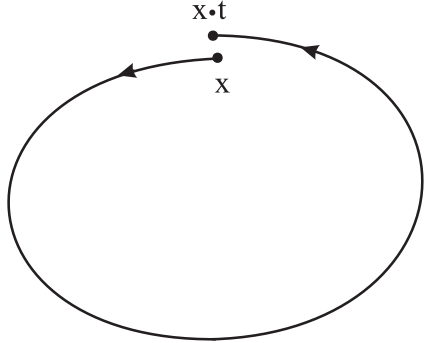}}
\caption{$(\delta, T)$-cycle.}
\end{center}
\end{figure}
\begin{theorem}[Farber, Kappeler, Latschev, Zehnder \cite{FKLZ1}] \label{lyapmain}
Let $V$ be a smooth vector field on a smooth closed manifold $M$. Denote by
$\Phi: M\times \R\to M$ the flow generated by $V$. Let $\xi\in H^1(M;\R)$ be a
cohomology class such that the restriction $\xi|_{R_\xi}$, viewed as a \v Cech
cohomology class $\xi|_{R_\xi}\in \check H^1(R_\xi;\R)$, vanishes. Then the
following properties of $\xi$ are equivalent:

{\rm (I)} There exists a smooth Lyapunov 1-form for $(\Phi, R_\xi)$ in the
cohomology class $\xi$ and the subset $C_\xi$ is closed.

{\rm (II)} For any Riemannian metric on $M$ there exist $\delta>0$ and $T>1$
such that the homology class $z\in H_1(M;\Z)$ associated with an arbitrary
$(\delta, T)$-cycle $(x,t)$ of the flow, with $x\in C_\xi$, satisfies $\langle
\xi, z\rangle \leq -1$.

{\rm (III)} The subset $C_\xi$ is closed and there exists a constant $\eta>0$
such that for any $\Phi$-invariant positive Borel measure $\mu$ on $M$ the
asymptotic cycle $\A_\mu=\A_\mu(\Phi)\in H_1(M;\R)$ satisfies
\begin{eqnarray}
\langle \xi, \A_\mu\rangle \, \leq\,   -\eta\cdot \mu(C_\xi).\label{tkinequa3}
\end{eqnarray}

{\rm (IV)} The subset $C_\xi$ is closed and for any $\Phi$-invariant positive
Borel measure $\mu$ on $X$ with $\mu(C_\xi)>0$, the asymptotic cycle
$\A_\mu=\A_\mu(\Phi)\in H_1(M;\R)$ satisfies
\begin{eqnarray}
\langle \xi, \A_\mu\rangle < 0.\label{tkinequa}
\end{eqnarray}
\end{theorem}

Condition (\ref{tkinequa3}) can be reformulated using the notion of a
quasi-regular point. Recall that $x\in X$ is a {\it
quasi-regular}\index{Quasi-regular point} point of the flow $\Phi: X\times
\R\to X$ if for any continuous function $f: X\to \R$ the limit
\begin{eqnarray}
\lim_{t\to \infty} \frac{1}{t} \int_0^t f(x\cdot s)ds
\end{eqnarray}
exists. It follows from the ergodic theorem that the subset $Q\subset X$ of all
quasi-regular points has full measure with respect to any $\Phi$-invariant
positive Borel measure on $X$. From the Riesz
representation theorem (see, e.g., \cite{Rud}, page 256) one deduces that for
any quasi-regular point $x\in Q$ there exists a unique positive flow-invariant
Borel measure $\mu_x$ with $\mu_x(X)=1$ satisfying
\begin{eqnarray}
\lim_{t\to \infty} \frac{1}{t} \int_0^t f(x\cdot s)ds = \int_X fd\mu_x
\end{eqnarray}
for any continuous function $f$. We use below the well-known fact that any
positive, $\Phi$-invariant Borel measure $\mu$ with $\mu(X)=1$ belongs to the
weak$^\ast$ closure of the convex hull of the set of measures $\mu_x$, $x\in
Q$.

If the subset $C_\xi\subset X$ is closed, and hence compact, one can apply the
above-mentioned facts to the restriction of the flow to $C_\xi$. Let $\omega$
be an arbitrary smooth closed 1-form lying in the cohomology class $\xi$. For
any quasi-regular point $x\in C_\xi$ of the flow $\Phi|_{C_\xi}$ one has
\begin{eqnarray}
\lim_{t\to \infty} \frac{1}{t}\int_x^{x\cdot t} \omega = \lim_{t\to
\infty}\frac{1}{t}
\int_0^t \iota_V(\omega)(x\cdot s)ds\\
= \int_{M}\iota_V(\omega)d\mu_x = \langle \xi, \mathcal
A_{\mu_x}\rangle.\nonumber
\end{eqnarray}

We therefore conclude that condition {\rm (III)} is equivalent to:

{$\rm (III')$} {\it The subset $C_\xi$ is closed and there exists a constant
$\eta>0$ such that for any quasi-regular point $x\in C_\xi$,
\begin{eqnarray}
\lim_{t\to \infty} \frac{1}{t}\int_x^{x\cdot t} \omega\, \, \leq\, \,
-\eta,\label{stam2}
\end{eqnarray}
where $\omega$ is an arbitrary closed 1-form in the class $\xi$.}

The value of the limit (\ref{stam2}) is independent of the choice of a closed
1-form $\omega$; the only requirement is that $\omega$ lies in the cohomology
class $\xi$.

In the special case $\xi=0$ the set $C_\xi$ is empty and $R=R_\xi$. The above
statement then reduces to the following well-known theorem of C. Conley (see
\cite{Cnl} and  \cite{Shu}, Theorem 3.14):

\begin{theorem}\label{conley} {\rm (C. Conley)} Let $V$ be a smooth vector field on a smooth closed manifold $M$. Denote by
$\Phi: M\times \R\to M$ the flow generated by $V$ and by $R$ the chain
recurrent set of $\Phi$. Then there exists a smooth Lyapunov function $L: M\to
\R$ for $(\Phi, R)$. This means that $V(L)<0$ on $M-R$ and $dL=0$ pointwise on
$R$.
\end{theorem}

A proof of Theorem \ref{lyapmain} can be found in \cite{FKLZ1}.

We also refer to \cite{FKLZ} where continuous Lyapunov 1-forms are studied. The
paper \cite{FKLZ} also contains a discussion comparing the results about the
existence of Lyapunov 1-forms with the theorems of D. Fried \cite{Fr:1} about
the existence of cross sections for flows.

Finally we describe a class of examples of flows for which there exists a
cohomology class $\xi$ satisfying all the conditions of Theorem \ref{lyapmain}.

Let $M$ be a closed smooth manifold with a smooth vector field $v$. Let $\Psi:
M\times \R\to M$ be the flow of $v$. Assume that the chain recurrent set
$R(\Psi)$ is a union of two disjoint closed sets $R(\Psi)=R_1\cup R_2$, where
$R_1\cap R_2=\emptyset$. Out of this data we construct a flow $\Phi$ on
$$X=M\times S^1$$
such that $R_\xi(\Phi)=R_1\times S^0$, $C_\xi=R_2\times S^1$. Here $\xi\in
H^1(X;\Z)$ denotes the de Rham cohomology class of the 1-form $-d\theta$ where
$\theta\in [0,2\pi]$ denotes the angle coordinate on the circle $S^1$;
$S^0\subset S^1$ is a two-point set.

We will need two vector fields $w_1$ and $w_2$ on $S^1$, $w_1=\cos
(\theta)\cdot \frac{\partial}{\partial \theta}$ and
$w_2=\frac{\partial}{\partial \theta}$. The field $w_1$ has two zeros $\{p_1,
p_2\}=S^0\subset S^1$ corresponding to the angles $\theta=\pi/2$ and
$\theta=3\pi/2$.

Let $f_i: M\to [0,1]$, where $i=1, 2$, be two smooth functions having disjoint
supports and satisfying $f_1|_{R_1}=1$, $f_2|_{R_2}=1$.

Consider the flow $\Phi: X\times \R\to X$ determined by the vector field
$$V\, = \, v+f_1w_1+f_2w_2.$$
Any trajectory of $V$ has the form $(\gamma(t), \theta(t))$, where
$\dot\gamma(t)=v(\gamma(t))$, i.e., $\gamma(t)$ is a trajectory of $v$. It
follows that the chain recurrent set of $V$ is contained in $R(\Psi)\times
S^1$. Over $R_1$ we have the vertical vector field $w_1$ along the circle which
has two points $S^0\subset S^1$ as its chain recurrent set. Over $R_2$ we have
the vertical vector field $w_2$ which has all of $S^1$ as the chain recurrent
set. We see that $R_1\times S^0 = R_\xi(\Phi)$, $R_2\times S^1=C_\xi$. Hence
$$\xi|_{R_\xi}=0$$ and $C_\xi$ is closed. One easily checks that condition ({\rm {III})
of Theorem \ref{lyapmain} (and hence the other conditions as well) is
satisfied.

Further examples can be found in \S 7 of \cite{FKLZ}.

An approach to study dynamical systems using cocycles instead of Lyapunov closed 1-forms was suggested by H. Fan and J. Jost \cite{FJ1}, \cite{FJ2}.

\section{Notions of category with respect to a cohomology class}\label{lscat}

In this section we describe several generalizations of the classical notion of Lusternik - Schnirelmann category which reflect interesting dynamical properties of flows on manifolds. These new notions depend on a choice of a cohomology class $\xi\in H^1(X;\R)$ and are all equal to the classical invariant $\cat(X)$ in the case $\xi=0$.

\subsection{Movable subsets} First we define the notion of movability of a subset with respect to a closed one-form.

 \begin{definition}\label{defmovable} Let $\omega$ be a closed one-form on a topological space $X$.
 A subset $A\subset X$ is called $N$-movable with respect to $\omega$ (where $N$ is an integer), if there is a homotopy $H: A\times [0,1]\to X$ such that $H(a,0)=a$ for all $a\in A$, and
\begin{eqnarray*}
\int_a^{H_1(a)} \omega & \leq &   -N
\end{eqnarray*}
for all $a\in A$.
\end{definition}
Note that the latter inequlaity

Below we shall think of the form $\omega$ being fixed and of integer $N$ being large, or tending to infinity.
This may happen only in the case when the cohomology class $\xi= [\omega]\in H^1(X;\R)$ is nonzero. Indeed, it is clear that in the case of an exact form $\omega=df$, a nonempty subset $A\subset X$ can be $N$-movable with respect for $\omega$ only for $N\le \max f - \min f$.

Any subset $A\subset X$ such that the inclusion $A\to X$ is null-homotopic is $N$-movable with respect to any closed 1-forms $\omega$ assuming that the cohomology class $\xi=[\omega]\not= 0\in H^1(X; \R)$ is nonzero.

As another example consider the following situation. Assume that $X$ is a closed smooth manifold and $\omega$ is a smooth closed one-form on $X$. If $\omega$ has no zeros then any subset $A\subset X$ is $N$-movable with respect to $\omega$ for any integer $N$. In this case a homotopy
$H:A\times I\to X$ is given by the gradient flow (appropriately scaled) of $\omega$ with respect to a Riemannian metric on $X$.

In general there are topological obstructions to movability of subsets which are captured by topological
invariants of Lusternik - Schnirelmann type described below.

\subsection{Various notions of category}

Here is the first version of the notion of a category $\cat(X,\xi)$ with respect to a cohomology class. The definition given below is equivalent to the original definition of \cite{farber}, see \cite{farkap}.

\begin{definition}\label{defcat}
 Let $X$ be a finite CW-complex and $\xi\in H^1(X;\R)$. Fix a closed 1-form $\omega$ representing $\xi$. Then the number $$\cat(X,\xi)$$ is defined the minimal integer $k$ such that for every $N>0$ there exists a closed subset $A\subset X$ which is $N$-movable with respect to $\omega$ and such that $\cat_X(X-A)\leq k$.
\end{definition}

Recall that for $A\subset X$, $\cat_X(A)$ is the minimal integer $k$ such that $A$ can be covered by $k$ open subsets of $X$ each of which is null-homotopic in $X$.

Note that the number $\cat(X, \xi)$ does not depend on the choice of a closed one-form $\omega$ and may depend only on the cohomology class $\xi=[\omega]\in H^1(X,;\R)$. Indeed, if $\omega'$ is another closed 1-form lying in the same cohomology $\xi=[\omega']$ then $\omega'=\omega+df$ for some continuous function $f: X\to \R$. Then for any path $\gamma:[a, b]\to X$ one has
$$\int_\gamma \omega' = \int_\gamma \omega +f(\gamma(b))- f(\gamma(a)).$$
Since $X$ is compact we see that there exists a constant $C>0$ such that for all paths $\gamma$ in $X$ one has
$$\left| \int_\gamma \omega' - \int_\gamma \omega\right|< C.$$
Hence any subset $A\subset X$ which is $N$-movable with respect to $\omega$ is $(N-C)$-movable with respect to $\omega'$.

 It may happen that $$\cat(X, \xi) \not=\cat(X, -\xi),$$ see Example 10.10 in \cite{farbook}. Therefore it makes sense to introduce (as was suggested first in  \cite{schman}) a symmetric version $\cat_s(X, \xi)$
 of $\cat(X, \xi)$ which suit better dynamics applications:

 \begin{definition}\label{defcats}
 Let $X$ be a finite CW-complex and $\xi\in H^1(X;\R)$. Fix a closed 1-form $\omega$ representing $\xi$. Then the number $$\cat_s(X,\xi)$$ is defined the minimal integer $k$ such that for every $N>0$ there exists a closed subset $A\subset X$ which is $N$-movable with respect to both $\omega$ and $-\omega$ and such that $\cat_X(X-A)\leq k$.
\end{definition}

Clearly, one has
$$\cat_s(X,\xi) =\cat_s(X, -\xi)\ge \max\{\cat(X, \xi), \cat(X, -\xi)\}.$$

Reversing the quantifiers in Definition \ref{defcat} gives a slightly different notion $\cat^1(X, \xi)$ which was initially introduced in \cite{farkap}:
\begin{definition}\label{defcatone}
 Let $X$ be a finite CW-complex and $\xi\in H^1(X;\R)$.
 Fix a closed 1-form $\omega$ representing $\xi$. Then the number $\cat^1(X,\xi)$ is defined as the minimal integer $k$ such that there exists a closed subset $A\subset X$, which is $N$-movable with respect to $\omega$ for every $N>0$, and such that
 $\cat_X(X-A)\leq k$.
\end{definition}

The invariant $\cat^1(X,\xi)$ also has a symmetric modification $\cat_s^1(X, \xi).$ The latter is defined similarly to $\cat^1(X, \xi)$ with the only difference that in Definition \ref{defcatone} one requires that the closed subset $A\subset X$ is $N$-movable with respect to both one-form $\omega$ and $-\omega$.

Yet another invariant similar in spirit is given by the following definition:

\begin{definition}\label{bigcat}
 Let $X$ be a finite CW-complex and $\xi\in H^1(X;\R)$. Fix a closed 1-form $\omega$ representing $\xi$. Then $\Cat(X,\xi)$ is defined as the minimal integer $k$ such that there exists an open subset $U\subset X$ satisfying

 (a) $\cat_X(X-U)\leq k$;

 (b) for some homotopy $h:U\times [0,\infty) \to X$ one has
 \begin{eqnarray*} \label{limits}
 h(x,0)=x \quad \mbox{and}\quad \lim_{t\to \infty} \int_x^{h_t(x)} \omega \, =\, - \infty
\end{eqnarray*}
 for any point $x\in U$;

 (c) the limit in (b) is uniform in $x\in U$.
 \end{definition}

The symbol $\int_x^{h_t(x)}\omega$ denotes the integral $\int_\gamma \omega$ where the path $\gamma:[0,t]\to X$ is given by $\gamma(\tau)=h_\tau(x)$.

\begin{remark}
The definition of $\Cat(X, \xi)$ given above differs from the original definition suggested in \cite{farbe4}; the latter was symmetric and stronger in some examples.
\end{remark}

 Independence of $\cat^1(X,\xi)$ and $\Cat(X,\xi)$ of the choice of closed one-form $\omega$ with $[\omega]=\xi$ follows similarly to the argument given above for the case of $\cat(X, \xi)$.

The role the invariants $\cat(X,\xi)$, $\cat^1(X, \xi)$ and $\Cat(X,\xi)$ play in dynamics will be discussed in the following sections.

\begin{lemma} In the case of the trivial cohomology class $\xi=0\in H^1(X;\R)$ the numbers
$$\cat(X, \xi) = \cat^1(X, \xi)= \Cat(X, \xi) = \cat(X)$$
coincide with the classical Lusternik - Schnirelmann category $\cat(X)$.
\end{lemma}
\begin{proof} We will give the proof for $\cat(X,\xi)$. If $\xi=0$ we may take $\omega=0$ and hence any subset $A\subset X$ which is $N$-movable with respect to $\omega$ is empty assuming that $N>0$. Hence we obtain that $\cat(X,\xi)=\cat_X(X)=\cat(X)$.
\end{proof}
\begin{example}\label{ex1} Suppose that $p: X \to S^1$ is a locally trivial fibration and $\xi=p^\ast(\eta)$ where $\eta\in H^1(S^1;\R)$, $\eta\not=0$. Then
\begin{eqnarray}\label{varcats1}
 \cat(X,\xi)= \cat^1(X,\xi)= \Cat(X,\xi)=0
\end{eqnarray}
(see \cite{farbook}). Note that $\cat(X)$ can be arbitrarily large in this example.
Hence the new \lq\lq cats\rq\rq\, are not always equal the classical $\cat(X)$. \end{example}

Comparing the definitions one trivially has
\begin{eqnarray}\label{varcats}
 \cat(X,\xi)\leq \cat^1(X,\xi)\leq \Cat(X,\xi).
\end{eqnarray}
We will show later that $\cat(X,\xi)$ may be significantly smaller than $\cat^1(X, \xi)$. At the moment we do not have examples
when $\cat^1(X,\xi)< \Cat(X, \xi)$ but we believe that such examples exist.
\begin{lemma}\label{lm10}
Suppose that $\xi\not=0$ and $X$ is connected. Then
\begin{eqnarray}\label{catineq}
  \Cat(X,\xi)\leq \cat(X)-1.
\end{eqnarray}
\end{lemma}
The proof of this Lemma can be found in \cite{farbook}.

As another example consider a bouquet $X=Y\vee S^1$, where $Y$ is a finite polyhedron,
and assume that the class $\xi\in H^1(X,\R)$ satisfies
$\xi|_Y =0$ and $\xi|_{S^1}\neq 0$. One can show (see \cite{farbook}) that then
\[{\rm cat}(X,\xi) \, =\, {\rm cat}(Y)-1.\]

\subsection{Homotopy invariance}
\begin{proposition}
 Let $f:Y\to X$ be a homotopy equivalence between finite connected CW-complexes, and $\xi\in H^1(X;\R)$. Then
\begin{eqnarray*}
 \cat(X,\xi)&=&\cat(Y,f^\ast\xi).
\end{eqnarray*}
This statement remains true if one replaces $\cat$ by $\cat^1$ or $\Cat$.
\end{proposition}

\begin{proof}
 Let $\omega$ be a closed 1-form on $X$ representing $\xi$. Let $g:X\to Y$ and $K: X\times [0,1]\to X$ satisfy $K:{\rm id}_X\simeq fg$. Because of compactness of $X$, there exists a constant $C>0$ such that
\begin{eqnarray*}
 \left|\int_{\alpha_x}\omega\right|&\leq& C
\end{eqnarray*}
for all $x\in X$, where $\alpha_x(t)=K(x,t)$.
Assume that a subset $A\subset Y$ is $(N+C)$-movable with respect to the closed one-form $\omega'=f^\ast\omega$ on $Y$. Let $H: A\times [0,1]\to Y$ be a homotopy such that $H(a, 0)=a$ and $\int_{\gamma_a}\omega'\geq N+C$ for all $a\in A$ where $\gamma_a(t)=H(a, 1-t)$. Then the set $g^{-1}(A)\subset X$ is $N$-movable with respect to $\omega$. Indeed, one may define a homotopy $h_t: g^{-1}(A) \to X$ by
$$h_t(x)= \left\{
\begin{array}{ll}
K(x, 2t)&\mbox{for}\quad 0\le t\le 1/2,\\
f(H(g(x), 2t-1))&\mbox{for}\quad 1/2\leq t\leq 1.
\end{array}\right.
$$
and for $x\in g^{-1}(A)$ one has $\int_{\beta_x}\omega \geq N$ where $\beta_x(t)=h_{1-t}(x)$.

If $\cat(Y, f^\ast \xi)\le k$ then for any $N$ there exists a closed subset $A\subset Y$ which is $(N+C)$-movable with respect to $\omega'=f^\ast \omega$ and such that $\cat_Y(Y-A)\leq k$. Then $g^{-1}(A)\subset X$ is $N$-movable with respect to $\omega$ and, as is well-known, one has $\cat_X(X-g^{-1}(A))\leq k$.
This proves that $\cat(X,\xi)\leq \cat(Y,f^\ast\xi)$.

The inverse inequality follows similarly if $gf\simeq \id_Y$.
The arguments for $\cat^1(X, \xi)$ and $\Cat(X, \xi)$ are similar.
\end{proof}
\subsection{Spaces of category zero} Here we collect some simple observations about spaces of category zero. More information can be found in
\cite{schman}, \S3.

\begin{lemma} Let $X$ be a finite CW complex and $\xi\in H^1(X;\R)$. The following properties
are equivalent:

{\rm (a)} $\cat(X, \xi)=0$.

{\rm (b)} $\cat^1(X,\xi)=0$.

{\rm (c)} $\Cat(X, \xi)=0$.

{\rm (d)} There exists a continuous closed 1-form $\omega$ on $X$ representing
$\xi$ (in the sense of subsection \ref{cohclass}) and a homotopy $h_t: X\to X$,
where $t\in [0,1]$, such that for any point $x\in X$ one has
\begin{eqnarray}\label{ineq}
\int_x^{h_1(x)}\omega \, <\, 0.
\end{eqnarray}
In (\ref{ineq}) the integral is calculated along the curve $t\mapsto h_t(x)$,
$t\in [0,1]$.

{\rm (e)} For any continuous closed 1-form $\omega$ on $X$ representing $\xi$
there exists a homotopy $h_t: X\to X$, where $t\in [0,1]$, such that for any
point $x\in X$ inequality (\ref{ineq}) holds.
\end{lemma}

\begin{proof} By Definition \ref{defcat}, $\cat(X,\xi)=0$ means that
the whole space $X$ is $N$-movable
 for any $N>0$,
i.e. given $N>0$, there exists a homotopy $H_t: X\to X$, where $t\in [0,1]$,
such that $H_0(x)=x$ and
\begin{eqnarray}\label{large} \int_x^{H_1(x)}\omega <-N
\end{eqnarray}
for any $x\in X$. Hence, (a) implies (d).

 Conversely, given property (d), using compactness of $X$
we find $\epsilon>0$ such that (\ref{ineq}) can be
 replaced by $\int_x^{h_1(x)}\omega <-\epsilon$. Now, one may iterate this
 deformation as follows. The $k$-th iteration is a homotopy $H^k_t: X\to X$,
 where $t\in [0,1]$, defined as follows. Denote by $h_1^{(i)}:X\to X$ the $i$-fold
 composition $h_1^{(i)} = h_1\circ h_1\circ \dots\circ h_1$ ($i$ times). Then for $t\in
 [i/k, (i+1)/k]$ one has
 \begin{eqnarray}
 H^k_t(x) = h_{kt-i}(h_1^{(i)}(x)).
 \end{eqnarray}
 If for any $x\in X$ one has $\int_x^{h_1(x)}\omega <-\epsilon$ then
 for the $k$-th iteration we obtain $\int_x^{H^k_1(x)}\omega <-k\epsilon$
 and (d) follows assuming that $k>N/\epsilon$.
This shows equivalence between (a) and (d).

(e) $\implies$ (d) is obvious. Now suppose that (d) holds for $\omega$ and let
$\omega_1$ be another continuous closed one-form lying in the same cohomology
class, i.e. $\omega_1=\omega +df$ where $f: X\to \R$ is continuous. Using compactness of $X$ we may find $C$ such that for
any path $\gamma:[0,1]\to X$ one has $|\int_\gamma df|<C$. Fix $N>C$ and apply
equivalence between (a) and (c) to find a homotopy $h_t:X\to X$ with
$\int_x^{h_1(x)}\omega <-N$. Then one has $\int_x^{h_1(x)}\omega_1 <0$, i.e.
(e) holds.

It is obvious that (b) $\implies$ (a). Hence we are left to show that (a)
implies (b). Given (b) fix a deformation as described in (c). Let $C>0$ be such
that for any $x\in X$ and for any $t\in [0,1]$ one has $\int_1^{h_t(x)}\omega
<C$. Then for any iteration $H^k_t: X\to X$ (see above) one has
$\int_x^{H^k_t(x)}\omega <C$ and the result follows.
\end{proof}

In the case when $X$ is a closed smooth manifold a deformation as appearing in
(b) can be constructed as the flow generated by a vector field $v$ on $X$
satisfying $\omega(v)<0$.

The remark of the previous paragraph explains why the following statement can
be viewed as an analogue of the classical Euler - Poincar\'e theorem:

\begin{theorem}\label{ep} $\cat(X,\xi)=0$ implies $\chi(X)=0$.
\end{theorem}
\begin{proof} Suppose that $\cat(X,\xi)=0$. Then $\xi\not=0$, i.e.
the rank $r$ of class $\xi$ is positive, $r>0$. By Lemma \ref{bair} below there
exists $\xi$-transcendental bundle $L\in \V_\xi=(\C^\ast)^r$. If $H^q(X;L)\not=0$ for
some $q$ then one may apply Theorem \ref{perfect} below with $k=0$ obtaining
$\cat(X,\xi)>0$ and contradicting our hypothesis. Hence $H^q(X;L)=0$ for all
$q$ which implies
$$\chi(X)=\sum_q (-1)^q \dim H^q(X;L)= 0.$$
\end{proof}
Pairs $(X,\xi)$ with $\cat(X,\xi)=0$ form an {\it \lq\lq ideal\rq\rq}\,  in the
following sense:

\begin{lemma}\label{ideal} Let $X_1$ and $X_2$ be finite cell complexes and $\xi_1\in
H^1(X_1;\R)$, $\xi_2\in H^1(X_2;\R)$. If $\cat(X_1, \xi_1)=0$ then
$\cat(X_1\times X_2, \xi)=0$ where $\xi = \xi_1\times 1+ 1\times \xi_2\in
H^1(X_1\times X_2;\R).$
\end{lemma}

\begin{proof} The statement follows directly by applying the definitions.
\end{proof}

\section{Focusing effect}\label{secfocus}

In this section we start surveying applications of the invariant $\cat(X,\xi)$ in dynamics. We describe the focusing effect discovered in \cite{farber}, see also \cite{farbook}.

The nature of the Lusternik-Schnirelmann
type theory for closed 1-forms is very much distinct from
both the classical Lusternik-Schnirelmann theory for functions and
from the Novikov theory. As we know, in any nonzero cohomology class $\xi\in
H^1(M;\Z)$, $\xi\not=0$, one may always find a representing closed 1-form
having at most one (highly degenerate) zero. Hence, quantitative estimates on the number of zeros may
be obtained only under some additional assumptions. It turns out that these
additional assumptions can be expressed in terms of some interesting dynamical
properties of gradient-like vector fields for a given closed 1-form. This makes
the new theory potentially very useful for dynamics.

Let $v$ be a smooth vector field on a closed smooth manifold $M$. Recall that {\it a homoclinic orbit} of $v$ is defined as an integral trajectory
$\gamma(t)$,
$$\dot\gamma(t) = v(\gamma(t)), \quad t\in \R$$
 such that both limits $\lim_{t\to + \infty}\gamma(t)$
 and $\lim_{t\to - \infty}\gamma(t)$ exist and are equal
 $$\lim_{t\to + \infty}\gamma(t)=\lim_{t\to - \infty}\gamma(t).$$

More generally, {\it a homoclinic cycle of length $n$} is a
sequence of integral trajectories $\gamma_1(t)$, $\gamma_2(t),$ $\dots, \gamma_n(t)$ of $v$ such that
$$\lim_{t\to +\infty} \gamma_i(t) = p_{i+1}= \lim_{t\to -\infty} \gamma_{i+1}(t)$$
for $i=1, \dots, n-1$ and
$$\lim_{t\to +\infty} \gamma_n(t) = p_1 = \lim_{t\to -\infty} \gamma_{1}(t).$$

\begin{figure}[h]
\begin{center}
\resizebox{8cm}{3.6cm}{\includegraphics[53,483][550,705]{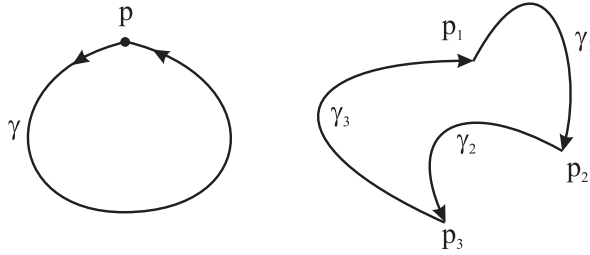}}
\end{center}
\caption{Homoclinic orbit (left) and homoclinic cycle (right).}\label{homocli}
\end{figure}

Homoclinic orbits were discovered by H. Poincar\'e and were studied by S.
Smale. In the mathematical literature there are many results about existence of
homoclinic orbits in Hamiltonian systems. For obvious reasons homoclinic orbits cannot exist in
gradient systems for functions corresponding to
the classical Lusternik-Schnirelmann theory.

One of the main theorems of this section (Theorem \ref{lsmain1}, originally established in \cite{farber}) states that any smooth closed
1-form $\omega$ on a smooth closed manifold $X$ must have at least $\cat(X,\xi)$
geometrically distinct zeros (where $\xi=[\omega]\in H^1(X;\R)$ denotes the cohomology
class of $\omega$) assuming that $\omega$ admits a gradient-like vector field with no
homoclinic cycles.

Viewed differently, the main result of the section claims that
any gradient-like vector field of a closed 1-form $\omega$ has a homoclinic cycle if the number of zeros of $\omega$
is less than $\cat(M,\xi)$.

\begin{definition}
Let $M$ be a smooth closed manifold and let $\omega$ be a smooth closed 1-form on $M$ with the set of zeros
$$Y_\omega = \{p\in M; \omega_p=0\}.$$ We say that a smooth vector field $v$ on $M$ is {\it a gradient-like vector field} for $\omega$ if
$\omega$ is a Lyapunov 1-form for the pair $(\Phi, Y_\omega)$ where $\Phi$ is the flow on $M$ generated by the field $-v$, see Definition \ref{def11}.
\end{definition}

Note that this implies that
\begin{eqnarray}\label{gradientlike}
\omega(v)>0\end{eqnarray} on $M-Y_\omega$ and $Y_\omega$ is invariant with respect to the flow generated by $v$.

In the special case when the set of zeros $Y_\omega$ is finite, a vector field $v$ is a gradient-like vector field for $\omega$ iff the sets of zeros of $\omega$ and of $v$ coincide and
the inequality (\ref{gradientlike}) holds on $M-Y_\omega$.

\begin{theorem} [Farber \cite{farber}]\label{lsmain1} Let $\omega$ be a smooth closed 1-form on a closed smooth manifold $M$.
If $\omega$ admits a gradient-like vector field $v$ with no homoclinic cycles then $\omega$ has at least $\cat(M, [\omega])$ geometrically distinct zeros.
\end{theorem}

Here $[\omega]\in H^1(M;\R)$ denotes the cohomology class of $\omega$.

Next we give a reformulation of the above theorem:

\begin{theorem}\label{lsmain2}
If the number of zeros of a smooth closed 1-form $\omega$ is less than $\cat(M, [\omega])$,
 then any
 gradient-like vector field $v$ for $\omega$ has a homoclinic cycle.
\end{theorem}

Note that the definition of gradient-like vector field in this paper is slightly more general than the one used in \cite{farber}, \cite{farbook}.

A slightly more informative statement was proven in \cite{farbook}:

\begin{theorem}[Farber \cite{farbook}, Theorem 10.16]\label{lsnnn}
Let $\omega$ be a smooth closed 1-form on a closed manifold $M$ having less than $\cat(M, \xi)$ zeros, where $\xi=[\omega]\in H^1(M;\R)$.
Then there exists an integer $N>0$ such that any gradient-like vector field $v$ for $\omega$ has a homoclinic cycle $\gamma_1, \dots, \gamma_n$ satisfying
\begin{eqnarray}
\sum_{i=1}^n \, \int\limits_{\gamma_i}\omega \, \leq \, N.
\end{eqnarray}
\end{theorem}

One may combine these results with Theorem \ref{colliding} which guarantees existence of
closed 1-forms with at most one zero in any nonzero integral cohomology class.
This shows
that there always exist homoclinic cycles which cannot be destroyed while perturbing the
gradient-like vector field!

This {\it \lq\lq focusing effect\rq\rq} starts when the
number of zeros of a closed 1-form becomes less than the number $\cat(M,\xi)$.
It is a new phenomenon, not occurring in Novikov theory. Indeed, if we assume
that the zeros of $\omega$ are all Morse type, then (by the Kupka-Smale Theorem
\cite{Sm3}) it is always possible to find a gradient-like vector field $v$ for
$\omega$ such that any integral trajectory connecting two zeros comes out of a
zero with higher Morse index and goes into a zero with lower Morse index; such
a vector field $v$ has no homoclinic cycles.

As an illustration we will state here the following result which is a consequence of Theorem \ref{lsmain2} and depends on a calculation of $\cat(X, \xi)$ in the case when $X$ is a product of surfaces, see Theorem \ref{surface} below.

\begin{theorem}
Let $M^{2k}$ denote the product $\Sigma_1\times \Sigma_2\times \dots\times \Sigma_k$ where each $\Sigma_i$ is a closed orientable surface of genus $g_i>1$. For a cohomology class $\xi=[\omega]\in H^1(X;\R)$ we denote by $r(\xi)$ the number of indices $i\in \{1, \dots, k\}$ such that
$\xi|\Sigma_i=0$. Let $\omega$ be a smooth closed 1-form on $M$ lying in the cohomology class $\xi$ and having at most $2r(\xi)$ zeros.
Then any gradient-like vector field for $\omega$ has a homoclinic cycle.
\end{theorem}

Theorem \ref{lsmain2} was generalized by J. Latschev \cite{Lat} who studied a more general case of closed 1-forms having infinitely many zeros. The result of Latschev gives a \lq\lq comparison\rq\rq\, between the topology of the set of zeros of $\omega$ and the invariant $\cat(M, [\omega])$.
We prove below in \S \ref{jankoproof} of this paper the following strengthening of the main Theorem of \cite{Lat} and of Theorem \ref{lsnnn}:

\begin{theorem} \label{jankothm1} Let $\omega$ be a smooth closed 1-form on a smooth closed manifold $M$. Assume that the set of zeros $Y=\{p\in M; \omega_p=0\}$ admits a neighbourhood $Y\subset U\subset M$ such that $\omega|U$ is exact\footnote{This condition is automatically satisfied if either the cohomology class $[\omega]$ is integral or $Y$ is an ENR, see Lammas \ref{lmintegral}, \ref{lmenr}.} and $Y$ has finitely many connected components $Y_1, \dots, Y_k$ which satisfy
\begin{eqnarray}\label{jankoineq1}
 \sum_{i=1}^k \cat_M(Y_i) < \cat(M, [\omega])
\end{eqnarray}
where $[\omega]\in H^1(M;\R)$ denotes the cohomology class of $\omega$. Then there exists an integer $N>0$ such that
for any gradient-like vector field $v$ for $\omega$ there exists
a finite chain of orbits $$\gamma_1, \dots, \gamma_\ell, \gamma_{\ell+1}=\gamma_1$$ lying in $M-Y$
(with $1\le \ell \le k$) and labels $1\le i(j)\le k$ such that
\begin{eqnarray}\label{jankocycle1}
\bigcap_{t\in \R} \overline{\gamma_j([t, +\infty))} \subset Y_{i(j)} \quad \mbox{and} \quad \bigcap_{t\in \R} \overline{\gamma_{j+1}((-\infty, t])}\subset Y_{i(j)},
\end{eqnarray}
for all $ 1\le j\le \ell$ and additionally one has
\begin{eqnarray}\label{lessn}
\sum_{i=1}^\ell\,  \int\limits_{\gamma_i}\omega \, \leq\,  N.\label{lsen}
\end{eqnarray}
\end{theorem}

In other words, if the set of zeros of a closed 1-form $\omega$ is  \lq\lq small\rq\rq\, in some sense, i.e. it satisfies the inequality
(\ref{jankoineq1}), then any gradient-like vector field $v$ for $\omega$ has a generalized homoclinic cycle
making at most $N$ full twists with respect to $\omega$.

Note that conditions (\ref{jankocycle1}) describe a generalization of the notion of homoclinic cycle. Indeed, (\ref{jankocycle1}) intuitively means that the $j$-th
trajectory $\gamma_j$ \lq\lq ends\rq\rq\, at the same connected component
$Y_{i(j)}$ of $Y$ from which the next $(i+1)$-th trajectory $\gamma_{j+1}$ \lq\lq originates\rq\rq. Such sequence of trajectories $\gamma_1, \dots, \gamma_\ell$ will be called {\it a generalized homoclinic cycle}.

Theorem \ref{jankothm1} reduces to Theorem \ref{lsmain1} under an additional assumption that the set $Y$ is finite.
 \begin{figure}[h]
\begin{center}
\resizebox{7cm}{5cm} {\includegraphics[6,260][584,753]{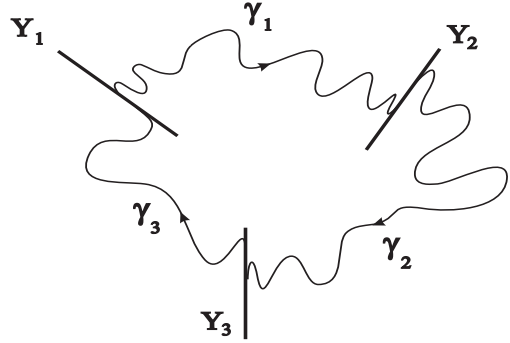}}
\end{center}
\caption {Generalized homoclinic cycle.}
\end{figure}

The proof of Theorem \ref{jankothm1} is postponed to section \S \ref{jankoproof}; in the following section we prove a theorem which is used in \S \ref{jankoproof}.

\section{Existence of flow-convex neighborhoods}

In this section we study some auxiliary problem which will be used in the proof
of Theorems \ref{jankothm1} and Theorem \ref{tkthm} and is also of independent interest.
The results of this section are typical for the Conley theory of isolated invariant sets.
We follow here the paper \cite{farkap}.

Consider a smooth vector field $v$ on a closed smooth manifold $M$. Let $\Phi:
M\times \R\to M$ be the flow of $v$. We will write $\Phi(x,t)$ as $x\cdot t$,
where $x\in M$ and $t\in \R$. The symbol $R$ denotes the chain recurrent set of
$\Phi$.

\begin{theorem}[Farber - Kappeler \cite{farkap}] \label{convexthm} Let $Z$ be a connected component of $R$ which is isolated in $R$, i.e. such
that there exists a neighborhood $Z\subset U$ with $U\cap R = Z$. Let $W\subset
M$ be a neighborhood of $Z$. Then there exists an open neighborhood $B$ of $Z$,
contained in $W$, with the following two properties:

{\rm {(A)}} For any $x\in M$, the open set
$$J_x= \{t; x\cdot t\in B\}\subset \R$$
is convex (i.e. it is either empty or an interval);

{\rm {(B)}} Let $A$ be the set of points $x\in M$ such that the interval $J_x$
is nonempty and bounded below. Then the function
$$A\to \R, \quad x\mapsto \inf J_x,$$
is continuous.
\end{theorem}

A neighborhood $B$ of $Z$ having properties (A) and (B) is called {\it convex}
with respect to the flow $\Phi$.

Theorem \ref{convexthm} could be compared to Lemma B.1 in Appendix B of \cite{farber}.

By Conley's theorem \cite{Cnl, Cnl1}, there exists a smooth Lyapunov function
$L: M\to \R$ for the flow; see also \cite{FKLZ1}, Proposition 2, where the
proof of the smooth version of Conley's theorem is given. The Lyapunov function
$L$ satisfies: $V(L)<0$ on the complement $M-R$ of $R$ and the differential
$dL$ vanishes on $R$.

Fix a point $x\in M-R$ and consider its orbit $x(t)=x\cdot t$. The function
$t\mapsto L(x\cdot t)$ is strictly decreasing. Hence, as $t$ tends to $+\infty$
the limit \begin{eqnarray} \ell(x)= \lim_{t\to \infty}L(x\cdot t)
\label{limit}\end{eqnarray} exists and is finite. If $x\in R$ then the function
$t\mapsto L(x\cdot t)$ is constant.
\begin{figure}[h]
\begin{center}
\resizebox{7cm}{4cm} {\includegraphics[80,474][482,751]{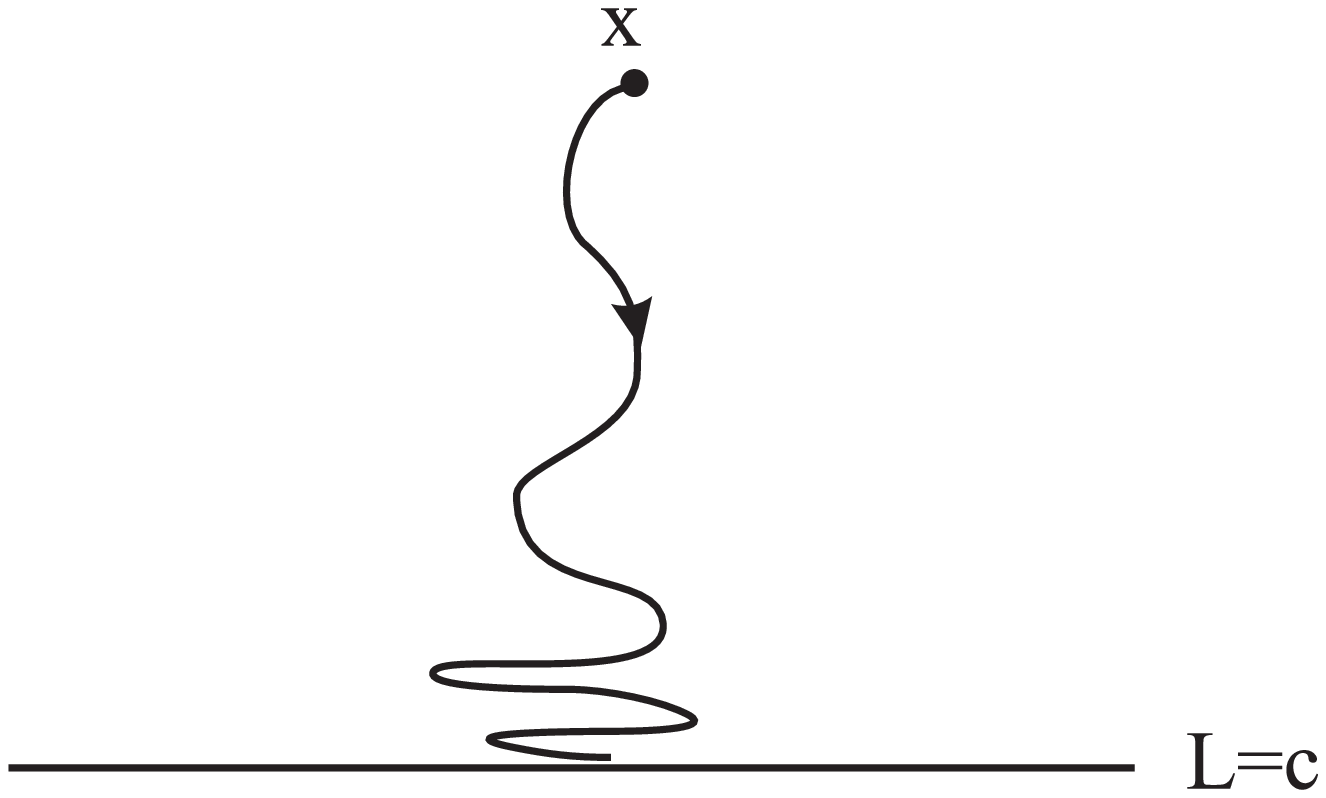}}
\end{center}
\end{figure}

For any $x\in M$ the number $\ell(x)$ is a critical value of $L$. Indeed, the
$\omega$-limit set $\omega(x)$ is contained in the level set $L^{-1}(\ell(x))$
and, on the other hand, $\omega(x)$ is a part of the chain recurrent set $R$
which is the set of critical points of $L$.

\begin{lemma}
The function $\ell: M\to \R$ is upper semi-continuous: if $x_n\to x\in M$ then
\begin{eqnarray}
\ell(x)\geq \lim\sup \ell(x_n).\label{semi}
\end{eqnarray}

\end{lemma}

\begin{proof}
Given any $\epsilon >0$ there exists $T>0$ such that $0< L(x\cdot T) - \ell(x)
<\epsilon$. Since $L$ is continuous and the map $x\mapsto x\cdot T$ is
continuous, there exists a neighborhood $U\subset M$ of $x$ such that
$|L(y\cdot T)-L(x\cdot T)|<\epsilon$ for any $y\in U$. Hence for all $y\in U$
one has
$$\ell(y) \leq L(y\cdot T) <L(x\cdot T)+\epsilon<\ell(x) +2\epsilon.$$
Hence, $\ell(y)\leq \ell(x)$.
\end{proof}

The function $L:M\to \R$ restricted to $Z$ is constant. Indeed, $L(Z)\subset
\R$ must be connected (as the image of a connected set) and it has measure zero
by Sard's theorem. Hence it is a single point. Denote $c=L(Z)$.

Let $Z\subset W$ be an open neighborhood. As $Z$ is supposed to be isolated in
$R$, we may assume (without loss of generality) that $\overline W\cap R = Z$.

Fix $\epsilon >0$ and denote by $A_+(\epsilon)$ the set of all points $x\in M$
with the properties: (a) $L(x) = c+\epsilon$; (b) $\omega(x) \subset Z$. Note
that (b) implies that $\ell(x) = c$.

\begin{lemma}\label{lmcon}
For any sufficiently small $\epsilon >0$ the set $A_+(\epsilon)$ is
contained in the neighborhood $W$. \end{lemma}
\begin{proof}
Assuming the contrary, there exists a convergent sequence $x_n\in M-W$ such
that $L(x_n)\to c$ and $L(x_n)>c$, $\ell(x_n)=c$ and $\omega(x_n)\subset Z$,
where $\omega(x)$ denotes the $\omega$-limit set of the trajectory $x\cdot t$.
If $x_0=\lim x_n\in M$ is the limit of $x_n$ then (using (\ref{semi})) one has
$\ell(x_0)\geq c$. On the other hand, $L(x_0)=c$ and hence $\ell(x_0)=c$. We
obtain that $x_0$ belongs to the set $L^{-1}(c)\cap (R-Z).$ In other words, it
lies in a connected component of $R$ distinct from $Z$.

Fix a Riemannian metric on $M$. Denote by $d>0$ the distance between $Z$ and
$R-Z$. Denote by $K>0$ a constant such that the norm of the vector $V$ is less
than $K$ at every point of $M$. Such $K$ exists since $M$ is compact.

Fix some $\delta>0$ such that $\delta< d/2$. The function $V(L):M\to \R$
restricted to the complement of the $\delta$-neighborhood of $R$ is negative
and moreover can be estimated from above $V(L)\leq -\eta$ for some positive
$\eta=\eta_\delta>0$. Now, we may find a large $n$ such that
$$L(x_n)-c \, <\,
\frac{\eta (d-2\delta)}{K}$$ and $x_n$ lies in the $\delta$-neighborhood of
$R-Z$. The trajectory $x_n\cdot t$ approaches $Z$ for large $t$. The length of
the trajectory is given by
$$\int_{t_1}^{t_2} |\dot x(t)|dt = \int_{t_1}^{t_2} |V(x(t)|dt \, \leq \, K|t_2-t_1|.$$
One concludes that the time $\tau_n$ the trajectory $x_n\cdot t$ for $t>0$
spends in the complement of the $\delta$-neighborhood of $R$ can be estimated
by
$$\tau_n\, \geq \, \frac{d-2 \delta}{K}.$$
Therefore, for large $t>0$ one has
$$L(x_n\cdot t) - L(x_n) = \int_0^t V(L)(x_n\cdot t) dt \leq -\eta\cdot
\tau_n\leq -\frac{\eta(d-2\delta)}{K}$$ These inequalities show (since $L(x_n)$
tends to $c$) that for large $t$ one has $L(x_n\cdot t)<c$ contradicting the
assumption $\ell(x_n)=c$.
\end{proof}

\begin{remark}\label{rm1}{\rm
Define $A_-(\epsilon)$ as the set of all points $x\in M$ with $L(x)=c-\epsilon$
and $\alpha(x)\subset Z$ where $\alpha(x)$ denotes the $\alpha$-limit set of
the trajectory $x\cdot t$. Applying the above arguments to the time reversed
flow one obtains that for any neighborhood $W$ of $Z$ the set $A_-(\epsilon)$
is contained in $W$ for any sufficiently small $\epsilon
>0$.}
\end{remark}

The arguments of the proof could be summarized as follows: the point $x_n$ is
very close to a component $Z'$ of $R$ lying in $L^{-1}(c)$ and distinct from
$Z$. The trajectory starting at $x_n$ cannot approach $Z$ since while it passes
the distance separating $Z$ and $Z'$ it descends with respect to $L$ so that
the point $x_n\cdot t$ slips below the level $c$.

\begin{lemma} For any sufficiently small $\epsilon >0$
the set $A_+(\epsilon)$ is closed.
\end{lemma}
\begin{proof} Let $W$ be a neighborhood of $Z$ with $\overline W\cap R=Z$ and let
$\epsilon>0$ be such that $A_+(\epsilon)\subset W$. Assume that a sequence
$x_n\in A_+(\epsilon)$ converges to a point $x_0$. Then, for any $t>0$ the
point $x_n\cdot t$ lies in $W$. Hence we obtain that for any $t>0$ the point
$x_0\cdot t$ lies in $\overline W$. It follows that $\omega(x_0)\subset
\overline W\cap R=Z$. Hence, $x_0$ is an element of $A_+(\epsilon)$.
\end{proof}
\begin{lemma}\label{lm22} Let $W$ be an open neighborhood of $Z$ such that $\overline W\cap
R=Z$. Let $\epsilon>0$ be such that $A_+(\epsilon')$ is contained in $W$ for
any $0<\epsilon'\leq \epsilon$. Then there exists an open neighborhood
$U_+(\epsilon)$ of the set $A_+(\epsilon)$ in the level set
$L^{-1}(c+\epsilon)$ and an open neighborhood $U_0$ of $Z$ in $L^{-1}(c)$ with
the following properties:
\begin{enumerate}
\item If $x\in U_+(\epsilon) - A_+(\epsilon)$ then for some $\tau_x>0$ the
point $x\cdot \tau_x$ lies in $U_0-Z$. \item The mapping $x\mapsto x\cdot
\tau_x$ is a homeomorphism of $U_+(\epsilon) - A_+(\epsilon) \to U_0-Z$. \item
For any $0\leq t\leq \tau_x$ the point $x\cdot t$ lies in $W$. \item The number
$\tau_x$ depends continuously on $x\in U_+(\epsilon)-A_+(\epsilon)$ and it
tends to $+\infty$ as $x\in U_+(\epsilon)-A_+(\epsilon)$ approaches
$A_+(\epsilon)$.
\end{enumerate}
\end{lemma}
\begin{figure}
\begin{center}
\resizebox{6cm}{4.8cm} {\includegraphics[107,358][542,715]{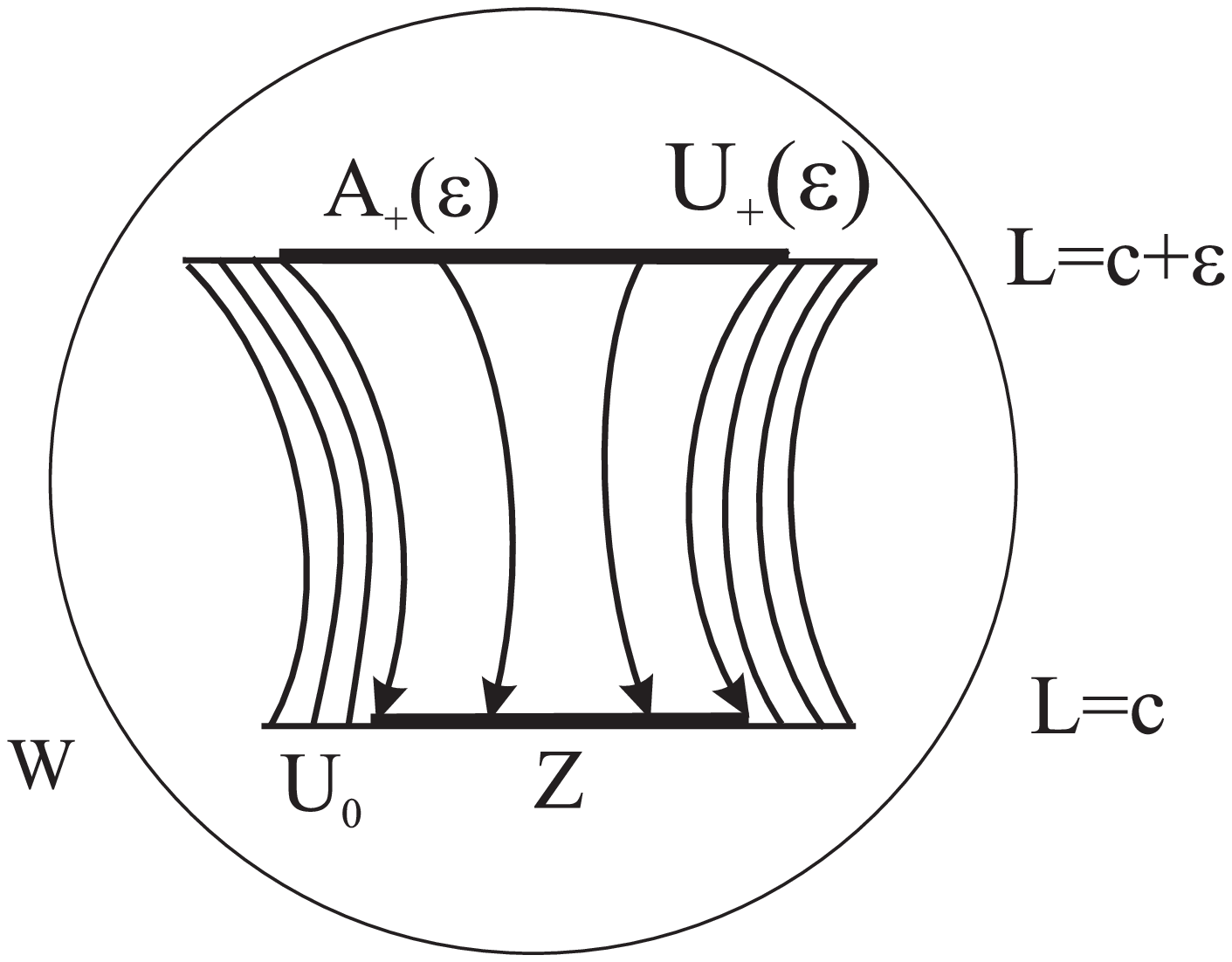}}
\end{center}
\end{figure}
\begin{proof}
If $L(y)=c+\epsilon$ and $y$ is close enough to the set $A_+(\epsilon)$ then
for $t>0$ the orbit $y\cdot t$ does not leave $W$ before it leaves the set
$L\geq c$. Indeed, if this claim is false then there exists a sequence of
points $y_n\in M$ and a sequence of numbers $t_n>0$ such that
$L(y_n)=c+\epsilon$, $y_n\to x\in A_+(\epsilon)$, $L(y_n\cdot t_n)\geq c$, and
$y_n\cdot t_n\notin W$. Fix a Riemannian metric on $M$. Let $K>0$ be such that
the length of the vector $V(x)$ is less or equal than $K$ at every point $x\in
M$. Choose a small neighborhood $G$ of $Z$ such that its closure is contained
in $W$. Let $d>0$ be the distance between $\overline G$ and $M-W$. The function
$V(L)$ satisfies in $W-G$ the inequality $V(L)\leq -\eta$ for some $\eta>0$.
Now, let $G'\subset G$ be a neighborhood of $Z$ such that
$$L(G')\subset (-\infty, c+\eta\cdot\frac{d}{K}).$$
As $x$ belongs to $A_+(\epsilon)$ there exists $t'>0$ with $x\cdot t'\in G'$.
Then for $n$ large enough $y_n\cdot t'\in G'$. Arguing as in the proof of Lemma
\ref{lmcon}, one obtains that the point $y_n\cdot t$ cannot reach $M-W$ before
it leaves the domain $L\geq c$.

Similarly we observe: if $L(y)=c$ where $y\in W-Z$ is sufficiently close to $Z$
then the trajectory $y\cdot t$ for some $t<0$ reaches the level $L=c+\epsilon$
without leaving $W$. Indeed, if this statement is false one finds a sequence
$y_n\in W-Z$, $L(y_n)=c$ such that $y_n\to x\in Z$ and for some $t_n<0$,
$y_n\cdot t_n\in
\partial W$, $L(y_n\cdot t_n)\leq c+\epsilon$. We may assume
that the sequence $y_n\cdot t_n$ converges to a point $z\in
\partial W$. If $\ell(z)>c$ the forward trajectory $z\cdot t$ where $t>0$ crosses the level
surface $L=c$ at a point which is not in $Z$ and the trajectory $y_n\cdot t$
must cross this level at a nearby point as $n\to \infty$, a contradiction.
Hence we have to consider only the case $\ell(z)=c$ which leaves two
possibilities: either $\omega(z)$ is contained in a connected component of $R$
distinct from $Z$ which also lies on the level surface $L=c$, or
$\omega(z)\subset Z$. The first possibility leads to a contradiction, arguing
as in the proof of Lemma \ref{lmcon}. The second possibility means that $z$
belongs to $A_+(\epsilon')$ where $0<\epsilon'\leq \epsilon$. This is a
contradiction since we assume that $A_+(\epsilon')$ is contained in the open
set $W$.

Let $U_0$ be the union of the set $Z$ and the set of all points $y\in W\cap
L^{-1}(c)$ such that the trajectory $y\cdot t$ for $t<0$ reaches the set $W\cap
L^{-1}(c+\epsilon)$ without leaving $W$.

Let $U_+(\epsilon)$ be the union of $A_+(\epsilon)$ and the set of all $y\in
W\cap L^{-1}(c+\epsilon)$ such that the trajectory $y\cdot t$ for $t>0$ reaches
the level surface $L^{-1}(c)$ without leaving $W$.

$U_0$ is an open neighborhood of $Z$ in $L^{-1}(c)$ and $U_+(\epsilon)$ is an
open neighborhood of $A_+(\epsilon)$ in $L^{-1}(c+\epsilon)$ as was shown
above. It is then easy to see that the statements of Lemma \ref{lm22} hold for
$U_0$ and $U_+(\epsilon)$. Let us show, for example, that if $x_n\in
U_+(\epsilon)-A_+(\epsilon)$ and $x_n\to x\in A_+(\epsilon)$ then $\tau_{x_n}$
tends to $+\infty$. If not one may pass to a subsequence such that the sequence
$\tau_{x_n}$ has a finite limit $\tau$. Then $L(x\cdot \tau) =\lim L(x_n\cdot
\tau_{x_n})= c$. Thus the trajectory $x\cdot t$ arrives in finite time at $Z$
which is impossible since $Z$ is flow-invariant and $x\notin Z$.
\end{proof}
Here is an analog of Lemma \ref{lm22}:

\begin{lemma}\label{lm23} Let $W$ be an open neighborhood of $Z$ such that $\overline W\cap
R=Z$. Let $\epsilon>0$ be such that $A_-(\epsilon')$ is contained in $W$ for
any $0<\epsilon'\leq \epsilon$. Then there exists an open neighborhood
$U_-(\epsilon)$ of the set $A_-(\epsilon)$ in the level set
$L^{-1}(c-\epsilon)$ and an open neighborhood $U_0$ of $Z$ in $L^{-1}(c)$ with
the following properties:
\begin{enumerate}
\item If $x\in U_-(\epsilon) - A_-(\epsilon)$ then for some $T_x<0$ the point
$x\cdot T_x$ lies in $U_0-Z$. \item The mapping $x\mapsto x\cdot T_x$ is a
homeomorphism
$$U_-(\epsilon) - A_-(\epsilon) \to U_0-Z.$$ \item For any $T_x<
t<0$ the point $x\cdot t$ lies in $W$. \item The number $T_x$ depends
continuously on $x\in U_-(\epsilon)-A_-(\epsilon)$ and it tends to $-\infty$ as
$x\in U_-(\epsilon)-A_-(\epsilon)$ approaches $A_-(\epsilon)$.
\end{enumerate}
\end{lemma}
\begin{proof} It is similar to the proof of Lemma \ref{lm22}.
\end{proof}
\begin{proof}[Proof of Theorem \ref{convexthm}]
Let $W\subset M$ be an open neighborhood of $Z$. We may assume without loss of
generality that $\overline W\cap R=Z$. By Lemma \ref{lmcon} and Remark
\ref{rm1} we may find $\epsilon>0$ such that for any $0<\epsilon'\leq \epsilon$
the sets $A_+(\epsilon')$ and $A_-(\epsilon')$ are contained in $W$.

Let $C$ denote the set of all points $x\in W\cap L^{-1}(c-\epsilon,
c+\epsilon)$ such that there exist numbers $t_x<0<\tau_x$ with
\begin{eqnarray}
L(x\cdot t_x)=c+\epsilon, \quad L(x\cdot \tau_x)=c-\epsilon
\label{neigh}\end{eqnarray} and $x\cdot t$ is contained in $W$ for any $t\in
(t_x, \tau_x)$. Observe that if $x\in C$ then $x\cdot (t_x, \tau_x)$ is
contained in $C$.

Define the set
$$B=C\cup Z\cup \bigcup_{0\, <\,\epsilon'\,<\, \epsilon}A_\pm(\epsilon').$$
We are going to show that it satisfies the requirements of Theorem \ref{convexthm}. This set
is open. Indeed, the terms in (\ref{neigh}) are pairwise disjoint. The set $C$
is clearly open. Any point of $A_\pm(\epsilon')$ (where $0<\epsilon'<\epsilon$)
has a neighborhood contained entirely in $B$ (by Lemmas \ref{lm22} and
\ref{lm23}). Similarly, any sufficiently small neighborhood $G$ of a point
$x\in Z$ is contained in $B$: if $y\in G$ and $L(y)>c$ then the trajectory
$y\cdot t$ for $t>0$ either approaches $Z$ (in that case $y$ lies in
$A_+(\epsilon')$  for some $0<\epsilon'<\epsilon$) or it hits the level surface
$L=c$. In the second case the intersection point is very close to $Z$ and hence
by Lemmas \ref{lm22} and \ref{lm23} the trajectory continues all the way till
it reaches the level $L=c-\epsilon$ without leaving $W$.

Given $x\in M$ consider the set $J_x=\{t;x\cdot t\in B\}$. Let us assume that
$L(x)>c$. Then there exist the following possibilities:
\begin{enumerate}
\item $\ell(x)>c$; then $J_x=\emptyset$. \item $\omega(x)\subset J_x)$; then
$J_x$ is a half infinite interval $(a, +\infty)$. \item $\ell(x)=c$ and
$\omega(x)\cap Z=\emptyset$; then $J_x=\emptyset$.\item $\ell(x)<c$; then $J_x$
is either the empty set or a finite interval $(a,b)$.
\end{enumerate}

In the case $L(x)\leq c$ the arguments are similar.

The continuity of $x\mapsto \inf J_x$ follows from the Implicit Function
Theorem: the number $\inf J_x=t$ is the solution of the equation $L(x\cdot
t)=c+\epsilon$ and the partial derivative of $L(x\cdot t)$ with respect to $t$
is strictly negative.
\end{proof}

\section{Proof of Theorem \ref{jankothm1}}\label{jankoproof}

The arguments are similar to those use in the proof of Theorem 4.1 of \cite{farber}.

Suppose that Theorem \ref{jankothm1} is false, i.e. for any $N>0$ there exists a gradient-like vector field $v=v_N$ for $\omega$ having no generalized homoclinic cycles of integral trajectories $\gamma_1, \dots, \gamma_\ell$ satisfying (\ref{jankocycle1}) and (\ref{lessn}). Our goal is to show that then
\begin{eqnarray}\label{ineq3}
\cat(M, [\omega]) \leq \sum_{i=1}^k \cat_M(Y_i),\end{eqnarray}
i.e. the inequality (\ref{jankoineq1}) is violated. Without loss of generality we may assume that the number of connected components $k$ is finite since
if $k=\infty$ the inequality (\ref{ineq3}) is obviously true, as $\cat_M(Y_i)\geq 1$ for any $i$.

Fix a Riemannian metric on $M$.
We will denote the flow $M\times \R\to M$ generated by $-v$ by $t\mapsto m\cdot t$, where $m\in M$ and $t\in \R$. Since $\omega(v)>0$ on $M-Y$, the integral $\int_p^{p\cdot t}\omega\leq 0$ is  non-positive and non-increasing for $t>0$.

We start by proving the following Lemma:

\begin{lemma}\label{lmnbh} Assume that for an integer $N>0$ there exists a gradient-like vector field $v$ for $\omega$ having no generalized homoclinic cycles of integral trajectories $\gamma_1, \dots, \gamma_\ell$ satisfying (\ref{jankocycle1}) and (\ref{lessn}). Let $Y_i\subset U_i\subset M$ be an open neighbourhood of $Y_i$, where $i=1, \dots, k$. Then there exists a family of smaller closed neighbourhoods $V_i\subset U_i$ of $Y_i$ satisfying the following properties:
\begin{enumerate}
\item[(a)] Each set $V_i$ lifts to the covering space $\pi: \tilde M\to M$ corresponding to the kernel of the homomorphism of periods $\pi_1(M)\to \R$ determined by the cohomology class $\xi=[\omega]\in H^1(M;\R)$.

    \item[(b)] The lift of $V_i$ into $\tilde M$ is flow convex (i.e. it satisfies properties (A) and (B) of Theorem \ref{convexthm}) with respect to the natural lift $\tilde v$ of the flow $v$ into $\tilde M$.

    \item[(c)] Let $\partial_-V_i$ denote the exit set of $V_i$, i.e. the set of all points $p\in V_i$ such that for all sufficiently small $\tau>0$ one has $p\cdot \tau\notin V_i$. Then there exist no $p\in \partial_-V_i$ and $t>0$ such that $p\cdot t\in \Int V_i$ and $\int_p^{p\cdot t} \omega \geq -N$.
\end{enumerate}
\end{lemma}
The intuitive meaning of property (c) is that it is impossible for a trajectory to leave $V_i$ and then revisit $V_i$ in time $t>0$ without making the quantity
$\int_p^{p\cdot t} \omega$  smaller than $-N$.
\begin{figure}[h]
\begin{center}
\resizebox{10cm}{4.5cm} {\includegraphics[20,356][569,664]{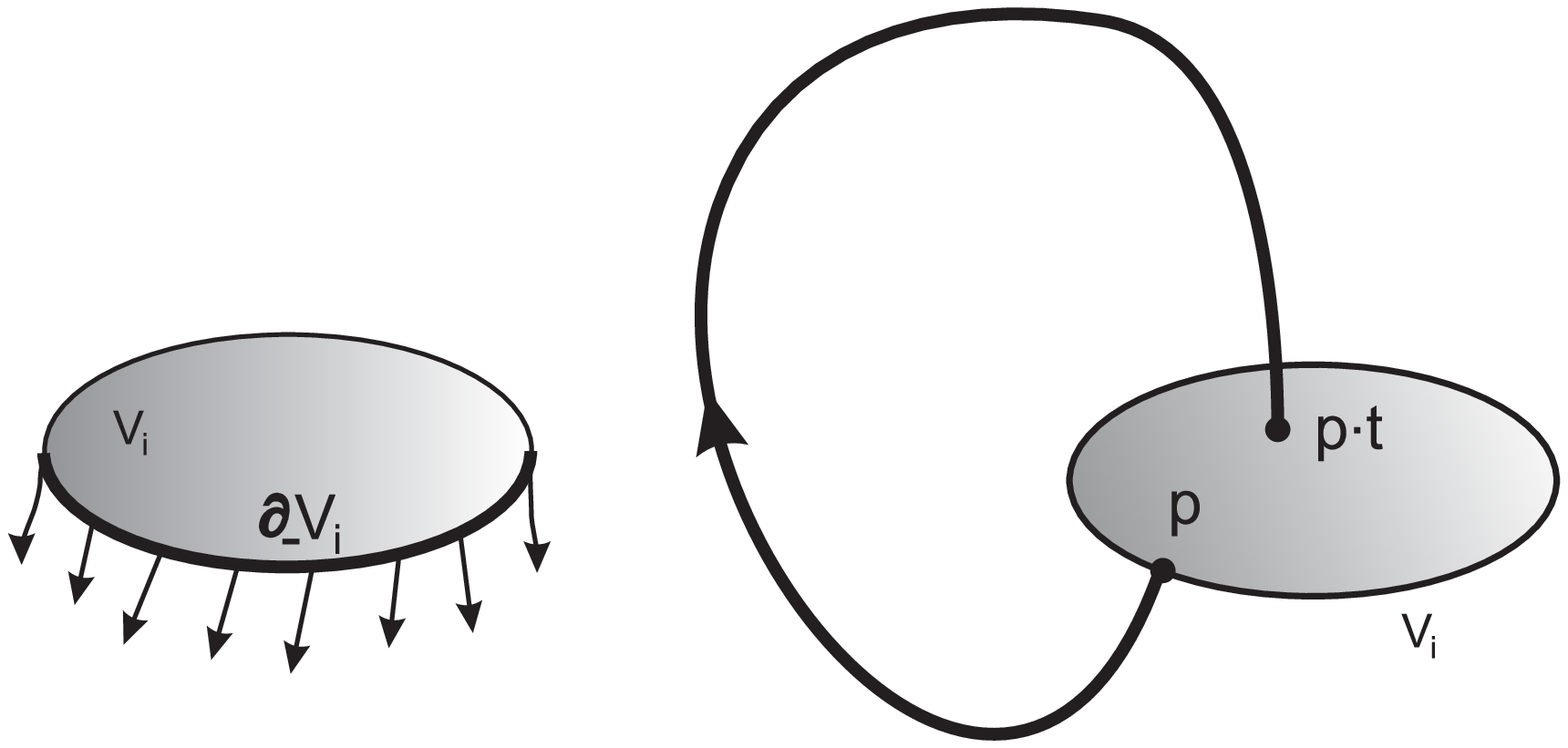}}
\end{center}
\end{figure}
\begin{proof} Without loss of generality we may assume that the initial neighbourhoods $U_i$ satisfy properties (a) and (b) of Lemma \ref{lmnbh}.

Let $V_i\subset U_i$ be any closed neighborhood of $Y_i$ satisfying (b); then (a) is automatically satisfied. Denote by $V$ and $U$ the unions $V=V_1\cup \dots \cup V_k$, $U=U_1\cup \dots\cup U_k$ and $\partial_-V = \partial_-V_1\cup \dots\cup \partial_-V_k$.  We claim:
\begin{enumerate}
\item[(i)] {\it There exists $a>0$ such that for any $p\in \partial_-V$ and $t>0$
with $p\cdot t\in V$ one has $t\geq a$.}
\end{enumerate}
Indeed, as follows from the
gradient convexity of the sets $U_i$, any
trajectory $p\cdot t$ starting at $p\in \partial_-V$ at $t=0$ leaves $U$ before it can
re-enter $V$. Hence, we may take $a=\min\{l_i v_i^{-1}; i=1, \dots, k\}$,
where $l_i>0$ denotes the distance between $V_i$ and $M-{\rm Int}\, U_i$, and
$v_i =\max |v(x)|$ for $x\in U_i -{\rm Int}\, V_i$.

Note that if we shrink the sets $V_i$, the number $a>0$ may only increase,
assuming that $V_i$ are sufficiently small.

\begin{enumerate}
\item[(ii)] {\it There exists $b>0$ such that for any $p\in \partial_-V$ and $t>0$
with $p\cdot t\in V$ one has}
$$\int_p^{p\cdot t}\omega < -b.$$
\end{enumerate}
Indeed, we may take
$b=a\cdot\min\{\omega(v(x));x\in U-V\}.$

From (i) and (ii) it follows that:
\begin{enumerate}
\item[(iii)] {\it There exists an integer $K>0$ such that for any $p\in \partial_-V$
the set
$$I(p) =\{t>0; \, p\cdot t\in V \, \mbox{and}\quad \int_p^{p\cdot t}\omega \geq -N\}$$
is a union of at most $K$ disjoint intervals, i.e.
$$I(p) = \bigcup_{s=1}^{k(p)}[a_s(p), b_s(p)],\quad k(p)\leq K$$
and
$$0<a_s(p)\leq b_s(p) <a_{s+1}(p).$$
}
\end{enumerate}
Indeed, the number of gaps between the intervals $[a_s(p), b_s(p)]$ is at most $[N/b]$ as follows from (ii).

For a point $q\in M$ we denote by
\begin{eqnarray}
\A(q) =\bigcap_{t\in \R} \overline{q\cdot (-\infty, t]}, \quad \ZZ(q) =\bigcap_{t\in \R} \overline{q\cdot [t, \infty)},
\end{eqnarray}
{\it the backward and forward limit sets} of the trajectory passing through $q$. These sets are nonempty, compact, connected and flow-invariant.

Now, suppose that we can never achieve (c) by shrinking the initially chosen neighbourhoods $V_i\supset Y_i$.
We want to show that then there exists a sequence of points $q_1, \dots, q_\ell\in M$  and an injection $j: \{1, \dots, \ell\}\to \{1, \dots,k\}$
such that
$$\A(q_i)\subset Y_{j(i)},\quad \ZZ(q_i)\subset Y_{j(i+1)}$$
where $i=1,\dots, \ell$ and $j(\ell+1)=j(1)$ and additionally the following inequality is satisfied
\begin{eqnarray}\label{sum3}\sum_{i=1}^\ell \int_{\gamma_i}\omega \ge -N\end{eqnarray}
where the curve $\gamma_i: \R\to M$ is given by $\gamma_i(t)=q_i\cdot t$ for $t\in \R$ which would contradict our assumptions.

There exists an infinite sequence of points
$p_{n}\in \partial_- V$, where $n=1, 2, \dots$, and two sequences of real
numbers $t_{n}>0$ and $s_{n}<0$ such that

\begin{enumerate}{\it
\item[(1)] the set $p_{n}\cdot [s_{n},0]$ is contained in a fixed connected component $V_i\subset V$ and
the distance $d(p_{n}\cdot s_{n}, Y_i)$ tends to $0$ as $n\to \infty$;

\item[(2)] $d(p_{n}\cdot t_{n}, Y_i)$ converges to $0$ as $n\to \infty$;

\item[(3)] one has
\begin{eqnarray}\label{sum4}
\int\limits_{p_{n}}^{p_{n}\cdot t_{n}}\omega \geq -N.\end{eqnarray}}
\end{enumerate}

Passing to a subsequence, we may assume that $p_{n}$ converges to a point
$p\in \partial_- V_i$ and the sequences $s_{n}$ and $t_{n}$ have finite
or infinite limits, which we denote by $s$ and $t$ respectively.
Because of property (iii), we may also assume that the number of intervals $\kappa(p_n)=\kappa$ is independent of $n$ and the sequences $a_s(p_{n})$ and
$b_s(p_{n})$ have finite or infinite limits denoted by $a_s$ and $b_s$ respectively for $s=1, \dots, \kappa$. Moreover, for the same reasons, we may assume that as $n\to \infty$ the sequences
$$p_n\cdot a_s(p_n)\in V\quad \mbox{and}\quad p_n\cdot b_s(p_n)\in V$$ converge to points $q_s, q'_s\in V$ correspondingly, where $s=1, \dots, \kappa$.

It is easy to see that (1) and (2) imply that $t=+\infty$ and $s =-\infty$.

Note that $\A(p)\subset Y_i$. Indeed, by (1), one has $p\cdot (-\infty, 0]\subset V_i$ and therefore the set $\A(p)$ must be contained in the maximal invariant subsets $Y_i$ of $V_i$.
\vspace{-10pts}
\begin{figure}[h]
\begin{center}
\resizebox{9.5cm}{7cm} {\includegraphics[7,242][585,698]{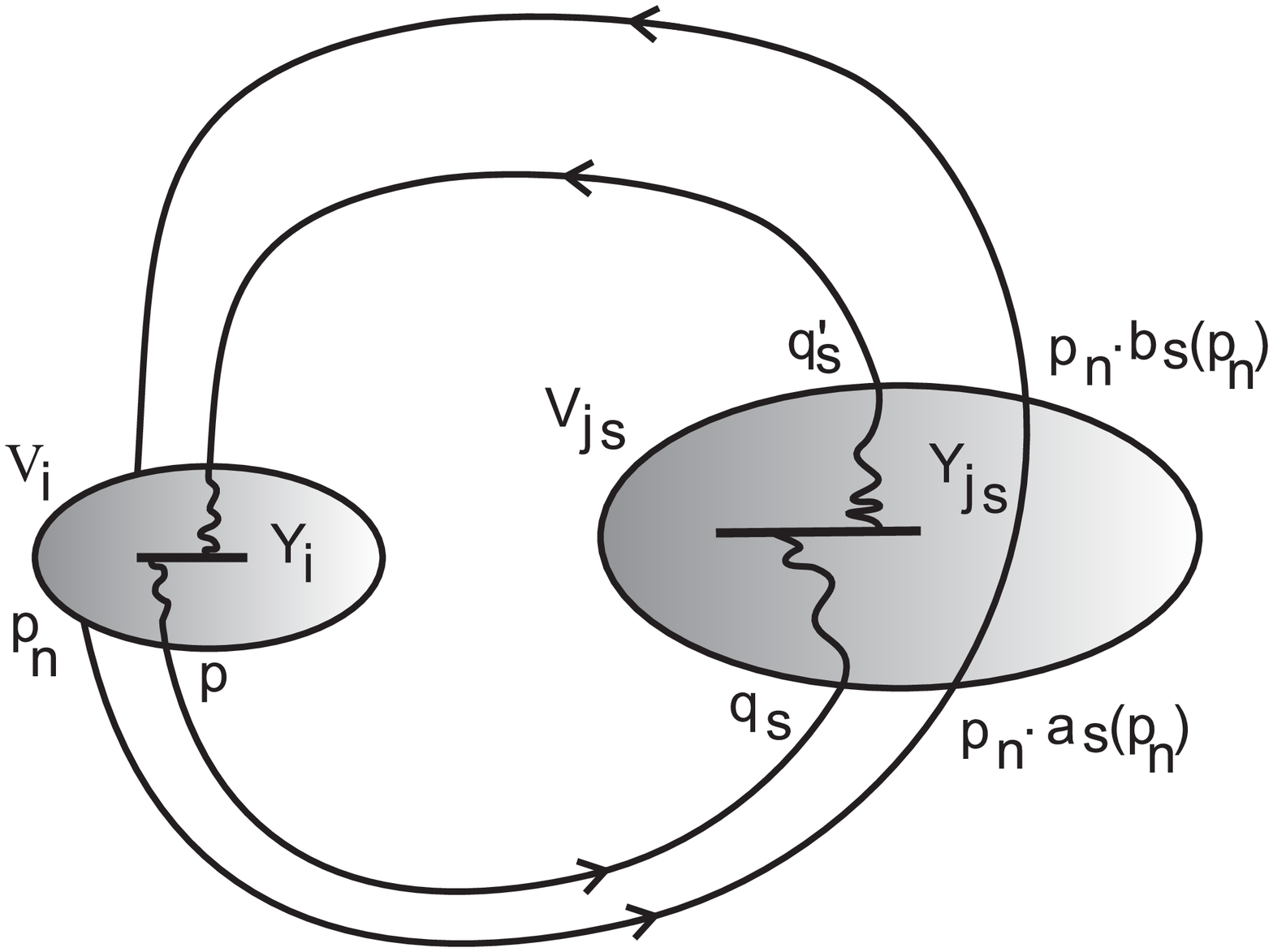}}
\end{center}
\end{figure}
Clearly the points $q_s, q'_s$ lie on the boundary of the same connected component $V_{j_s}$ of $V$ where $s\mapsto j_s$ is a map
$\{1, \dots, \kappa\} \to \{1, \dots, k\}$ and $j_\kappa=i$.
The points $q_s$ and $q'_s$ either lie on the same orbit of the flow or one has
\begin{eqnarray}\label{rel}
\ZZ(q_s) \subset Y_{j_s}, \quad \A(q'_s) \subset Y_{j_s}
\end{eqnarray}
as follows from arguments mentioned above. We may ignore the first case since then we may simply relabel the points and apply the following observations.
The points $q'_s$ and $q_{s+1}$ lie on the same orbit and hence we may rewrite (\ref{rel}) in the form
\begin{eqnarray}\label{rel1}
\ZZ(q_s) \subset Y_{j_s}, \quad \A(q_{s+1}) \subset Y_{j_s}
\end{eqnarray}
where $s=1, \dots,\kappa -1$. For the last point $q_\kappa$ we have
$$\ZZ(q_\kappa)\subset Y_i \quad \mbox{and}\quad \A(q_1)=\A(p)\subset Y_i.$$
Hence we obtaine a closed chain of integral trajectories -- a generalized homoclinic cycle. Inequality (\ref{sum3}) follows from (\ref{sum4}) by passing to the limit and observing that the form $\omega|V_j$ is exact, $\omega|V_j=df_j$ where $f_j\to \R$ is smooth. Moreover the function $f_j$ is constant on $Y_j$. Indeed,
 the image $f_j(Y_j)\subset \R$ is connected and has measure zero by Sard's theorem; hence $f_j(Y_j)$ is a single point. These remarks imply that
 $$\lim_{n\to\infty} \int_{p_n\cdot a_s(p_n)}^{p_n\cdot b_s(p_n)} \omega \, = \, \lim_{T\to \infty} \int_{q_s}^{q_s\cdot T}\omega \, +\,  \lim_{S\to -\infty}\int_{q'_s\cdot S}^{q'_s} \omega$$
 for any $s=1, \dots, \kappa$.

 This completes the proof of Lemma \ref{lmnbh}.
\end{proof}

\begin{proof}[Proof of Theorem \ref{jankothm1}]
Assume that under the condition of Theorem \ref{jankothm1} there exist no generalized homoclinic cycles satisfying (\ref{lessn}). Applying Lemma
\ref{lmnbh} with neighbourhoods $W_i$ such that $\cat_M(W_i)=\cat(Y_i)$ we obtain a system of neighbourhoods $Y_i\subset V_i$ satisfying (a)-(c) of Lemma \ref{lmnbh} and also $\cat_M(V_i)=\cat_M(Y_i)$ where $i=1, \dots, k$.

Define sets $F_1, \dots, F_k\subset M$ as follows.
We say that $p\in F_i$ if for some $t_p>0$ the point $p\cdot t_p$ belongs to the interior of $V_i$,
and
$$\int_p^{p\cdot t_p}\omega >-N.$$
It is clear that $F_i$ is open.

Let us show that $F_i\subset M$ can be deformed into $V_i\subset F_i$.
For any point $p\in M$ let $J_p\subset \R$ denote the set
$J_p =\{t\geq 0; p\cdot t\in V_i\}$. Because of our assumption about the
gradient-convexity of neighborhood $V_i$, the set $J_p$ is a union of disjoint
closed intervals, and some of these intervals may degenerate to a point.
Consider the first interval $[\alpha_p,\beta_p]\subset J_p$. If this interval
degenerates to a point (i.e., the trajectory through $p$ touches $V_j$), then
$p$ does not belong to the set $F_j$, according to our assumption (c); see
above. For the same reason, points of $\partial_-V_j$ do not belong to $F_j$.

Assume now that $p\in F_j$ and $p\notin {\rm Int}\, V_i$.
Then the point $p\cdot t$ lies in the interior of $V_j$ for $\alpha_p<t<\beta_p$. Also, we have
$$\int_p^{p\cdot {\alpha_p}}\omega > -N.$$
The function
$\phi_i: F_i\to \R$, given by
$$
\phi_i(p) = \left\{
\begin{array}{ccl}
0,&\mbox{for}& p\in {\rm Int}\,  V_i,\\ \\
\alpha_p,&\mbox{for}&p\in F_i -{\rm Int}\, V_i,
\end{array}
\right.
$$
is continuous since $(\partial_-V_i)\cap F_i=\emptyset$.
One may now define a homotopy
$$h^i_\tau: F_i \to M, \quad h_\tau(p) = p\cdot (\tau\phi_i(p)), \quad p\in F_i, \quad \tau \in [0,1].$$
Here $h_0$ is the inclusion $F_i\to M$ and $h_1$ maps $F_i$ into the interior of $V_i$.
By Lemma 19.5 from \cite{DNF} we obtain
$$\cat_M(F_i) \leq \cat_M(V_i) = \cat_M(Y_i)$$
and therefore
\begin{eqnarray}\cat_M(\bigcup_{i=1}^k F_i) \le \sum_{i=1}^k \cat_M(Y_i).
\end{eqnarray}

Consider now the complement $A = M-\bigcup_{i=1}^k F_i$. For any point $p\in A$ there exists $t_p>0$ such that
$\int_p^{p\cdot t_p}\omega =-N$. It is clear that $p\mapsto t_p$ is a continuous real valued function on $A$. One may define a homotopy
$$h_\tau: A \to A \quad \mbox{by}\quad h_\tau(p)=p\cdot(\tau t_p), \quad \tau\in [0,1].$$ Then for any $x\in A$ one has
$$\int_x^{h_1(x)}\omega =-N$$
i.e. the closed set $A$ is $N$-movable with respect to $\omega$ and $$\cat_M(M-A)\le \sum_{i=1}^k \cat_M(Y_i).$$

According to Definition \ref{defcat} we obtain that $$\cat(M,[\omega])\le \sum_{i=1}^k \cat_M(Y_i)$$ as claimed.
This completes the proof. \end{proof}

\section{Topology of the chain recurrent set $R_\xi$}

In this section we prove the following result from \cite{farkap}.

\begin{theorem}[Farber - Kappeler]\label{tkthm}
Consider a smooth flow $\Phi:M\times \R\to M$ on a closed smooth manifold $M$.
Let $\xi\in H^1(M;\R)$ be a cohomology class such that the following conditions
are satisfied:

(1) The chain recurrent set $R_\xi$ is isolated in the full chain recurrent set
$R$ of $\Phi$.

(2) The restriction
$$\xi|_{R_\xi}\in \check H^1(R_\xi;\R)$$
(viewed as a \v Cech cohomology class) vanishes.

(3) For any $\Phi$-invariant, positive Borel measure $\mu$ on $M$ with
$\mu(R)>\mu(R_\xi)$, Schwartzman's asymptotic cycle $\A_\mu=\A_\mu(\Phi)\in
H_1(M;\R)$ satisfies
$$
\langle \xi, \A_\mu\rangle < 0.\label{inequa}
$$

Then one has the following inequality
\begin{eqnarray}
\sum\limits_{i=1}^r \cat_M(R_\xi^i)\, \geq\, \cat^1_s(M,\xi),\label{inmain}
\end{eqnarray}
where $R_\xi^1, \dots, R_\xi^r$ denote the connected components\footnote{The
number of connected components $r$ of $R_\xi$ can be infinite. In this case
inequality (\ref{inequa}) is trivially satisfied as its LHS is infinite.} of
the chain recurrent set $R_\xi$.
\end{theorem}

A proof of Theorem \ref{tkthm} is given in section \ref{pr}.

Let us compare Theorem \ref{tkthm} with Theorem 6.6 of \cite{farbe4}. Theorem
\ref{tkthm} allows arbitrary sets $R_\xi$ whereas Theorem 6.6 of \cite{farbe4}
assumes that $R_\xi$ consists of finitely many isolated points. Another
important advantage of Theorem \ref{tkthm} is that it does not require condition
($\ast$), see \cite{farbe4}, page 108; this condition is difficult to check in
concrete examples. On the other hand, since $\cat^1_s(X,\xi)\leq \Cat(X,\xi)$,
the estimate of Theorem 6.6 of \cite{farbe4} is potentially slightly sharper.

Next we state some corollaries of Theorem \ref{tkthm}.

In the special case $\xi=0$ one has $R_\xi=R$ and the assumptions (1), (2), (3)
of Theorem \ref{tkthm} are automatically satisfied. Hence we obtain the
following result which presumably is well known but for which we failed to find
a reference:

\begin{theorem}\label{xi=0}
Consider a smooth flow on a closed smooth manifold $M$. Let $R$ be the chain
recurrent set of the flow. Then
\begin{eqnarray}
\sum_{i=1}^r \cat_M(R^i) \, \geq \, \cat(M).
\end{eqnarray}
Here $R^1, \dots, R^r$ denote the connected components of $R$.
\end{theorem}

In the special case where the chain recurrent set $R$ consists of finitely many
points (the fixed points of the flow) Theorem \ref{xi=0} says that {\it the
number of fixed points is at least the Lusternik - Schnirelman category of
$M$}. This is one of the fundamental results of the classical Lusternik --
Schnirelman theory, see \cite{DNF}.

Consider another special case of Theorem \ref{tkthm} when $R_\xi$ consists of
finitely many points (they are the fixed points of the flow). Condition (2) of
Theorem \ref{tkthm} is automatically satisfied under these assumptions. We
obtain:

\begin{theorem} [Farber - Kappeler]\label{thm22}
Consider a smooth flow on a closed smooth manifold $M$. Let $\xi\in H^1(M;\R)$
be a cohomology class such that the chain recurrent set $R_\xi$ consists of
finitely many points which are isolated in the full chain recurrent set $R$ of
the flow. Assume additionally that condition (3) of Theorem \ref{tkthm} is
satisfied. Then the flow has at least $\cat^1_s(M,\xi)$ fixed points.
\end{theorem}

Here is a reformulation of Theorem \ref{thm22}:

\begin{theorem}[Farber - Kappeler]\label{thm23}
Let $\xi\in H^1(M;\R)$ be a cohomology class. Consider a flow $\Phi$ on $M$
such that the chain recurrent set $R_\xi$ consists of less than
$\cat^1_s(M,\xi)$ fixed points and such that condition (3) of Theorem
\ref{tkthm} is satisfied. Then at least one of the fixed points of the flow is
not isolated in the full chain recurrent set $R$.
\end{theorem}

Here is a typical application of Theorem \ref{thm22}:

\begin{theorem}[Farber - Kappeler] \label{thm24} Let $\Phi$ be a smooth flow on a closed smooth manifold such
that the chain recurrent set $R$ of $\Phi$ is a union of finitely many circles
and isolated points. Let $\xi\in H^1(M;\R)$ be a cohomology class such that
$\langle \xi, z\rangle \leq 0$ for the homology class $z\in H_1(M)$ of any
periodic orbit. Then
\begin{eqnarray}
p_0+p_1+2 p_2\, \geq \, \cat^1_s(M,\xi),
\end{eqnarray}
where $p_0$ denotes the number of fixed points of the flow, $p_1$ denotes the
number of periodic orbits which are null-homotopic in $M$, and $p_2$ denotes
the number of periodic orbits which are homotopically nontrivial but their
homology classes $z\in H_1(M)$ satisfy $\langle \xi, z\rangle =0$.
\end{theorem}

\begin{proof}
In the special case of Theorem \ref{thm24} condition (3) of Theorem \ref{tkthm}
is equivalent  to the requirement that $\langle \xi,z\rangle \leq 0$ for
the homology class $z\in H_1(M)$ of any periodic orbit. Under the assumptions
of Theorem \ref{thm24} the set $R_\xi$ is the union of the fixed points and of
the periodic orbits satisfying $\langle \xi, z\rangle =0$. If $C$ is a periodic
orbit then the number $\cat_M(C)$ equals 1 or 2 depending on whether $C$ is or
is not null-homotopic in $M$. These consideration show that Theorem \ref{thm24}
follows from Theorem \ref{tkthm}.
\end{proof}

As a side remark we point out that under the conditions of Theorem \ref{thm24},
\begin{eqnarray} p_0+p_1+2p_2+2q_2 \, \geq \, \cat(M)\end{eqnarray} as
follows from Theorem \ref{xi=0}; here $q_2$ denotes the number of periodic
orbits satisfying $\langle \xi, z\rangle \not=0$.

\section{Proof of Theorem \ref{tkthm}}\label{pr}

Consider a smooth flow $\Phi: M\times \R\to M$, $\Phi(x,t)=x\cdot t$ and a real
cohomology class $\xi\in H^1(M;\R)$ satisfying the conditions of Theorem
\ref{tkthm}.

If the chain recurrent set $R_\xi$ has infinitely many connected components
then the RHS of inequality (\ref{inmain}) is infinite and so this inequality is
obviously satisfied. Hence without loss of generality we may assume that the
number of connected components of $R_\xi$ is finite.

Let $R_\xi=R_\xi^1\cup R_\xi^2\cup\dots\cup R_\xi^r$ be the connected
components of $R_\xi$. Fix mutually disjoint open neighborhoods $W_i\supset
R_\xi^i$, where $i=1, 2, \dots, r$. We shall assume that $W_i$ is so small that
\begin{eqnarray}
\cat_M(W_i) =\cat_M(R_\xi^i), \quad i=1, \dots, r.\label{stam3}
\end{eqnarray}
Compare \cite{DNF}, Lemmas 19.2 and 19.6.

Applying Theorem \ref{convexthm} we may find an open neighborhood $B_i\subset W_i$
of $R_\xi^i$ which is convex with respect to the flow $\Phi$. Clearly,
\begin{eqnarray}
\cat_M(\overline B_i) =\cat_M(R_\xi^i), \quad i=1, \dots, r,\label{stam1}
\end{eqnarray}
as a consequence of (\ref{stam3}).

Let $U_i$ be an open neighborhood of $R_\xi^i$ such that the closure $\overline
U_i$ is contained in $B_i$. We have the inclusions $R_\xi^i\subset U_i\subset
\overline U_i\subset B_i\subset W_i$.

Let $A\subset M$ denote the set of points $x\in M$ such that for any $t\in \R$
the point $x\cdot t$ does not belong to the union $U=\cup_{j=1}^r U_j$. The set
$A$ is closed and flow-invariant.

Applying Theorem \ref{lyapmain} we find that there exists a smooth Lyapunov
1-form $\omega$ for the pair $(\Phi, R_\xi)$ lying in the cohomology class
$\xi=[\omega]\in H^1(M;\R)$. This means that the function $\omega(V):M\to \R$
is negative on $M-R_\xi$ (where $V$ denotes the vector field on $M$ generating
the flow $\Phi$) and $\omega(V)|_{R_\xi}=0$. Moreover, the restriction of
$\omega$ on some open neighborhood of $R_\xi$ is the differential of a smooth
function (this is equivalent to condition (1) of Theorem \ref{tkthm}).

Let us show that for any integer $N$ the set $A$ is $N$-movable with respect to
both $\omega$ and $-\omega$. Since $M-U$ is compact there exists $\epsilon
>0$ such that $\omega(V)<-\epsilon$ on $M-U$. Then for any $x\in A$ one has
$$\int\limits_x^{x\cdot t}\omega \, <\,  -\epsilon t, \quad \mbox{for}\quad t>0.$$
Hence a continuous homotopy $h_t: A\to M$, $t\in [0,1]$ defined by
$$h_t(x) = x\cdot \left(\frac{Nt}{\epsilon}\right)$$
satisfies the condition of Definition \ref{defmovable}. This shows that $A$ is
$N$-movable with respect to $\omega$. In a similar way one shows that for any
integer $N$ the set $A$ is $N$-movable with respect to $-\omega$.

Now the inequality (\ref{inmain}) follows once one shows that
\begin{eqnarray}
\cat_M(M-A)\, \, \leq\, \, \sum_{j=1}^r \cat_M(R_\xi^j),\label{once}
\end{eqnarray} see Definition \ref{defcats}.

Let $F_i\subset M$ denote the set of points $x\in M$ such that $x\cdot t$
belongs to $B_i$ for some $t\in \R$. Then $J_x^i=\{t\in \R; x\cdot t\in B_i\}$
is a nonempty open interval. For $x\in F_i$ define
$$
\tau_i(x) = \left\{
\begin{array}{lll}
0, & \mbox{if} &0\in J_x^i,\\
\inf J_x^i, & \mbox{if}& J_x^i\subset (0, \infty),\\
\sup J_x^i, & \mbox{if}& J_x^i\subset (-\infty, 0).
\end{array}
\right.
$$
The function $\tau_i: F_i\to \R$ is continuous as follows from Theorem
\ref{convexthm}.

The mapping $x\mapsto x\cdot \tau_i(x)$ is a deformation retraction $F_i\to
\overline B_i$ and hence
$$\cat_M (F_i) \leq \cat_M(\overline
B_i)=\cat_M(R_\xi^i)$$ in view of (\ref{stam1}). Here we use Lemma 19.5 from
\cite{DNF}. Since $M-A$ is contained in the union $\cup_{i=1}^r F_i$ one
obtains
$$\cat_M(M-A) \leq \sum_{i=1}^r \cat_M(F_i)\leq\sum_{i=1}^r \cat_M(R_\xi^i).$$
This proves inequality (\ref{once}) and hence completes the proof of Theorem
\ref{tkthm}.

\section{Cohomological estimates for $\cat(X,\xi)$}

The effectiveness of dynamical applications described in sections \S\S \ref{secfocus} - \ref{pr} depend crucially on the ability to calculate, or to estimate from below, the quantity
$\cat(X, \xi)$ where $X$ is a finite polyhedron and $\xi\in H^1(X;\R)$ is a cohomology class. The situation here is quite similar to the classical Lusternik - Schnirelmann theory where $\cat(X)$ is most often computed by a combination of upper and lower bounds and the most popular lower bound for $\cat(X)$
uses the cohomological cup-length, see \cite{DNF}.

In this section we
discuss results of this type for $\cat(X, \xi)$ following our paper \cite{farsch}; some lower bounds for $\cat(X, \xi)$ were obtained earlier in the papers \cite{farber} and \cite{farbook}.
We give in this section the main definitions, state principal results and illustrate them by several specific examples; however for complete proofs we refer the reader to our original papers \cite{farsch}, \cite{farsce} and \cite{farscm}.

Let $X$ be a finite polyhedron and $\xi \in H^1(X;\R)$. Denote by $\ker (\xi)$
the kernel of the homomorphism $\pi_1(X)\to \R$ given by evaluation on $\xi$.
Then $H=\pi_1(X)/\ker(\xi)$ is a free abelian group of finite rank $r$ where
$r$ denotes the rank of class $\xi$. Consider the cover $p: \tilde X\to X$
corresponding to $\ker(\xi)$. It has $H$ as the group of covering translations.

Let $\V_\xi= (\C^\ast)^r= \Hom(H, \C^\ast)$ denote the variety of all complex
flat line bundles $L$ over $X$ such that the induced flat line bundle $p^\ast
L$ on $\tilde X$ is trivial. If $t_1, \dots, t_r\in H$ is a basis, then the
monodromy of $L\in \mathcal V_\xi$ along $t_i$ is a nonzero complex number
$x_i\in \C^\ast$ and the numbers $x_1, \dots, x_r\in \C^\ast$ form a coordinate
system on $\mathcal V_\xi$. Given a flat line bundle $L\in \V_\xi$ the
monodromy representation of $L$ is the ring homomorphism
\begin{eqnarray}\label{monodromy}{\rm
Mon}_L: \C[H]\to \C\end{eqnarray}
sending each $t_i\in H$ to $x_i\in \C^\ast$.
The dual bundle $L^\ast\in \V_\xi$ is such that $L\otimes L^\ast$ is trivial;
if $x_1, \dots, x_r\in \C^\ast$ are coordinates of $L$ then $x_1^{-1}, \dots,
x_r^{-1}\in \C^\ast$ are coordinates of $L^\ast$.

Any nontrivial element $p\in \C[H]$ lying in the kernel of ${\rm Mon}_L$ can be
viewed as a (Laurent) polynomial equation between the variables $x_1, \dots,
x_r$. Alternatively, we will consider algebraic subvarieties $V\subset \V_\xi$.
Any such $V$ is the set of all solutions of a system of equations of the form
$$q_i(x_1, \dots, x_r, x_1^{-1}, \dots, x_r^{-1})=0, \quad i=1, \dots, m$$
where $q_i$ is a Laurent polynomial with complex coefficients $$p_i\in \C[x_1,
\dots, x_r, x_1^{-1}, \dots, x_r^{-1}].$$ This is equivalent to fixing an ideal
$J\subset \Q[H]$ and studying the set of all flat line bundles $L\in \V_\xi$
such that ${\rm Mon}_L(J)=0.$

\begin{definition}\label{transcendental}{\it
(1) We say that a bundle $L\in \V_\xi$  is {\it $\xi$-algebraic} if the monodromy homomorphism
${\rm Mon}_L: \Z[H]\to \C$ has nontrivial kernel.

(2) We say that $L\in \V_\xi$ is {\it a $\xi$-algebraic integer} if the kernel of
${\rm Mon}_L: \Z[H]\to \C$ contains a nontrivial polynomial $P\in \Z[H]$ with $\xi$-top coefficient $1$ (see below).

(3) We say that $L$ is
{\it $\xi$-transcendental} if ${\rm Mon}_L:
\Z[H]\to \C$ is injective. }
\end{definition}

Any nonzero $P\in \Z[H]$ can be written as $P=\sum_{i=1}^k\alpha_ih_i$ where $\alpha_i\in \Z$, $\alpha_i\not=0$, $h_i\in H$ and $\xi(h_1)<\xi(h_2)<\dots <\xi(h_k).$ The nonzero integer $\alpha_k\in \Z$ is called {\it the $\xi$-top coefficients of $P$.}

If $t_1, \dots, t_r\in H$ is a basis and if $a_i\in \C$ denotes the monodromy
of $L$ along $t_i$, i.e. $a_i={\rm Mon}_L(t_i)$, then $L$ is $\xi$-algebraic iff
there exist a nontrivial Laurent polynomial equation with integral coefficients
$q(t_1, \dots, t_r)$ such that $q(a_1, \dots, a_r)=0$.

There exist countably many nonzero Laurent polynomials $q$ with integral
coefficients and for each such $q$ the set of solutions $q(a_1, \dots, a_r)=0$
is nowhere dense in $\V_\xi$. Since $\V_\xi=(\C^\ast)^r$ is homeomorphic to a
complete metric space we obtain:

\begin{lemma}\label{bair} The set of all $\xi$-transcendental bundles $L\in \V_\xi$ is of Baire category 2. In
particular, the set of $\xi$-transcendental $L\in \V_\xi$ is dense in the variety
$\V_\xi$.
\end{lemma}

\begin{lemma}\label{lomeshane} For any $\xi$-transcendental flat line bundle $L\in \V_\xi$ the dimension of the vector space $H^q(X;L)$ equals the $q$-dimensional
Novikov - Betti number $b_q(\xi)$,
$$\dim H^q(X;L)=b_q(\xi).$$
\end{lemma}
\begin{proof} The monodromy homomorphism ${\rm {Mon}}_L: \Z[H]\to \C$ defines a
left $\Z[H]$-module structure $\C_L$ on $\C$ and by definition
\begin{eqnarray}\label{deftwisted}
H^q(X;L)=H^q(\Hom_{\Z[H]}(C_\ast(\tilde X), \C_L)) \end{eqnarray}
 where
$C_\ast(\tilde X)$ is the cellular chain complex of the covering $\tilde X\to
X$ corresponding to $\ker(\xi)$. If $L$ is transcendental then ${\rm {Mon}}_L$
gives a field extension $Q(H)\to \C$ where $Q(H)$ is the field of fractions of
$\Z[H]$. We obtain therefore (using finiteness of $C_\ast(\tilde X)$ over
$\Z[H]$ and (\ref{deftwisted})):
\begin{eqnarray*}
H^q(X;L) &\simeq & H^q(\Hom_{\Z[H]}(C_\ast(\tilde X);Q(H)))\otimes_{Q(H)}\C_L
\\ &\simeq & H^q(X;Q(H))\otimes_{Q(H)}\C_L. \end{eqnarray*}
This implies that for any transcendental $L$ one has
\begin{eqnarray}\label{nov}
\dim_\C H^q(X;L) \, =\, \dim_{Q(H)}H^q(X;Q(H))
\end{eqnarray}
and the right hand side. Finally we invoke Proposition 1.30 of \cite{farbook} which explains why the right hand side of (\ref{nov}) equals the Novikov - Betti number $b_q(\xi)$.
\end{proof}

\begin{theorem}[Farber - Sch\"utz, \cite{farsch}] \label{perfect}
Let $X$ be a finite cell complex and $\xi\in H^1(X;\R)$. Let $L\in \V_\xi$ be
$\xi$-transcendental. Assume
 that there exist
cohomology classes $v_0\in H^{d_0}(X;L)$ and $v_i\in H^{d_i}(X;\C)$ where $i=1,
\dots, k$ such that $d_i>0$ for $i\in \{1, \dots, k\}$ and the cup-product
\begin{eqnarray}
v_0\cup v_1\cup\dots\cup v_k\, \not=0 \, \in H^\ast(X;L)
\end{eqnarray}
is nontrivial. Then $\cat(X,\xi)>k$.
\end{theorem}

Theorem \ref{perfect} combines simplicity with remarkable efficiency.
Below in this section we test this theorem in many specific
examples.

\subsection{The notion of cup-length $\clz(X,\xi)$}

In view of Theorem \ref{perfect} we introduce the following notation.

Let $X$ be a finite cell complex and $\xi\in H^1(X;\R)$. We denote by
$\clz(X,\xi)$ the maximal integer $k\geq 0$ such that Theorem \ref{perfect}
could be applied to $(X,\xi)$; if Theorem \ref{perfect} is not applicable (i.e.
if $H^\ast(X;L)=0$ for any $\xi$-transcendental $L\in \V_\xi$, compare Lemma
\ref{lomeshane}) we set
$$\clz(X,\xi)=-1.$$
Hence, \begin{eqnarray}\clz(X,\xi)\, \in \, \{-1, 0, 1, \dots\}.
\end{eqnarray}

In other words, $\clz(X,\xi)\geq k$ where $k\geq 0$ iff there exists a
$\xi$-transcendental flat line bundle $L\in \V_\xi$ and there exist cohomology
classes $v_0\in H^{d_0}(X;L)$ and $v_i\in H^{d_i}(X;\C)$ where $i=1, \dots, k$
and $d_i>0$ for $i\in \{1, \dots, k\}$ such that the cup-product
\begin{eqnarray}\label{cupprod}
v_0\cup v_1\cup\dots\cup v_k\, \not=0 \, \in H^\ast(X;L)
\end{eqnarray}
is nontrivial.

Note that for $\xi=0$ the number $\clz(X,\xi)$ coincides with the usual
cup-length $\clz(X)$; recall that the later is defined as the largest integer
$r$ such that there exist cohomology classes $u_i\in H^{d_i}(X;\C)$ where $i=1,
\dots, r$ of positive degree such that their cup-product $ u_1\cup\dots\cup
u_k\, \not=0 \, \in H^\ast(X;L) $ is nontrivial.

One can restate Theorem \ref{perfect} as follows:

\begin{theorem}\label{rephrase} For any finite complex $X$ and $\xi\in H^1(X;\R)$ one has $$\cat(X,\xi)\geq \clz(X,\xi)+1.$$
\end{theorem}

The next useful Lemma suggests several different ways to characterize the
number $\clz(X,\xi)$. This Lemma plays an important role in the sequel.

\begin{lemma}\label{equivalent} Let $X$ be a finite cell complex and $\xi\in H^1(X;\R)$.
The following statements regarding an integer $k\geq 0$ are equivalent:
\begin{enumerate}
\item[{\rm (A)}] $\clz(X,\xi)\geq k$;

\item[{\rm (B)}] There exists cohomology class $v_0\in H^{d_0}(X;L)$ where
$L\in \V_\xi$ is a $\xi$-transcendental flat line bundle and there exist $k$ integral
cohomology classes $v_i\in H^{d_i}(X;\Z)$ where $i=1, \dots, k$ and $d_i>0$ for
$i\in \{1, \dots, k\}$, such that the cup-product (\ref{cupprod}) is
nontrivial.

\item[{\rm (C)}] Let $H$ denote $H_1(X;\Z)/\ker (\xi)$ and let $Q(H)$ denote the field of fractions of the group
ring $\Z[H]$. Then there exist cohomology classes $w_0\in H^{d_0}(X;Q(H))$ and $v_i\in H^{d_i}(X;\Z)$
where $d_i>0$ for $i=1, \dots, k$ such that the cup-product
\begin{eqnarray*}
w_0\cup v_1\cup \dots \cup v_k\not=0\in H^\ast(X;Q(H))
\end{eqnarray*}
is nontrivial.

\item[{\rm (D)}] For any $\xi$-transcendental flat line bundle $L\in \V_\xi$ there
exists cohomology class $v_0\in H^{d_0}(X;L)$ such that the cup-product
(\ref{cupprod}) is nontrivial for some integral cohomology classes $v_i\in
H^{d_i}(X;\Z)$ with $d_i>0$ for $i\in \{1, \dots, k\}$.

\item[{\rm (E)}] For any $\xi$-transcendental flat line bundle $L\in \V_\xi$ there
exist cohomology classes $v_0\in H^{d_0}(X;L)$ and $v_i\in H^{d_i}(X;\C)$ where
$i=1, \dots, k$ and $d_i>0$ for $i\in \{1, \dots, k\}$ such that the
cup-product (\ref{cupprod}) is nontrivial.
\end{enumerate}
\end{lemma}
\begin{proof}
Let us show that (A)$\implies$(B). Fix a $\xi$-transcendental bundle $L\in \V_\xi$,
and $v_0\in H^\ast(X;L)$ such that (\ref{cupprod}) is nontrivial for some
$v_i\in H^{d_i}(X;\C)$ with $d_i>0$. Consider now the cup-products
\begin{eqnarray}\label{function}
v_0\cup v'_1\cup\dots \cup v'_k
\end{eqnarray}
with arbitrary integral cohomology classes $v'_i\in H^{d_i}(X;\Z)$; here the
degrees $d_i$ are assumed to be fixed. (\ref{function}) is a multi-linear
function of the classes $v'_i$. Since the integral classes generate
$H^{d_i}(X;\C)$ over $\C$ we obtain that (\ref{function}) must be nonzero for
some choice of classes $v'_i$, i.e. (B) holds.

Now we show that (B)$\implies$(C). Fix $L\in \V_\xi$ and the classes $v_0\in
H^{d_0}(X;L)$ and $v_i\in H^{d_i}(X;\Z)$ satisfying conditions described in
(B). The monodromy homomorphism ${\rm {Mon}}_L: \Z[H]\to \C$ is an injective
ring homomorphism, it extends to the field of fractions $Q(H)\to \C$. The image
of the induced homomorphism on cohomology
\begin{eqnarray}\label{flat}
\psi: H^{d_0}(X;Q(H)) \to H^{d_0}(X;L)
\end{eqnarray}
generates $H^{d_0}(X;L)$ over $\C$ and (\ref{flat}) is injective (for reasons
mentioned in the proof of Lemma \ref{lomeshane}). Fix cohomology classes
$v_i\in H^\ast(X;\Z)$ where $i=1, \dots, k$. For a cohomology class $w_0\in
H^{d_0}(X;Q(H))$ the function
\begin{eqnarray}\label{com} \psi(w_0)\cup v_1\cup \dots\cup v_k= \psi(w_0\cup
v_1\cup \dots\cup v_k)\, \in \, H^\ast(X;L)
\end{eqnarray} extends to a $\C$-linear function of $v_0\in H^{d_0}(X;L)$
$$v_0\mapsto v_0\cup v_1\cup \dots\cup v_k \in H^\ast(X;L).$$
If this function is nonzero then it may not vanish on the image of $\psi$, i.e.
(C) holds.

The implication (C)$\implies$(D) follows from injectivity of homomorphism
(\ref{flat}) and from (\ref{com}).

Implications (D)$\implies$(E) and (E)$\implies$(A) are obvious. This completes
the proof.
\end{proof}

\begin{lemma}\label{clgeq}
Assume that $X$ and $Y$ are path connected finite cell complexes and $\xi\in
H^1(X\times Y;\R)$.
 Then
\begin{eqnarray}\label{cl}
\clz(X\times Y, \xi)\geq \clz(X,\xi|_X) + \clz(Y, \xi|_Y).
\end{eqnarray}
\end{lemma}
\begin{proof} Denote $\clz(X, \xi|_X)=k$ and $\clz(Y, \xi|_Y)=r$.
Any flat line bundle $L$ over $X\times Y$ has the form $L_1\boxtimes L_2$
(exterior tensor product) where $L_1$ and $L_2$ are flat line bundles over $X$
and $Y$ respectively. Note that if $L$ lies in the variety
$\V_\xi=\Hom(H_1(X\times Y;\Z)/\ker (\xi),\C^\ast)$ then $L_1$ and $L_2$ are
obtained by restrictions and hence $L_1\in \V_{\xi|_X}$ and $L_2\in
\V_{\xi|_Y}$. We will use equivalence between (A) and (E) of Lemma
\ref{equivalent}. Fix a $\xi$-transcendental bundle $L=L_1\boxtimes L_2\in \V_\xi$
over $X\times Y$. Then both $L_1$ and $L_2$ are $\xi$-transcendental. Find classes
$v_0\in H^\ast(X;L_1)$, $v_1, \dots, v_k\in H^\ast(X;\C)$, $u_0\in
H^\ast(Y;L_2)$, $u_1, \dots, u_r\in H^\ast(Y;\C)$ such that $v_0\cup v_1\cup
\dots\cup v_k\not=0$ and $u_0\cup u_1\cup \dots\cup u_r\not=0$. Now we have
cohomology classes $v_0\times u_0\in H^\ast(X\times Y;L)$ and $v_i\times 1,
1\times u_j\in H^\ast(X\times Y;\C)$ and the product
\begin{eqnarray*}
(v_0\times u_0) \, \cup \, \prod_{i=1}^k (v_i\times 1) \, \cup \, \prod_{j=1}^r
(1\times u_j)\not=\, \, 0\, \in\,  H^\ast(X\times Y;L)
\end{eqnarray*}
is nonzero. Here we use the K\"unneth formula which states
\begin{eqnarray}
H^\ast(X\times Y; L_1\boxtimes L_2) \simeq H^\ast(X;L_1) \otimes H^\ast(Y;L_2).
\end{eqnarray}
This proves (\ref{cl}).
\end{proof}

\subsection{Some examples}\label{sec5}

In this section we apply Theorem \ref{perfect} in a few simple examples.

\subsubsection{} First consider the case $\xi=0$. We know that for $\xi=0$ the number
$\cat(X,\xi)$ coincides with the classical LS category $\cat(X)$, see
\cite{farbook}, Example 10.8. Let us examine what Theorem \ref{perfect} gives in
this case. The variety $\V_\xi$ has only one point -- the trivial flat line
bundle $\C$ over $X$. The support $\Supp(X,\xi)=\emptyset$ is always empty for
$\xi=0$. We may therefore take $v_0=1\in H^0(X;\C)$ applying Theorem
\ref{perfect}. Thus, we see that Theorem \ref{perfect} claims in the special
case $\xi=0$ that if there exist cohomology classes $v_1, \dots, v_k\in
H^{>0}(X;\C)$ with $v_1\cup\dots\cup v_k\not=0$ then $\cat(X)>k$. This claim is
the classical cup-length estimate for the LS category $\cat(X)$.

\subsubsection{}
Note that if $\xi\not=0$ and $X$ is connected then $H^0(X;L)=0$ for any
nontrivial $L\in \V_\xi$. Note also that the trivial flat line bundle $\C\in
\V_r$ is never $\xi$-transcendental. Therefore the degree of the class
$v_0$ (which appears in Theorem \ref{perfect}) in the case $\xi\not=0$ must be
positive. Hence for $\xi\not=0$ the number $k$ in Theorem \ref{perfect} satisfies
$k \leq \dim X-1$. This explains why Theorem \ref{perfect} cannot give
$\cat(X,\xi) \geq \dim X+1$ for $\xi\not=0$.

Lemma \ref{lm10} yields
\begin{eqnarray}\label{lesss}
\cat(X,\xi)\leq \cat(X)-1 \leq \dim X
\end{eqnarray}
assuming that $X$ is connected and $\xi\not=0$. This is consistent with the
remark of the previous paragraph.

\subsubsection{}
The following example shows that (\ref{lesss}) can be satisfied as an equality
i.e. that $\cat(X,\xi)=\dim X$ is possible.
Consider the bouquet $X=Y\vee S^1$ where $Y$ is a finite polyhedron, and assume
that the class $\xi\in H^1(X;\R)$ satisfies $\xi|_Y=0$ and $\xi|_{S^1}\not=0.$
We know that in this case
\begin{eqnarray}\label{minus} \cat(X,\xi) =
\cat(Y)-1,\end{eqnarray} see Example 10.11 from \cite{farbook}.

We are going to apply Theorem \ref{perfect}. The variety $\V_\xi$ in this case
coincides with the set $\C^\ast=\C-\{0\}.$ The support $\Supp(X,\xi)$ contains
in this case only the trivial line bundle. $L\in \V_\xi$ is $\xi$-transcendental if
the monodromy along the circle $S^1$ is a $\xi$-transcendental complex number. For
any $L\in \V_\xi$ the restriction $L|_Y$ is trivial and the restriction
homomorphism $H^i(X;L)\to H^i(Y;\C)$ is onto.

Suppose that the cohomological cup-length of $Y$ with $\C$-coefficients equals
$\ell$, i.e. there exist cohomology classes of positive degree $u_0, u_1,
\dots, u_{\ell-1}\in H^{>0}(Y;\C)$ such that the product $u_0\cup \dots\cup
u_{\ell-1}\not=0$ is nonzero. By the above remark, for a nontrivial $L\in
\V_\xi$ we obtain cohomology classes $v_0\in H^\ast(X;L)$ and $v_1, \dots,
v_{\ell-1}\in H^{>0}(X;\C)$ such that $v_i|_Y=u_i$. Hence $v_0\cup v_1\cup
\dots\cup v_{\ell-1} \not=0\in H^\ast(X; L)$. By Theorem \ref{perfect} we
obtain $\cat(X,\xi) > \ell-1$ which is equivalent (taking into account
(\ref{minus})) to $\cat(Y)>\ell$. The last inequality is the classical
cup-length estimate for the usual category.

\subsubsection{} \label{specific}
Let us now consider a very specific example: $X=T^2\vee S^1$. In this case
$H^1(X;\R)=\R^3$ and we describe $\cat(X,\xi)$ as function of $\xi\in
H^1(X;\R)=\R^3$. We denote by $\ell\subset \R^3$ the set of all classes $\xi$
such that $\xi|_{T^2}=0$. Clearly $\ell$ is a line through the origin in
$\R^3$. We claim that:
\begin{eqnarray}\label{cases}
\cat(X,\xi)=\left\{ \begin{array}{ll} 1, & \mbox{if $\xi\in \R^3 -\ell$},\\
2, & \mbox{if $\xi\in \ell-\{0\}$},\\
3, & \mbox{if $\xi=0$}.
\end{array}
\right.
\end{eqnarray}
Indeed, consider first the case $\xi\notin \ell$, i.e. $\xi|_{T^2}\not=0$. Let
us show that $\cat(X,\xi)\leq 1$. Denote $p=T^2\cap S^1$ and let $q\in S^1$ be
a point distinct from $p$. Set $F=X-\{q\}$ and $F_1=S^1-\{p\}$. Then $F\cup
F_1=X$ is an open cover of $X$ with $F_1\to X$ null-homotopic and with $F$
being $N$-movable in $X$ for any $N$ (assuming that $\xi|_{T^2}\not=0$; this
follows from homotopy invariance of $\cat(X,\xi)$ and from Example \ref{ex1}.

Since $\cat(X,\xi)=0$ would imply $\chi(X)=0$ by Theorem \ref{ep} stated below
we obtain that $\cat(X,\xi)>0$ for any $\xi$ (as $\chi(X)=-1\not=0$). This
proves the first line of (\ref{cases}).

If $\xi\in \ell-\{0\}$ we apply (\ref{minus}) to obtain $\cat(X,\xi)=\cat(T^2)-1=2$.

For $\xi=0$ we easily find $\cat(X,\xi)=\cat(X)=3$.

\section{Upper bounds for $\cat(X, \xi)$ and relations with the Bieri - Neumann - Strebel invariants}

Bieri, Neumann and Strebel introduced in \cite{binest} a geometric invariant of
discrete groups $G$ which captures information about the finite generation of
kernels of abelian quotients of $G$. In this section we describe a relation
between this invariant and properties of $\cat(X,\xi)$.

Let us recall the definition. We always assume that $G$ is finitely presented
as this is sufficient for our purposes. Let $S(G)$ denote
$(\Hom(G,\R)-\{0\})/\R_+$ where $\R_+$ acts on $\Hom(G,\R)$ by multiplication.
Clearly $S(G)$ is a sphere of dimension $n-1$ where $n$ is the rank of the
abelianization of $G$. Denote by $[\chi]$ the equivalence class of a nonzero
homomorphism $\chi:G\to \R$. The Bieri-Neumann-Strebel invariant associates to
$G$ a subset $\Sigma(G)\subset S(G)$ defined as follows\footnote{We rely on
Theorem 5.1 of \cite{binest} which states that the definition of $\Sigma(G)$ given
above coincides in the case of finitely presented $G$ with the main definition
of \cite{binest}.}. Let $X$ be a finite cell complex with $\pi_1(X)=G$ and let $p:
\tilde X \to X$ be the universal abelian cover of $X$. A homomorphism $\chi\in
\Hom(G,\R)$ can be viewed as a cohomology class lying in $H^1(X;\R)$. One has
$\chi\in \Sigma(G)$ if and only if the inclusion $N\to \tilde X$ induces an
epimorphism $\pi_1(N, x_0)\to \pi_1(\tilde X, x_0)$ where $N\subset\tilde X$ is
a connected neighborhood of infinity with respect to $\chi$, see \S
2 of \cite{farsch} and Lemma 5.2 from \cite{binest}.

The following Theorem summarizes the results proven in \cite{farsch}:

\begin{theorem}[Farber - Sch\"utz]\mbox{ }  \\[-0.6cm]
\begin{enumerate}
\item[(a)] Let $X$ be a finite connected polyhedron and let $\xi\in H^1(X;\R)$ be nonzero. Then
$$\cat(X, \xi) \le \dim X.$$

\item[(b)] Let $M$ be a closed connected smooth manifold and let $\xi\in H^1(M;\R)$ be nonzero. Then
$$\cat(M, \xi) \le \dim (M) -1.$$

\item [(c)] Let $M$ be a closed connected smooth manifold of dimension $\ge 5$ and let $\xi \in H^1(M;\R)$ be nonzero.
If either $$\xi\in \Sigma(\pi_1(M))\quad \mbox{or}\quad -\xi\in \Sigma(\pi_1(M))$$
then
$$\cat(M,\xi) \le \dim M - 2.$$

\item [(d)] Let $M$ be a closed connected smooth manifold of dimension $\ge 5$ and let $\xi \in H^1(M;\R)$ be nonzero.
If both inclusions take place $$\xi\in \Sigma(\pi_1(M))\quad \mbox{and}\quad -\xi\in \Sigma(\pi_1(M))$$ then
$$\cat(M,\xi) \le \dim M - 3.$$
\end{enumerate}
\end{theorem}

\section{Homological category weights, estimates for $\cat^1(X, \xi)$ and calculation of $\cat(X, \xi)$, $\cat^1(X,\xi)$ for products of surfaces}

In this section we describe briefly the results obtained in \cite{farsce} which has as its main goal obtaining cohomological lower estimates for
$\cat^1(X, \xi)$. In particular we will see that there are many examples when $\cat(X, \xi) < \cat^1(X, \xi)$ and the difference $\cat^1(X, \xi) - \cat(X, \xi)$ can be arbitrarily large.

The first step is the introduction of a new notion of category weight of
homology classes which is somewhat dual to the cohomological
notion introduced by E. Fadell and S. Husseini \cite{FH}. It was a pleasant surprise to discover that the
homological category weight is homotopy
invariant unlike the cohomological version of Fadell and Husseini.

The notion of homological category weight allows us to obtain improved
cohomological lower bounds for $\cat^1(X;\xi)$.

\subsection{Homological category weight}\label{basdef}

The classical cohomological lower bound for the Lusternik -
Schnirelmann category $\cat(X)$ states that $\cat(X)>n$  if there
exist $n$ cohomology classes of positive degree $u_j\in
H^\ast(X;R_j)$ where $j=1, 2, \dots, n$, such that their
cup-product $u_1 u_2 \dots u_n\not=0 \in H^\ast(X; R)$ is
nontrivial. Here $R_j$ denotes a local coefficient system on $X$
and $R$ is the tensor product $R=R_1\otimes \dots \otimes R_n$.

E. Fadell and S. Husseini \cite{FH} improved this estimate by introducing the
notion of a {\it category weight} $\cwgt(u)$ of a cohomology class $u\in
H^q(X;R)$.

\begin{definition}[Fadell - Husseini, \cite{FH}]\label{def1}
Let $u\in H^q(X; R)$ be a nonzero cohomology class where $R$ is a
local coefficient system on $X$. One says that $\cwgt(u)\geq k$
(where $k\geq 0$ is an integer) if for any closed subset $A\subset
X$ with $\cat_XA\leq k$ one has $u|_A=0\in H^q(A;R)$.
\end{definition}

Recall that the inequality $\cat_XA\leq k$ means that $A$ can be
covered by $k$ open subsets $U_i\subset X$ such that each
inclusion $U_i\subset X$ is null-homotopic, $i=1, \dots, k$.

According to Definition \ref{def1} one has $\cwgt(u) \geq 0$ in
general and $\cwgt(u)\geq 1$ for any nonzero cohomology class of
positive degree. As Fadell and Husseini \cite{FH} showed,
$\cwgt(u)>1$ in some special situations which improves
the lower estimate for $\cat(X)$. Indeed, one has
$$\cat(X) \geq 1
+\sum_{i=1}^n \cwgt(u_i)$$ assuming that the cup-product $u_1u_2\dots
u_n\not=0$ is nonzero.

Y. Rudyak \cite{Ru} and J. Strom \cite{Str} studied a modification
of $\cwgt(u)$, called the strict category weight ${\rm
{swgt}}(u)$. The latter has the advantage of being homotopy invariant.
However in some examples the strict category weight is
considerably smaller than the original category weight of Fadell
and Husseini.

In \cite{farsce} we introduced and exploited a \lq\lq dual\rq\rq\,
notion of category weight of homology classes. It has the
geometric simplicity and clarity of category weight as defined by
Fadell and Husseini and has a surprising advantage of being
homotopy invariant.

\begin{definition}[Farber - Sch\"utz, \cite{farsce}]
Let $z\in H_q(X;R)$ be a singular homology class with coefficients in a local
system $R$ and let $k\geq 0$ be a nonnegative integer. We say that
$\cwgt(z)\geq k$ if for any closed subset $A\subset X$ with $\cat_XA\leq k$
there exists a singular cycle $c$ in $X-A$ representing $z$. We say that
$\cwgt(z)=k$ iff $\cwgt(z)\geq k$ and $\cwgt(z)\not\geq k+1$.
\end{definition}

In other words, ${\cwgt}(z)\geq k$ is equivalent to the fact that $z$ can be
realized by a singular cycle avoiding any closed subset $A\subset X$
with $\cat_XA\leq k$.

For example, $\cwgt(z)\geq 1$ iff $z$ can be realized by a singular cycle
avoiding any closed subset $A\subset X$ such that the inclusion $A\to X$ is
homotopic to a constant map.

It will be convenient to define the category weight of the zero homology class
as $+\infty$.

Formally $\cwgt(z)\geq k$ if $z$ lies in the intersection
$$\bigcap_A \im[H_q(X-A;R)\to H_q(X;R)]$$
where $A\subset X$ runs over all closed subsets with $\cat_XA\leq k$.

The relation $\cwgt(z) \leq k$ means that there exists a closed subset
$A\subset X$ with $\cat_XA \leq k+1$ such that any geometric realization of $z$
intersects $A$. In particular we obtain the following inequality
\begin{eqnarray}\label{one1}
\cat(X)\geq \cwgt(z) +1
\end{eqnarray}
for any nonzero homology class $z\in H_q(X;R)$, $z\not=0$. The last inequality
can be also rewritten as
\begin{eqnarray}
0\leq \cwgt(z) \leq \cat(X)-1\leq \dim X\end{eqnarray} for any nonzero homology class.

Note that if $X$ is path-connected and $z$ is zero-dimensional, i.e. $z\in
H_0(X)$, then $\cwgt(z) = \cat(X)-1$.

\begin{lemma} Assume that $X$ is a simplicial polyhedron. Then $\cwgt(z)\geq k$
iff $z$ can be realized in $X-A$ for any sub-polyhedron $A\subset X$ with
$\cat_XA\leq k$.
\end{lemma}
\begin{proof}
We only need to show the 'if'-direction. Let $A\subset X$ be
closed with $\cat_X A \leq k$. We need to show that $z$ can be
realized by a cycle in $X-A$. We have $A\subset U_1\cup\dots\cup
U_k$ with each $U_i$ open and null-homotopic in $X$. Passing to a
fine subdivision of $X$, we can find a sub-polyhedron $B\subset X$
with $A\subset B \subset U_1\cup \dots\cup U_k$. Then $\cat_XB\leq
k$ and $z$ can be realized by a cycle lying in $X-B\subset X-A$.
\end{proof}

\begin{example} Assume that $X$ is a closed 2-dimensional manifold, i.e. a
compact surface. Let us show that any nonzero homology class $z\in
H_1(X)$ has $\cwgt(z)\geq 1$. Indeed, it is easy to see that any
closed subset $A\subset X$ which is null-homotopic in $X$ lies in
the interior of a disk $D^2\subset X$; but $H_1(X- \Int D^2) \to
H_1(X)$ is an isomorphism.
\end{example}

\begin{theorem}[Farber - Sch\"utz, \cite{farsce}] If $f: X\to Y$ is a homotopy equivalence then for any homology
class $z\in H_q(X;R)$ one has
\begin{eqnarray}
\cwgt(z)=\cwgt(f_\ast(z)).
\end{eqnarray}
Here $f_\ast(z)\in H_q(X;R')$ where $R'=g^\ast R$ is the local coefficient
system over $Y$ induced by the homotopy inverse $g: Y\to X$ of $f$.
\end{theorem}

Another important result from \cite{farsce} shows how one may use the notion of homological category weight in conjunction
with the notion of Fadell and Husseini to estimate $\cat(X)$:

\begin{corollary} Suppose that $X$ is a metric space and for some classes
$z\in H_q(X;R)$ and $u\in H^q(X;R')$ the evaluation
$$\langle u, z\rangle \, \not=0\, \in\,  R'\otimes R$$ is nonzero.
Then
\begin{eqnarray}\label{improved}
\cat(X) \geq \cwgt(z) +\cwgt(u) +1.
\end{eqnarray}
Here $\cwgt(z)$ is the category weight of homology class $z$ as defined above and $\cwgt(u)$ is the category weight of $u$ as defined by Fadell
and Husseini \cite{FH}.
\end{corollary}

The main idea standing behind this Corollary is that {\it \lq\lq the degree of nontriviality\rq\rq of a cohomology class $u$ can be \lq\lq measured\rq\rq\,  by \lq\lq quality\rq\rq\, of the homology class
$z$ satisfying $\langle u, z\rangle \not=0$.}

It is curious to observe that
 in the case of closed manifolds our notion of
category weight of homology classes coincides with the cohomological notion of Fadell and Husseini
\cite{FH} via the Poincar\`e duality. However for Poincar\'e complexes these notions are distinct as we
show by an example given below.

\begin{theorem}[Farber - Sch\"utz, \cite{farsce}]\label{manif}
Suppose that $X$ is a closed $n$-dimen\-sional manifold, $z\in H_q(X;R)$ where
$R$ is a local coefficient system. Let $u\in H^{n-q}(X; R\otimes \tilde \Z)$ be
the Poincar\'e dual cohomology class, i.e. $z=u\cap [X]$, see below. Then
\begin{eqnarray} \cwgt(z)
=\cwgt(u).\end{eqnarray} Here $\tilde \Z$ denotes the orientation local system
on $X$, i.e. for a point $x\in X$ the stalk of $\tilde \Z$ at $x$ is $\tilde
\Z_x=H_n(X, X-x;\Z)$, see \cite{Sp}.
\end{theorem}

\begin{example}
Let $X=\RP^n$ be the real projective space. For the unique nonzero
cohomology class $z\in H_q(X;\Z_2)$ one has $\cwgt(z) = n-q.$
Indeed, the dual homology class is $\alpha^{n-q}\in
H^{n-q}(X;\Z_2)$ where $\alpha\in H^1(X;\Z_2)$ is the generator.
Clearly, $\cwgt(\alpha^{n-q}) = n-q$.
\end{example}

Theorem \ref{manif} implies:

\begin{corollary}
If $X$ is a closed $n$-dimensional manifold then for any homology class $z\in
H_q(X;R)$ with $q<n$ one has \begin{eqnarray} \cwgt(z)\geq 1.\end{eqnarray}
\end{corollary}

Indeed, if $q<n$ then the dual cohomology class $u$ has positive degree and
hence $\cwgt(u)\geq 1$.

Consider now the case when $X$ is $n$-dimensional Poincar\'e
complex. The following example shows that Theorem
\ref{manif} is false for Poincar\'e complexes. It is a
modification of an argument due to D. Puppe showing that the
notion of category weight of cohomology classes is not homotopy
invariant.

\begin{example} Consider the lens space $L=S^{2n+1}/(\Z/p)$ where $p$ is
an odd prime and $\Z/p$ acts freely on $S^{2n+1}$. Denote by $r: S^{2n+1}\to L$
the quotient map. Let $X$ be the mapping cylinder of $r$, i.e.
$$X= L \sqcup S^{2n+1}\times [0,1]/\sim$$
where each point $(x, 0)\in S^{2n+1}\times [0,1]$ is identified with $r(x)\in
L$. Clearly $X$ is homotopy equivalent to $L$ and so it is a Poincar\'e
complex. By a theorem of Krasnoselski \cite{Kr}, the category of $X$ equals
$2n+2$. Hence for $z=1\in H_0(X;\Z_2)$ one has
$$\cwgt(z)=\cat(X)-1=2n+1,$$ see
above. The dual cohomology class $u$ is the generator $u\in
H^{2n+1}(X;\Z_2)$. Let us show that
$$\cwgt(u)=1.$$ Indeed, consider the sphere
$S=S^{2n+1}\times 1\subset X$. The restriction $u|_S\in H^{2n+1}(S;\Z_2)$
coincides with the induced class $r^\ast(v)$ where $v\in H^{2n+1}(L;\Z_2)$ is
the generator. Hence the cohomology class $u|_S$ is nonzero. However, the
sphere $S$ has category 2 and moreover $\cat_XS=2$ (as the inclusion $S\to X$
is not null-homotopic).
\end{example}

\subsection{Cohomological estimate for $\cat^1(X, \xi)$}
The following Theorem is the main result established in \cite{farsce}.

\begin{theorem}[Farber - Sch\"utz, \cite{farsce}]\label{main2}
Let $X$ be a finite cell complex and $\xi\in H^1(X;\R)$. Let $L\in
\V_\xi$ be a complex flat line bundle over $X$ which is not a $\xi$-algebraic integer (see Definition \ref{transcendental}).
Suppose that for an integral homology
class $z\in H_q(\tilde X;\Z)=H_q(X;\Lambda)$ and some cohomology
classes $u\in H^{d}(X;L)$ and $u_i\in H^{d_i}(X;\C)$, where
$d_i>0$ for $i=1, \dots, k$, the evaluation $\langle u\cup
u_1\cup\dots\cup u_k, p_\ast(z)\rangle \not=0\in \C$ is nonzero.
Here $p_\ast(z)\in H_d(X;L^\ast)$, $q=d+d_1+\dots+d_k$. Then one
has
\begin{eqnarray}
\cat^1(X,\xi)\geq \cwgt(z)+k +1. \end{eqnarray}
\end{theorem}

Here $\cwgt(z)$ denotes the category weight of $z$ viewed as a homology class
of $X$ with local coefficient system $\Lambda=\Z[H]$ where $H=H_1(X;\Z)/\ker (\xi)$ is a free abelian group which can be identified with the group of periods of $\xi$.

\subsection{$\cat(X, \xi)$ and $\cat^1(X, \xi)$ for products of surfaces}
\begin{theorem}[Farber - Sch\"utz \cite{farsch, farsce}] \label{surface}
Let $M^{2k}$ denote the product $$\Sigma_1\times \Sigma_2\times \dots\times
\Sigma_k$$ where each $\Sigma_i$ is a closed orientable surface of genus
$g_i>1$. Given a cohomology class $\xi\in H^1(M^{2k};\R)$, one has
\begin{eqnarray}
\cat(M^{2k}, \xi) = 1+2r
\end{eqnarray}
and
\begin{eqnarray}
\cat^1(M^{2k},\xi) = 1+k+ r \end{eqnarray} where $r=r(\xi)$
denotes the number of indices $i\in \{1, 2, \dots, k\}$ such that the
cohomology class $\xi|_{\Sigma_i}\in H^1(\Sigma_i,\R)$ vanishes. In particular
\begin{eqnarray}
\cat(M^{2k}, \xi)=1
\end{eqnarray}
and
\begin{eqnarray}\label{one2}
\cat^1(M^{2k},\xi) = 1+k
\end{eqnarray}
assuming that $\xi|_{\Sigma_i}\not=0\in H^1(\Sigma_i;\R)$ for any $i=1, \dots,
k$.
\end{theorem}
 We refer the reader to our papers \cite{farsch, farsce} for proofs and more detail.

 Theorem \ref{surface} demonstrates that the difference between two notions of category with respect to a cohomology class
 $$\cat^1(X, \xi)- \cat(X,\xi)$$ can be arbitrarily large.

Some further results developing the Lusternik-Schnirelman theory for closed
1-forms were obtained by J. Latschev \cite{Lat} and D. Sch\"utz \cite{schman};
see also \cite{farbe4} and \cite{farkap}.

\end{document}